\documentclass{amsart}
\usepackage{tikz-cd}
\usepackage[utf8]{inputenc}
\usepackage[margin=1in]{geometry}
\usepackage{mine,amsmath,amsthm,bbm,xfrac}

\DeclareMathOperator{\Aut}{Aut}
\newcommand{\integral}{\int}
\newcommand{\XP}{\xphi}
\DeclareMathOperator{\Ind}{Ind}
\DeclareMathOperator{\boldgsp}{\mathbf{GSp}}
\DeclareMathOperator{\Hodge}{Hdg}
\DeclareMathOperator{\diagonal}{diag}
\DeclareMathOperator{\pigs}{PGSp}
\DeclareMathOperator{\SP}{Sp}
\DeclareMathOperator{\GSO}{GSO}
\DeclareMathOperator{\PGL}{PGL}
\DeclareMathOperator{\boldgl}{\mathbf{GL}}
\DeclareMathOperator{\SO}{SO}
\renewcommand{\check}{\vee}
\DeclareMathOperator{\vol}{Vol}
\DeclareMathOperator{\debt}{det}
\newcommand{\fee}{\phi}
\DeclareMathOperator{\GL}{GL}
\DeclareMathOperator{\SL}{SL}
\DeclareMathOperator{\PGSO}{PGSO}

\DeclareMathOperator{\Sym}{Sym}
\DeclareMathOperator{\cl}{cl}
 \DeclareMathOperator{\CL}{cl}
\DeclareMathOperator{\Res}{Res}
\newcommand{\Sye}{\psi}
\DeclareMathOperator{\Val}{Val}
\renewcommand*\d{\mathop{}\!\mathrm{d}}
\DeclareMathOperator{\GSp}{GSp}
\renewcommand{\et}{\rm{\acute{e}t}}

\theoremstyle{plain}
\DeclareMathOperator{\Spin}{Spin}
\newtheorem*{claim}{Claim}
\newenvironment{claimproof}[1]{\par\noindent\emph{Proof of claim:}\space#1}{}
\newtheorem{theo}{Theorem}

\theoremstyle{definition}
\newtheorem*{rmk*}{Remark}
\title{Tate classes and endoscopy for $\operatorname{GSp}_4$ over totally real fields}
\author{Naomi Sweeting}
\email{naomisweeting@math.harvard.edu}
\address{Department of Mathematics, Harvard University}

\date{April 2022}

\begin{document}
\maketitle
\begin{abstract}
    The theory of endoscopy predicts the existence of large families of Tate classes on certain products of Shimura varieties, and it is natural to ask in what cases one can construct algebraic cycles giving rise to these Tate classes. This paper takes up the case of Tate classes arising from the Yoshida lift: these are Tate cycles in middle degree on the Shimura variety for the group $\Res_{F/\Q} (\operatorname{GL}_2 \times \operatorname{GSp}_4)$, where $F$ is a totally real field. A special case is the family of Tate classes which reflect the appearance of two-dimensional Galois representations in the middle cohomology of both a modular curve and a Siegel modular threefold. We show that a natural algebraic cycle generates exactly the Tate classes which are associated to \emph{generic} members of the endoscopic $L$-packets on $\operatorname{GSp}_{4,F}$. In the non-generic case, we give an alternate construction, which shows that the predicted Tate classes arise from Hodge cycles.
\end{abstract}
\tableofcontents

\section{Introduction}

Let $F $ be a totally real number field of degree $d$, and let $G =\GSp_{4, F} $. The unique elliptic endoscopic group for $G $ is $M=(\GL_2\times \GL_2/\mathbb G_m)_F$, where  $\mathbb G_m $ is embedded anti-diagonally and the $L$-embedding is induced by\begin {equation}\widehat M =\GL_2 (\C)\times_{\C ^\times}\GL_2 (\C)\hookrightarrow\GSp_4 (\C) =\widehat G.\end {equation} 
The functorial transfer of cuspidal automorphic forms from $M $ to $G $ has been studied by Roberts \cite{roberts2001global} and Weissauer \cite{weissauer2009endoscopy}. For any (unordered) pair of distinct cuspidal automorphic representations $\pi_1,\pi_2 $ of $\GL_2 (\A_F) $ with the same central character, one obtains an $L$-packet $\Pi(\pi_1,\pi_2)$ of cuspidal automorphic representations of $\GSp_4(\A_F)$. The members $\Pi_S(\pi_1,\pi_2)$ of this $L$-packet are indexed by finite sets $S $ of places of $F$ at which both $\pi_i$ are discrete series, such that $| S | $ is even. The unique generic member of the $L $-packet $\Pi (\pi_1,\pi_2)$ is $\Pi_\emptyset (\pi_1,\pi_2)$.

Let $\boldgsp_4 =\Res_{F/\Q} G $
be the restriction of scalars, with the Shimura datum induced from that of $\GSp_{4,\Q}$. Let $$S(\boldgsp_4)\coloneqq \varprojlim_K S_K(\boldgsp_4) $$
be the resulting pro-algebraic Shimura variety over $\Q$, where $K $ ranges over compact open subgroups of $\boldgsp_4 (\mathbb A_{f})$. (For the rest of the introduction, the same notation will apply when $\GSp_4$ is replaced by any $\Q$-group $H$ with a Shimura datum.) If $\pi_1 $ and $\pi_2$ are unitary and of parallel weights 4 and 2, respectively, then the automorphic representations $\Pi_S (\pi_1,\pi_2)$ contribute to the cohomology of $S$, and Kottwitz's conjecture \cite{kottwitz1990shimura} predicts the associated Galois representations in \'etale cohomology. 

To state the expected formula, choose  a  finite extension $E/\Q $ over which $\pi_1,$ $\pi_2,$ and $\Pi_S(\pi_1,\pi_2)$ are all defined. For  a finite place $\lambda $ of $E$, set $\rho_\Pi= \rho_{\pi_1} \oplus \rho_{\pi_2}(-1)$, where $\rho_{\pi_i}$ are the usual $\lambda$-adic Galois representations associated to  Hilbert modular forms, cf. \cite{blasius1989galois}. Then let $(\widetilde\rho_\Pi, V)$ and $(\widetilde \rho_2, V_0)$ be the tensor induction of $\rho_\Pi$ and $\rho_{\pi_2}$, respectively, from $\Gal(\overline \Q/F)$ to $\Gal(\overline \Q/\Q)$. There is a natural Galois-equivariant inclusion $V_0(-d)\hookrightarrow V$.
Consider the involution $s\in \End(\widetilde \rho_\Pi, V) $ such that, in each factor of the decomposition  \begin{equation}
    V = \otimes_{v|\infty} \rho_\Pi,\end{equation}
    $s$ acts as $-1$ on $\rho_{\pi_1}$ and $1$ on  $\rho_{\pi_2}(-1)$. Taking $s$-eigenspaces induces a decomposition \begin{equation}V= V^+\oplus V^-,\end{equation}
and $V_0(-d)$ lies inside $V ^ + $.

 Kottwitz's conjectures imply that the $\Pi_{S}(\pi_1,\pi_2)_f$-isotypic part of \'etale cohomology should be:
\begin {equation}\label{eq:gsp etale}
H _{\et,c}^{3d} (S(\boldgsp_4)_{\overline \Q}, E_\lambda)_{\Pi_{S}(\pi_1,\pi_2)_f}=\begin{cases} V ^ +, & | S _f|\text { even},\\V ^ -, & | S_f |\text { odd},\end{cases}\end{equation} 
where $S_f\subset S$ is the subset of finite places.
This expectation has been fully verified in the case $F = \Q$ \cite{weissauer2009endoscopy}.
Meanwhile, $V_0$ is also expected to appear in the \'etale cohomology of the  Shimura variety $S(\boldgl_2)$: 
\begin{equation}\label{eq:gl2 etale}
     H^{d}_{\et} (S(\boldgl_2)_{\overline{\Q}}, E_\lambda)_{\pi_{2,f}} = V_0.
\end{equation}
The Tate conjecture then suggests that, for all $S$ with $|S_f|$ even, there should be a $2d$-dimensional algebraic cycle on $S(\boldgsp_4)\times S(\boldgl_2)$ whose \'etale realization induces a nontrivial map \begin{equation}\label{eq: expected intro}H^{3d}_{\et,c}\left(S(\boldgsp_4)_{\overline \Q}, E_\lambda(d)\right)\big[\Pi_{S}(\pi_1,\pi_2)_f\big] \to H^d_{\et}\left(S(\boldgl_2)_{\overline \Q},E_\lambda\right)[\pi_{2,f}].\end{equation}

The natural candidate for this algebraic cycle is the sub-Shimura variety $ S(\boldsymbol H) \subset S(\boldgsp_4) \times S(\boldgl_2)$, where:
\begin{equation}\label{eq: candidate cycle}
    H\coloneqq \GL_2 \times_{\mathbb G_m}\GL_2 \xhookrightarrow{\iota,p} \GSp_4 \times \GL_2.
\end{equation}
Here $\iota: H \hookrightarrow\GSp_4$ is the standard inclusion and $p: H \to \GL_2$ is the first projection. 

Our first main theorem characterizes when the correspondence $S(\boldsymbol H)$  does indeed induce a nontrivial map (\ref{eq: expected intro}).
\begin{theo}[Theorem \ref{biggie thm}]\label{thm: intro global mean thm}
Let $\pi_1$ and $\pi_2$ be cuspidal automorphic representations of $\GL_2(\A_F)$ of parallel weights 4 and 2, respectively, and with the same unitary central character. Then  the composite map \begin{equation*}
\begin{split} H^{3d}_c(S(\boldgsp_4),E(d))[\Pi_{S}(\pi_1,\pi_2)_f] \xrightarrow{[S(\boldsymbol H)]}  H^d(S(\boldgl_2),E)\twoheadrightarrow  H^d(S(\boldgl_2),E)[{\pi_{2,f}}]\end{split}\end{equation*}
is nontrivial if and only if   $S_f= \emptyset$. In this case its image generates the $\GL_2(\A_{F,f})$-module $H^d(S(\boldgl_2),E)[{\pi_{2,f}}]$.
\end{theo}
One could instead project to the $\pi_f$-isotypic component for any cuspidal automorphic representation $\pi$ of $\GL_2(\A_F)$; however, if $\pi \neq \pi_2$, we prove that the resulting map is always trivial.
In the non-generic case $S_f\neq\emptyset $, we are not able to produce an algebraic cycle which induces a nontrivial map (\ref{eq: expected intro}). However, we are able to give an alternative realization of (\ref{eq: expected intro}) as the map induced by a nontrivial Hodge cycle:
\begin{theo}[Theorem \ref{main theorem  Hodge classes}]\label{Intro theorem Hodge classes}
Let $\pi_1 $ and $\pi_2 $
be cuspidal automorphic representations of $\GL_2 (\A_F) $
of parallel weights 4 and 2, respectively, with the same unitary central character. Let $S$ be a set
of places of $F $
at which both $\pi_i $
are discrete series, such that $|S_f|\geq 2$ is even. Then there exists a Hodge class $$\xi\in H ^ {4d} (S (\boldgsp_4)\times S (\boldgl_2), E (2d)) $$
such that:
\begin{enumerate}
    \item For all finite places $\lambda $ of $E $, the image of $\xi$
    in $\lambda $-adic \'etale cohomology is $\Gal (\overline\Q/F ^ c) $-invariant.
    \item The composite map \begin{equation*} H^{3d}_c(S(\boldgsp_4),E(d))[\Pi_{S}(\pi_1,\pi_2)_f] \xrightarrow{\xi_\ast}  H^d(S(\boldgl_2),E)\twoheadrightarrow H^d(S(\boldgl_2),E)[{\pi_{2,f}}] \end{equation*}
is  nontrivial, and its image generates the $\GL_2(\A_{F,f})$-module $H^d(S(\boldgl_2),E)[{\pi_{2,f}}]$.
\end{enumerate}
\end{theo}
\begin {rmk*}
\begin{enumerate}
\item
Assuming Kottwitz's  conjectures, one could show that $\xi $
is $\Gal (\overline\Q/\Q) $-invariant. However, our result is unconditional, and in particular does not rely on (\ref{eq:gsp etale}) and (\ref{eq:gl2 etale}).
\item 
In the text, we also prove higher-weight analogues of Theorems \ref{thm: intro global mean thm} and \ref{Intro theorem Hodge classes}; the results are described later in the introduction.
\end{enumerate}
\end{rmk*}
\subsection*{Overview of the proofs}
Both Theorem \ref{thm: intro global mean thm} and Theorem \ref{Intro theorem Hodge classes}
rely on the explicit realization of $\Pi_S (\pi_1,\pi_2) $
as a theta lift from a four-dimensional orthogonal group, cf. \cite{roberts2001global,weissauer2009endoscopy}.
Indeed, if $| S | $
is even, then there is a quaternion algebra $B $
over $F $
ramified exactly at the places in $S $, and the orthogonal group $\GSO (B)\simeq B ^\times\times B ^\times/\mathbb G_m $
is an inner form of $M $. The automorphic representation $\Pi_S (\pi_1,\pi_2) $ is the theta lift of $\pi_1 ^ B\boxtimes\pi_2 ^ B $
from $\GSO (B) $
to $\GSp_{4, F} $, where $\pi^ B $
is the Jacquet-Langlands transfer of $\pi_i $
to $B ^\times $.
This is crucial because it allows for the calculation of period integrals involving  $\Pi_S(\pi_1,\pi_2)$. 
\subsubsection*{Proof of Theorem \ref{thm: intro global mean thm}}
  Since the non-vanishing of $[S (\boldsymbol H)] $
  may be detected in $L^2$ cohomology, the theorem is essentially a statement about periods of $\Pi_S (\pi_1,\pi_2)\boxtimes\pi_2 ^\check $
  along the subgroup $H\subset\GSp_4\times\GL_2. $
  That is, we must compute integrals of the form
  \begin{equation}\label{equation: intro period definition}
     \mathcal P_S (\gamma,\beta)\coloneqq\integral_{Z_H (\A_F)H(F)\backslash H (\A_F)}\gamma (\iota(h))\beta (p (h))\d h,\;\;\gamma\in\Pi_S (\pi_1,\pi_2),\;\beta\in\pi.
  \end{equation}
  Because $\Pi_S (\pi_1,\pi_2) $
  is a theta lift from $\GSO (B)$, we can compute (\ref{equation: intro period definition}) using the seesaw diagram:
  \begin{center}
    \begin{tikzcd}
    \GSp_4 \arrow[d,dash]\arrow[dr,dash] & \GSO(B)\times_{\mathbb G_m} \GSO(B) \arrow[d,dash]\arrow[dl,dash] \\ \GL_2\times_{\mathbb G_m}\GL_2 & \GSO(B)
    \end{tikzcd}
\end{center}
Here the vertical lines are inclusions and the diagonals are dual reductive pairs inside $\GSp_{1 6} $.  Formally, the seesaw identity would read:
\begin {equation}\label{Intro formal D saw}
\mathcal P_S (\theta (\alpha),\beta) = \integral_{[\operatorname{PGSO}(B)]}\theta(\beta) (g)\theta (\mathbbm 1) (g)\alpha(g)\d g,\;\;\alpha\in \pi_1\otimes \pi_2,\;\;\beta\in \pi_2^\check,\end {equation}
where the theta lifts on the right are from $\GL_2 $
to $\GSO (B)$, and the  theta lifts on both sides depend on choices of Schwartz functions which must be made compatibly.  The integral defining $\theta (\mathbbm 1) $ is divergent,
so a regularization step is necessary to interpret (\ref{Intro formal D saw}). However, after regularization, $\theta(\mathbbm 1)$ can be recognized as 0 if $B $
is not split (i.e. if $S \neq\emptyset $), and as a certain Eisenstein series on $\GSO (B) $ if $B$ is split. The integral (\ref{Intro formal D saw}) then unfolds to an Euler product
which allows us to evaluate it explicitly. The result of the calculation is:
\begin{theo}[Theorems \ref{thm: non-generic vanishing}, \ref{theorem: new global. Pairing}]\label{thm:intro periods}
Let $\pi_1, $ $\pi_2, $ and $\pi$ be cuspidal automorphic representations of $\GL_2 (\A_F) $ such that $\pi_i $  and $\pi^\vee$ have the same central character, and let $S$ be a finite set of (possibly archimedean) places of $F$ at which both $\pi_i $ are discrete series, such that $| S | $ is even. Consider the period pairing \begin{equation}
    \mathcal P_S (\gamma,\beta)\coloneqq\integral_{Z_H (\A_F)H(F)\backslash H (\A_F)}\gamma (\iota(h))\beta (p (h))\d h,\;\;\gamma\in\Pi_S (\pi_1,\pi_2),\;\beta\in\pi,
\end{equation}
where $\d h$ is normalized as in (\ref{subsubsec:measure on H}).
\begin{enumerate}
\item If $\mathcal P_S (\gamma,\beta)\neq 0, $ then $S =\emptyset$, i.e. $\Pi_S (\pi_1, \pi_2)$ is generic, and $\pi\cong \pi_2^\vee$.
\item Suppose given factorizable Schwartz functions $$\xphi_i =\otimes_v\xphi_{i, v}\in\mathcal S (M_2(\A_F)),\;\; i = 1, 2 $$
and factorizable vectors $$\alpha =\otimes_v\alpha_v\in\pi_1\otimes\pi_2,\;\;\beta=\otimes_v\beta_v\in\pi_2^\vee.$$ Then the theta lift $\theta_{\xphi_1\otimes\xphi_2}(\alpha)$ lies in $\Pi_\emptyset (\pi_1,\pi_2) $ and, for a sufficiently large finite set $S$ of places of $F $, $$\mathcal P_\emptyset (\theta_{\xphi_1\otimes\xphi_2} (\alpha),\beta) = 2 | D_F | ^ {1/2}\cdot\pi ^ {- d}\frac {L ^ S(1,\pi_1\times\pi_2 ^\vee) L ^ S (1,\Ad \pi_2)} {\zeta_F ^ S (2)}\prod_{v\in S}\frac {\mathcal Z_v (\xphi_{1,v},\xphi_{2, v},\alpha_v,\beta_v)} {1 - q_v ^ {-1}}. $$
Here $\mathcal Z_v (\xphi_{1, v},\xphi_{2, v},\alpha_v,\beta_v)$ is an explicit local zeta integral which is nonzero for appropriate choices of test data; $\xphi_1\otimes\xphi_2$ is the tensor product Schwartz function in $\mathcal S(M_2(\A_F)^2)$; the theta lift $\theta_{\xphi_1\otimes \xphi_2}(\alpha)$ is defined in  \S\ref{section: Similitude theta lifting}; and the other notations are introduced in (\ref{places etc}).
\end{enumerate}
\end{theo}
\begin{rmk*}
The $L $-values appearing in Theorem \ref{thm:intro periods} are nonzero by the classical  result of Shahidi \cite{shahidi1981amer}.
\end{rmk*}
In fact, Theorem \ref{thm:intro periods} amounts to a special case of the non-tempered Gan-Gross-Prasad conjectures in \cite{gan2020branching}: if $\pi_1 $
and $\pi_2 $
have trivial central character, then $\Pi_S (\pi_1,\pi_2) $
descends to $\pigs_4 =\SO_5, $
and the period (\ref{equation: intro period definition}) reduces to a period for the split GGP pair $\SO_4\subset\SO_5. $
Although $\Pi_S(\pi_1,\pi_2)$ is tempered, the  automorphic representation of $\SO_4 $
corresponding to the forms $\beta(p(h))$ on $H$ is not, and so this period falls outside the scope of the usual GGP conjecture. 

To deduce Theorem \ref{thm: intro global mean thm} from Theorem \ref{thm:intro periods}, one additional ingredient is needed. In the period integrals (\ref{equation: intro period definition}), one really wants to consider only vectors $\gamma$ and $\beta$ that contribute to cohomology, which in our case is equivalent to generating a minimal $K$-type at archimedean places. The most delicate part is to write such a vector $\gamma$ as a theta lift $\theta_\xphi(\alpha)$, which requires a particular choice of archimedean component for the Schwartz function $\xphi$; the correct choice is calculated using local Howe duality. Once we know which $\xphi$ to consider, we can  evaluate the relevant archimedean zeta integrals  to show that the periods (\ref{equation: intro period definition}) are nontrivial.
\subsubsection*{Proof of Theorem \ref{Intro theorem Hodge classes}}
The main difficulty in the proof of Theorem \ref{Intro theorem Hodge classes} is to find a nontrivial family of Hodge classes on $S(\boldgsp_4)\times S(\boldgl_2)$ (besides the ones coming from the algebraic cycle $S(\boldsymbol H)$). Once we have a good supply of Hodge classes, the proof that they induce nontrivial maps (\ref{eq: expected intro}) uses similar methods to the proof of Theorem \ref{thm: intro global mean thm}.

This family of Hodge classes is constructed using  nontempered, cohomological automorphic representations of $\GSp_6 (\A_F) $
which contribute to cohomology in degree $4d $, and whose contribution
consists entirely of Hodge cycles. More precisely, let $S= S_f\cup S_\infty$ with $|S_f|$ even, and let $B$ be the quaternion algebra over $F$ which is ramified exactly at $S_f $. Assume $S_f\neq\emptyset$, i.e. $B$ is nonsplit. Then for any auxiliary automorphic representation $\pi $ of $PB(\A_F)^\times$ of parallel weight 6, we consider $\Theta (\pi\boxtimes\mathbbm1) $, the theta lift from $\GSO (B) $
to $\GSp_6 $ of the automorphic representation $\pi\boxtimes\mathbbm 1 $
of $\GSO (B)\simeq B ^\times\times B ^\times/\mathbb G_m. $ 
We do not prove that $\Theta (\pi\boxtimes\mathbbm1)$ is irreducible, but for any   constituent $\widetilde \Pi$ of $\Theta (\pi\boxtimes\mathbbm1)$, we have:
\begin{equation}\label{Equation: intro Hodge type}
    H ^ {4d}_{(2)} (S (\boldgsp_6),\C) [\widetilde\Pi_f] = H ^ {2d, 2d}_{(2)} (S (\boldgsp_6),\C) [\widetilde\Pi_f]
\end{equation}
and
\begin{equation}\label{Equation: intro Tate}
\Gal(\overline \Q/F^c)  \text{ acts trivially on }  IH ^ {4d} (S(\boldgsp_6), \overline{\Q}_\l(2d)) [\widetilde\Pi_f],
\end{equation}
where we identify $\overline{\Q}_\l \simeq \C$ to make sense of (\ref{Equation: intro Tate}).
In fact, (\ref{Equation: intro Hodge type}) and (\ref{Equation: intro Tate}) remain true for any $\widetilde \Pi$ which is only \emph{nearly equivalent} to a  constituent of $\Theta(\pi\boxtimes \mathbbm 1)$. 

To prove (\ref{Equation: intro Hodge type}), it suffices (by Matsushima's formula) to understand the Lie algebra cohomology of any $\widetilde\Pi_\infty$ such that $\widetilde\Pi_f\otimes \widetilde\Pi_\infty$ is automorphic. For this it is better to restrict to $\SP_{6}$, where we may use the endoscopic classication of Arthur \cite{arthur2013endoscopic}. 
In particular, any irreducible constituent $\widetilde\Pi'$ of $\widetilde\Pi_f\otimes \widetilde\Pi_\infty|_{\SP_6(\A_F)}$ has a global Arthur parameter which depends only on $\pi$, and each component $\widetilde\Pi'_v$ of $\widetilde\Pi'$ lies in the corresponding local Arthur packet for $\SP_6(F_v)$. Since Lie algebra cohomology is issentially insensitive to restriction to $\SP_6$, it suffices to understand the cohomology of all $\widetilde\Pi'_v$ in the local Arthur packet, for each $v|\infty$. This is accomplished using the construction of archimedean  packets in \cite{adams1987endoscopic} and the classification of unitary representations with nonzero cohomology \cite{vogan1984unitary}.

To prove (\ref{Equation: intro Tate}), suppose that $p\neq\l $ splits completely in $F $ and that $\widetilde\Pi_v $ is  spherical  for all $v|p$. Then the generalized Eichler-Shimura relation proven by Lee \cite{lee2020eichler,lee2022semisimplicity} provides a polynomial $P (X) $ such that $P (\Frob_p) = 0$ on $IH ^\ast(S(\boldgsp_6), \overline{\Q}_\l) [\widetilde\Pi_f].$
The coefficients of $P (X)$ depend on the Satake parameters of $\widetilde\Pi_v $
for $v|p$, which in turn are determined by those of $\pi_v $ via the spherical theta correspondence for orthogonal-symplectic similitude pairs (Proposition \ref{prop:spherical}).  It turns out that $P (X) $
has a unique root of weight $4d$, which is $p^{-2d}$. (The other roots correspond to appearances of $\widetilde\Pi_f$ in higher cohomological degrees.) Thus $\Frob_p = p^{-2d}$ on  $IH ^{4d}(S(\boldgsp_6), \overline{\Q}_\l) [\widetilde\Pi_f]$ for all such $p$, which shows (\ref{Equation: intro Tate}) by the Chebotarev density theorem.

Let us now return to the main construction. For $\widetilde H\coloneqq\GSp_4\times_{\mathbb G_m}\GL_2\subset \GSp_{6},$
we have inclusions of Shimura varieties \begin{equation*}
S(\boldgsp_6) \xleftarrow{\iota_1}S(\widetilde {\boldsymbol H})\xrightarrow{\iota_2}S(\boldgsp_4)\times S(\boldgl_2)
\end{equation*}
such that $\iota_2$ is open and closed. Thus we obtain a well-defined map \begin{equation}
 \label{eq: intro induced} \iota_{2,\ast}\circ \iota^\ast_1: IH^\ast(S(\boldgsp_6), E(d)) \to IH^\ast(S(\boldgsp_4)\times S(\boldgl_2), E(d)).
\end{equation}
If $\widetilde\Pi$ is defined over $E$, then the subspace $IH^\ast(S(\boldgsp_6), E(d))[\widetilde\Pi_f]$ makes sense, and we obtain from (\ref{Equation: intro Hodge type}), (\ref{Equation: intro Tate}), and (\ref{eq: intro induced}) a space of Galois-invariant Hodge classes
$$\Hodge(\widetilde \Pi)\subset 
IH^{4d}(S(\boldgsp_4)\times S(\boldgl_2), E(d)).$$
However, we will see in a moment that, to obtain the nonvanishing of a Hodge class constructed in this way, we will have to allow $\pi$ -- hence also $\widetilde\Pi$ -- to be defined over an arbitrary number field.  
In general, let $S(\pi)$ denote the set of automorphic representations of $\GSp_6(\A_F)$ which are nearly equivalent to a constituent of $\Theta(\pi\boxtimes \mathbbm 1)$. Then the subspace 
$$\sum_{\sigma\in \Aut(\C/\Q)}\sum_{\widetilde\Pi\in S(\pi^\sigma)} IH^{4d} (S(\boldgsp_6), \C)[\widetilde\Pi_f] \subset IH^{4d}(S(\boldgsp_6),\C)$$
is stable under the $\Aut(\C/\Q)$ action on the coefficients. After descending to $E$, its image under (\ref{eq: intro induced}) defines a subspace
$$\Hodge(\pi) \subset IH^{4d}(S(\boldgsp_4)\times S(\boldgl_2), E(d))$$
which again consists of Galois-invariant Hodge classes.

  It remains to show that some element
 $\xi\in \Hodge(\pi)$  induces a nonzero map as claimed in Theorem  \ref{Intro theorem Hodge classes}. Similarly to the proof of Theorem \ref{thm: intro global mean thm}, we reduce this question to showing that the triple product period integral
\begin{equation}\label{eq: triple period}\int_{[Z_{\widetilde H}\backslash \widetilde H]} \theta(\alpha)(h, h') \beta(h)\gamma(h')\d (h, h'),\;\; \alpha\in \pi\boxtimes\mathbbm 1,\; \beta\in \Pi_S(\pi_1,\pi_2),\;\gamma\in \pi_2^\vee\end{equation}
is nonzero for some choice of $\pi$ and some choice of test vectors $\alpha$, $\beta$, and $\gamma$. Here $\widetilde H$ is parametrized by pairs $(h, h')\in \GSp_4\times \GL_2$, and the theta lift, which again depends on a choice of Schwartz function, is from $\GSO(B)$ to $\GSp_6$. 
The relevant seesaw diagram for this period is: 
\begin{center}
    \begin{tikzcd}
    \GSp_6 \arrow[d,dash]\arrow[dr,dash] & \GSO(B)\times_{\mathbb G_m} \GSO(B) \arrow[d,dash]\arrow[dl,dash] \\ \GSp_4\times_{\mathbb G_m} \GL_2 & \GSO(B)
    \end{tikzcd}
\end{center}
The seesaw identity reduces (\ref{eq: triple period}) to
\begin{equation}\label{eq:triple 2}
\int_{[\operatorname{PGSO}(B)]} \alpha (g)\theta (\beta) (g)\theta (\gamma) (g)\d g,
\end{equation}
where the theta lifts are now from $\GSp_4 $
and $\GL_2 $
to $\GSO (B) $. (Under the assumption that $B $
is nonsplit, all the integrals involved in the seesaw identity converge absolutely.)
The theta lift $\theta (\gamma) $
runs over $(\pi_2 ^ B)^\vee\boxtimes(\pi_2 ^ B)^\vee $
as $\gamma$ varies, and the image of the  theta lift $\theta (\beta) $
includes 
 $\pi_1^B\boxtimes \pi_2^B$ 
as $\beta$ varies. We choose $\alpha$ to be a Hilbert modular eigenform on $PB^\times(\A_F)$ such that  $\langle f_1^B\cdot f_2^B, \alpha\rangle _{\text{Pet}}\neq 0$, where $f_1^B\in \pi_1^B$ and $f_2^B\in (\pi_2^B)^\vee$ are holomorphic newforms, and let $\pi$ be the automorphic representation generated by $\alpha$. (Note that $\alpha$ may be defined over a larger number field than $f_1^B\cdot f_2^B$.)  Having made this choice of $\pi$ and $\alpha$, it  follows that (\ref{eq:triple 2}) is nonzero for  appropriate choices of $\beta$ and $\gamma$. 
\subsection*{Overview of the higher-weight case}
To simplify the notation, assume for now that $F =\Q $. The representations $\Pi_S(\pi_1,\pi_2)$ are cohomological whenever $\pi_1 $
and $\pi_2 $
have  weights $m_1\geq m_2 + 2 \geq 4$. In this situation one again finds nontrivial  Galois-invariant  classes in the \'etale cohomology of $S (\GSp_4)\times S (\GL_2), $
where now we take cohomology with coefficients in a local system  depending on $m_1 $ and $m_2. $
However, it is only for special choices of weights that we are able to formulate an analogue of Theorem \ref{thm: intro global mean thm}. 

So suppose that $\pi_1,\pi_2$ have weights $m+2$ and $m$ respectively, for an integer $m \geq 2.$ 
Let $V_{(m-2,0)}$ be the representation of $\GSp_4(\Q)$ with highest weight $(m-2,0)$ and central character $t\mapsto t^{m-2}$, and let $V_{m-2}$ be the representation $\Sym^{m-2} V_{\text{std}}$ of $\GL_2(\Q)$. We write $\mathcal V_{(m-2,0)}$ and $\mathcal V_{m-2}$ for  the corresponding local systems of $E$-vector spaces on $S(\GSp_4)$ and $S(\GL_2)$.  
(The normalization of these local systems is different in the text.)
With the corresponding normalization of the central characters of $\pi_i$, we have \cite{weissauer2009endoscopy}: 
\begin{equation}
   H ^ 3_{\et, c} (S (\GSp_4)_{\overline \Q}, \mathcal V_{(m-2,0),\lambda})_{\Pi_S (\pi_1,\pi_2)_f} =\begin {cases}\rho_{\pi_2}(1), &\infty\not\in S\\\rho_{\pi_1}, &\infty\in S\end {cases}
\end{equation}
In particular, when $\infty\not\in S$ one finds a nontrivial Galois-invariant map
\begin{equation}
     H ^ 3_{\et, c} (S (\GSp_4)_{\overline \Q}, \mathcal V_{(m-2,0),\lambda})[{\Pi_S (\pi_1,\pi_2)_f}] \to H^1_{\et}(S(\GL_2)_{\overline{\Q}}, \mathcal V_{m-2,\lambda})[\pi_{2,f}].
\end{equation}
The significance of our special choice of weights is that the local system
$$\iota ^\ast\mathcal V_{(m -2, 0)}^\check\otimes p ^\ast\mathcal V_{m -2}$$
on $S(H) $ has the constant local system $\underline\Q $ as a direct factor with multiplicity one. We may therefore define \begin {equation}[S (H)] \in H^4 \left( S (\GSp_4)\times S (\GL_2), \mathcal V_{(m-2,0)}^\check\boxtimes \mathcal V_{m -2}(2) \right)\end {equation}
using the pushforward of the fundamental class on $S (H) $.
\begin{theo}\label{Intro theorem higher weight cycle}
Let $\pi_1$ and $\pi_2$ be cuspidal automorphic representations of $\GL_2 (\A_\Q) $
of weights $m +2 $ and $m $, respectively, and the same central character (suitably normalized). Then the composite map $$
H ^ 3_c (S (\GSp_4), \mathcal V_{(m-2,0)} (1)) [\Pi_S (\pi_1,\pi_2)_f]\xrightarrow{[S (H)]_\ast}H ^ 1 (S (\GL_2), \mathcal V_{m -2})\twoheadrightarrow H ^ 1 (S (\GL_2), \mathcal V_{m -2}) [\pi_{2, f}]$$
is nonzero if and only if  $S =\emptyset $. In this case its image generates the $\GL_2 (\A_{\Q, f}) $-module $H ^ 1 (S (\GL_2), \mathcal V_{m -2}) [\pi_{2, f}].$
\end{theo}
It is likely that Theorem \ref{Intro theorem Hodge classes} could be generalized to the weights $m_1\geq m_2+2\geq 4$; however, we have restricted our attention to the context of Theorem \ref{Intro theorem higher weight cycle}.
\begin{theo}\label{Intro theorem higher weight Hodge classes}
Let $\pi_1 $ and $\pi_2 $
be cuspidal automorphic representations of $\GL_2 (\A_\Q) $
of  weights $m+2$ and $m$, respectively, with the same central character (suitably normalized). Let $S$ be a set
of places of $\Q $
at which both $\pi_i $
are discrete series, such that $|S|\geq 2$ is even and $\infty\not\in S$. Then there exists a Hodge class $$\xi\in H^4 \left( S (\GSp_4)\times S (\GL_2), \mathcal V_{(m-2,0)}^\check\boxtimes \mathcal V_{m -2}(2) \right) $$
such that:
\begin{enumerate}
    \item For all finite places $\lambda $ of $E $, the image of $\xi$
    in $\lambda $-adic \'etale cohomology is $\Gal (\overline\Q/\Q) $-invariant.
    \item The composite map \begin{equation*} H ^ 3_{ c} (S (\GSp_4), \mathcal V_{(m-2,0)})[{\Pi_S (\pi_1,\pi_2)_f}] \to H^1(S(\GL_2), \mathcal V_{m-2})\twoheadrightarrow H^1(S(\GL_2), \mathcal V_{m-2})[\pi_{2,f}] \end{equation*}
is  nontrivial, and its image generates the $\GL_2(\A_{F,f})$-module $H^d(S(\boldgl_2),\mathcal V_{m-2,\lambda})[{\pi_{2,f}}]$.
\end{enumerate}
\end{theo}
The proofs of Theorems \ref{Intro theorem higher weight cycle} and \ref{Intro theorem higher weight Hodge classes} follow the same lines as the overview given above, with only minor modifications. 
When $F\neq\Q $, we also have similar results assuming $\pi_1$ and $\pi_2$ have vector-valued weights $(m_v+2)_{v|\infty}$ and $(m_v)_{v|\infty}$. However, in this case it is more cumbersome to write down the definitions of the appropriate local systems. The precise results are included in Theorems \ref{biggie thm} and \ref{main theorem  Hodge classes} below.

One could also ask for an analogue of Theorem \ref{Intro theorem higher weight Hodge classes} that uses $\pi_1$, the higher-weight representation, in the place of $\pi_2.$ Unfortunately, our construction does not appear to yield any results in this direction.
\subsection*{Comparison with previous work}
The simplest case of  Langlands functoriality giving rise to Tate classes is the Jacquet-Langlands correspondence for cohomological representations of inner forms of $\GL_{2, F} $. 
For the transfer between quaternion algebras $B_1 $ and $B_2 $ which are split at exactly one archimedean place, the Shimura varieties associated to $B_1 ^\times $
and $B_2 ^\times $
are curves. The resulting Tate classes are known to arise from cycles by Faltings's isogeny theorem \cite{faltings1983endlichkeitssatze}, but no more explicit construction of these algebraic cycles is known.
When the relevant Shimura varieties have higher dimension, Ichino and Prasanna \cite{ichino2018hodge} have shown that the Jacquet-Langlands transfers (for general weights) are induced by Hodge cycles. Their construction is similar to the one used to prove Theorem \ref{Intro theorem Hodge classes}. However, in the Jacquet-Langlands setting there is no natural algebraic cycle such as $S(\boldsymbol H)$, so there is no analogue of Theorem \ref{thm: intro global mean thm}. 

In the context of Yoshida lifts, the case $S=\emptyset$ and $F = \Q$ of Theorem \ref{thm: intro global mean thm} was proven by Lemma in \cite{lemma2020algebraic}, using different methods that apply only to generic representations. The qualitatively different behavior of generic and non-generic members of the $L$-packet in Theorem \ref{thm: intro global mean thm} is a new phenomenon that had not previously appeared in the literature.
\subsection*{Arithmetic implications}
This work was originally motivated by a question of Weissauer in \cite{weissauer2005four}, which can be paraphrased as follows: if $F= \Q$ and $\pi_2$ is the automorphic representation associated to an elliptic curve $E/\Q$, then the motive associated to $E$ appears attached to members of  the $L$-packet $\Pi(\pi_1,\pi_2)$ in the cohomology of $S(\GSp_4)$. Can we then use Shimura curves on $S(\GSp_4)$ to construct interesting Selmer classes for $E$ in the spirit of Heegner points? Theorem \ref{thm: intro global mean thm} implies that, when applied to quaternionic Shimura curves and a generic representation $\Pi_\emptyset(\pi_1,\pi_2)$,  this construction would simply recover the Heegner points on $E$.
Indeed,  all appearances of the motive of $E$ attached to generic representations $\Pi_\emptyset(\pi_1,\pi_2)$
are fully accounted for by Hecke translates of the correspondence from $S(\GSp_4)$ to the modular curve $S(\GL_2)$ induced by (\ref{eq: candidate cycle}), and nonsplit quaternionic Shimura curves on $S(\GSp_4)$ are necessarily sent to CM divisors on $S(\GL_2)$ under this correspondence.  It is an intriguing question whether Weissauer's construction yields new Selmer classes when applied to quaternionic Shimura curves and the  non-generic members of the $L$-packets $\Pi(\pi_1,\pi_2)$. 

\subsection*{Organization of the paper}
In  \S\ref{sec:prelim}, we give some basic notations and conventions. In \S\ref{sec: cohomology}, we recall the plectic version of Matsushima's formula and its relation to vector-valued automorphic forms. In \S\ref{section: Similitude theta lifting}, we give notations and conventions for similitude theta lifts. This section also contains a proof of the $L$-functoriality for similitude theta lifts of spherical representations from orthogonal to symplectic groups (Proposition \ref{prop:spherical}); this is presumably well-known to experts. In \S\ref{sec: yoshida basics}, we recall the construction of the Yoshida lift $L$-packets via theta lifts, and compute the plectic Hodge structures associated to $\Pi_S(\pi_1,\pi_2)_f$. The material up to this point is necessary for all the main results. However, the proofs of Theorems \ref{thm:intro periods} and \ref{thm: intro global mean thm}, which are given in \S\ref{Section:. See Yoshida} and \S\ref{sec:coh generic proof}, respectively, are logically independent of the proof of Theorem \ref{Intro theorem Hodge classes}. The only exceptions are some results on the archimedean theta correspondence in \S\ref{subsec: Archimedean calculations}. In \S\ref{sec:nontempered}, we study the nontempered representations used for the construction of Hodge classes. In \S\ref{sec: triple product periods}, we compute the vector-valued triple product periods that are necessary for the nonvanishing of the Hodge classes. The proof of Theorem \ref{Intro theorem Hodge classes} is completed in \S\ref{sec: proof Hodge}. 

\subsection*{Acknowledgements}
The author is grateful to Mark Kisin, for his consistently valuable advice and encouragement; Wei Zhang, for pointing out the relation of Theorem \ref{thm:intro periods} to the nontempered GGP conjecture; and Si Ying Lee,  Siyan Daniel Li-Huerta, Aaron Pollack, Kartik Prasanna,  Alexander Petrov, Matteo Tamiozzo, and Salim Tayou, for a variety of helpful conversations and correspondence.  This work was supported by NSF Grant \#DGE1745303.
\section {Preliminaries}\label{sec:prelim}
\subsection{Basic notations}
\subsubsection{}\label{places etc}
Throughout this article, $F $ is a fixed totally real number field of degree $d$ and discriminant $D_F$, $\O_F$ is its ring of integers, and $\A_F $ is its ring of adeles. For each place $v $ of $F$, denote by $F_v$ the completion; if $v$ is non-archimedean, $\O_{v}$ is the valuation ring of $F_v $, $\varpi_v\in \O_{v}$ is the uniformizer, and $q_v = \#\O_{v}/\varpi_v$. For archimedean $v $, $q_v = 1. $ 
The Haar measure  on the additive group $\A_F$ is the product measure $\d a = \d a_f \prod_{v|\infty} \d a_v,$ where $\d a_f$ is the Haar measure on $\A_{F,f}$ such that $\widehat{\O}_F$ has volume 1 and $\d a_v$ is the standard measure on $F_v\cong \R$.
\subsubsection{}
If $G $ is an algebraic group over $F $, $[G] $ denotes the adelic quotient $G (F)\backslash G (\A_F) $. If $\d g$ denotes a Haar measure on $G(\A_F)$, then we write $\d g$ as well for the quotient Haar measure on $[G]$ (where $G(F)$ is given the counting measure). 
\subsubsection{}
We fix the additive character $\psi =\psi_0\circ\tr $ of $F\backslash\A_F$, where $\psi_0:\Q\backslash\A\to\C$ is the unique unramified character such that $\psi_0 (x) = e ^ {2\pi ix} $ for $x\in \R$.
\subsubsection{}\label{Subsubsection: Omega}
For any $m$, let $\omega_m:\R^\times \to \R^\times$ be the character $$t\mapsto t^{m - 2\cdot\lfloor\frac{m}{2}\rfloor}.$$ If $\boldsymbol m = (m_v)_{v|\infty}$, let $\omega_{\boldsymbol m}:(F\otimes \R)^\times \to\R^\times$ be the character $\otimes_{v|\infty}\omega_{m_v}.$
These characters will be used as the central characters for ``nearly unitary'' normalizations of automorphic forms appearing in cohomology.
\subsubsection{}
If $V$ is a vector space over a local field $k$ (either Archimedean or non-Archimedean), then $\mathcal S_{k}(V)$ is the Schwartz space of functions on $V$. If $V $
is a vector space over $F $ and $v$ is a place of $F $, then $\mathcal S_{F_v} (V)$
denotes the space of Schwartz functions on $V\otimes_F F_v$. 
Likewise, we write $\mathcal S_{F\otimes \R}(V)$ for the tensor product of the Schwartz spaces $\mathcal S_{F_v} (V) $ as $v$ ranges over archimedean places of $F $. 
\subsection{Conventions for $\GL_2$ and $\SL_2$}
\subsubsection{}
The standard Borel and unipotent subgroups of $\GL_2 $ are denoted $B $ and $N $, respectively; $\overline B$ denotes the image of $B $ in $\PGL_2. $ We shall abbreviate by $c\mapsto h_c$ the section of $\debt:\GL_2\to \mathbb G_m$ given by $h_c = \begin{pmatrix} 1 & 0 \\ 0 & c \end{pmatrix}.$ 
\subsubsection{}\label{subsubsec:measures}
For each non-archimedean place $v $ of $F $, we normalize the Haar measure $\d g_v $ on $\PGL_2 (F_v) $ to assign volume 1 to $\PGL_2 (\O_v) $, and likewise for $\SL_2 (F_v) $. For non-archimedean $v $, we choose the Haar measure $\d g_v $ on $\PGL_2(F_v)\cong \PGL_2(\R)$
given by:\begin{equation}
    \d g_v = \frac {\d a \d t \d \theta} {\pi t ^ 2},\;\; g_v =\begin {pmatrix} 1 & a\\0 & 1\end {pmatrix}\begin {pmatrix}  t & 0\\0 & 1\end {pmatrix}\begin {pmatrix}\cos\theta &\sin\theta\\ -\sin\theta &\cos\theta\end{pmatrix},\;\; a\in \R,t\in \R^\times, \theta\in [0,\pi).
\end{equation}
On $\SL_2 (F_v)\cong \SL_2(\R) $, we choose the Haar measure $\d g_v$ given by:
\begin{equation}
       \d g_v = \frac {\d a \d t \d \theta} {2\pi t ^ 2},\;\; g_v =\begin {pmatrix} 1 & a\\0 & 1\end {pmatrix}\begin {pmatrix}  t^{1/2} & 0\\0 & t^{-1/2}\end {pmatrix}\begin {pmatrix}\cos\theta &\sin\theta\\ -\sin\theta &\cos\theta\end{pmatrix},\;\; a\in \R,t\in \R_{>0}, \theta\in [0,2\pi).
\end{equation}
\subsubsection{}
For the standard compact subgroup $\SO (2) $
of $\SL_2(\R)$, we denote by $\chi_m: \SO(2)\mapsto \C^\times$ the character $$\begin{pmatrix}
 \cos\theta &\sin\theta\\-\sin\theta &\cos \theta\end{pmatrix} \mapsto (\cos\theta + i \sin \theta)^m.$$ 
\subsection{Conventions for symplectic groups}
\subsubsection{}\label{subsubsec: standard symp space}
Let $J$ be the matrix $\begin{pmatrix} 0 & 1 \\ -1 & 0 \end{pmatrix}$. Then, for any field $k$, the block-diagonal matrix $\begin {pmatrix} J & &\\&\ddots&\\& & J\end {pmatrix} $ defines a symplectic pairing on the $k$-space $W_{2n,k} =\langle e_1,\cdots, e_{2n}\rangle $
such that $$W_{2n,k} =\langle e_1, e_3,\cdots, e_{2n -1}\rangle\oplus\langle e_2, e_4,\cdots, e_{2n}\rangle $$
is a decomposition into maximal isotropic subspaces; we refer to $W_{2n,k} $ as the standard symplectic space of dimension $2n $. 
The symplectic group $\SP_{2n,k}$ and the general symplectic group $\GSp_{2n,k}$ are the isometry and similitude groups, respectively, of $W_{2n,k} $. When not otherwise specified, $k = F$.
\subsubsection{}\label{Sub sub: weights for unitary group}The maximal compact-modulo-center subgroup of the symplectic group $\GSp_{2n,\R}$ is $K_{n}\simeq (U(n) \times \R^\times) / \set{\pm 1}$, consisting of the matrices whose $2\times 2 $ blocks commute with $J $. When $K_n $ is viewed as a subgroup of $\GSp_{2n} (F_v) $
we write it $K_{n, v} $.
There is a maximal compact torus $T\subset U(n) $
such that
$$\mathfrak t = \begin{pmatrix}
   \alpha_1J & &
\\&\ddots&\\& &\alpha_nJ
\end{pmatrix},\;\; \alpha_i\in \R.$$
We parameterize the weights of $U(n)$ by tuples of integers $(m_1,\cdots, m_n) $,
 corresponding to the character
$$\begin{pmatrix}
 \alpha_1J & &\\&\ddots &\\& &\alpha_nJ
\end {pmatrix}\mapsto m_1\alpha_1+\cdots + m_n\alpha_n. $$
When $n = 1, $
the character $\chi_m\boxtimes\omega_m ^ {-1} $
on $U (1)\times Z_{\GL_2} $
descends to a character of $K_1, $
which we will also denote by $\chi_m $; we hope that this will cause no confusion.

\section{Cohomology of Shimura varieties}\label{sec: cohomology}
\subsection{Plectic Hodge structures}
\subsubsection{}\label{subsubsec:Plectic notation}
Let $G $ be a reductive group over $F $, and let $\boldsymbol G =\Res_{F/\Q} G_F $. Since\begin{equation}
    \boldsymbol G (\R)=\prod_{v |\infty} G(F_v)= \prod_{v|\infty}G_v(\R),
\end{equation}
a Shimura datum  $(\boldsymbol G, \boldsymbol X) $ is necessarily a product $\boldsymbol X\simeq\prod_{v |\infty} X_v $. If $K_v\subset G_v (\R) $
denotes the stabilizer of a distinguished point $h_v\in X_v $, then the stabilizer of the corresponding point $\boldsymbol h\in\boldsymbol X $
is
$$\boldsymbol K_\infty=\prod_{v|\infty} K_v.$$ Given a neat compact open subgroup $$K\subset\boldsymbol G(\A_{\Q,f})=G (\A_{F,f}), $$
one has a smooth algebraic Shimura variety $S_K ({\boldsymbol G},\boldsymbol X) $ such that $$S_K(\boldsymbol G, \boldsymbol X) (\C) = G(F)\backslash G(\A_{F,f})\times \boldsymbol X/ K;$$ the inverse limit over $K $
defines the pro-algebraic Shimura variety $ S(\boldsymbol G) $. (We usually drop $\boldsymbol X$ since it will be clear from context.)
Finally, given an algebraic representation $\rho $ of $\boldsymbol G$ 
on an $E$-vector space $V$, we have for each level subgroup $K$ the local system $\mathcal V_K$ on $ S_K(\boldsymbol G) $ whose total space over $S_K(\boldsymbol G)(\C)$ is \begin{equation}
    G(F) \backslash G (\A_{F,f})\times\boldsymbol X\times V/K.
\end{equation}
The local systems $\mathcal V_K$ are compatible as $K$ varies, and we write $\mathcal V$ for this compatible collection of local systems on $S(\boldsymbol G)$.

\subsubsection{}
Assuming $E\subset \C$,
Matsushima's formula for the  $L^2$ cohomology of $ S(\boldsymbol G) $ is:
\begin{equation}\label{eqn:matsushima}
    H ^\ast_{(2)} ( S(\boldsymbol G), \mathcal V_\C)\cong\bigoplus_{\pi =\pi_f\otimes\pi_\infty}m_{\text{disc}}(\pi)\cdot \pi_f\otimes H^\ast(\text{Lie }\boldsymbol {G}; \boldsymbol K_\infty, \pi_\infty^{\rm{sm}}\otimes V_\C)).
\end{equation}
Here $\pi $ runs over cuspidal automorphic representations of $G(\A)$, $m_{\text{disc}}(\pi)$ refers to the multiplicity in the discrete spectrum, and $\pi_\infty ^ {\rm {sm}} $
is the dense subspace of smooth vectors. Moreover (\ref{eqn:matsushima}) is equivariant for the natural actions of $G (\A_{F,f}) $ on both sides. Suppose  $V_\C=\otimes_{v} V_v $, where $V_v$ are $\C$-vector spaces equipped with algebraic representations $\rho_v$ of $G_v(\R)$, such that
 $\rho$ factors as \begin{equation}\label{eq:rho decomposes}\rho: G (F)\hookrightarrow G (F\otimes \R)\simeq\prod_v G _v(\R)\xrightarrow{\otimes \rho_v} \prod_v \Aut(V_v).\end{equation}
Since the Lie algebra of $\boldsymbol G $
is  $\prod_{v |\infty}\mathfrak g_v $, the right hand side of (\ref{eqn:matsushima})  has a decomposition (cf. \cite{nekovar2016introduction}):\begin{equation}\label{eqn:matsushima plectic}
    \bigoplus_{\boldsymbol p,\boldsymbol q} \left(\bigoplus_{\pi_f\otimes\pi_\infty}m_{\text{disc}}(\pi) \cdot \pi_f\otimes \bigotimes_{v|\infty} H^{p_v,q_v}(\mathfrak g_v,  K_v, \pi_v^{\rm{sm}}\otimes V_{v})\right).
\end{equation}
Here $\boldsymbol p $ and $\boldsymbol q $
are plectic Hodge types, i.e. tuples of positive integers $(p_v)_{v |\infty} $ and $(q_v)_{v |\infty} $.  Then (\ref{eqn:matsushima}) induces a plectic Hodge decomposition on $ H ^\ast_{(2)} ( S(\boldsymbol G), \mathcal V_\C) $, written:
\begin{equation}\label{EQ: plectic conch types}
    H ^ {\ast}_{(2)} ( S(\boldsymbol G),\mathcal V_\C) =\bigoplus_{\boldsymbol p,\boldsymbol q}H ^ {\boldsymbol p,\boldsymbol q}_{(2)} ( S(\boldsymbol G), \mathcal V_\C).
\end{equation}
\begin{rmk}\label{remark: Hodge theory disclaimer}
Because this decomposition does not take into account the variation of Hodge structures on $\mathcal V_\C $, it does not compare directly with the canonical mixed Hodge structure on $H ^\ast (S (\boldsymbol G),\mathcal V_\C) $.
For this reason, (\ref{EQ: plectic conch types})
should be viewed more as a computational tool then as a suitable definition of ``the'' plectic Hodge structure on $H ^\ast_{(2)} (S (\boldsymbol G),\mathcal V_\C) $.
\end{rmk}
\subsection{Realizing automorphic forms in cohomology}
\subsubsection{}
The complex structure on $X_v $ induces a decomposition $$\mathfrak g_{v,\C} =\mathfrak k_\infty \oplus\mathfrak p _ {v,+}\oplus\mathfrak p _{v, -}. $$
We define
\begin{equation}
    \wedge ^ {\boldsymbol p,\boldsymbol q}\mathfrak p_G ^ {\ast}\coloneqq\otimes_{v |\infty} (\wedge ^ {p_v}\mathfrak p_{v,+} ^ {\ast}\otimes\wedge ^ {q_v}\mathfrak p_{v,-} ^ {\ast}),
\end{equation}
and let $(\sigma ^ {\boldsymbol p,\boldsymbol q}, \wedge ^ {\boldsymbol p,\boldsymbol q} )$
be the corresponding natural representation of $\boldsymbol K_\infty $.
The vector bundle  $\Omega ^ {\ast} $
of differential forms on $ S(\boldsymbol G) $
has a decomposition $$\Omega ^\ast =\oplus_{\boldsymbol p,\boldsymbol q}\Omega ^ {\boldsymbol p,\boldsymbol q}, $$
where
the vector bundle $\Omega ^ {\boldsymbol p,\boldsymbol q} $ of $(\boldsymbol p,\boldsymbol q) $-forms on $ S(\boldsymbol G) $
corresponds to the local system whose complex points are:
$$G (F)\backslash G (\A_{F,f})\times\boldsymbol G (\R)\times\wedge ^ {\boldsymbol p,\boldsymbol q}\mathfrak p_G^\ast/\boldsymbol K_\infty. $$
In particular, the space $\Gamma_{(2)}(\Omega^{\boldsymbol p,\boldsymbol q}\otimes \mathcal V_\C)$ of $L^2$ global $( {\boldsymbol p,\boldsymbol q}) $-forms with coefficients in $\mathcal V_\C$ is identified with:
\begin{equation}
    \set {f\in C_{(2)} ^\infty (G (\A_F))\otimes V_\C\otimes\wedge ^ {\boldsymbol p,\boldsymbol q}\mathfrak p_G^\ast\,:\, f (\gamma gk) =\rho (\gamma)\sigma^ {\boldsymbol p,\boldsymbol q} (k ^ {-1})f (g),\,\forall\gamma\in G (F),\, k\in\boldsymbol K_\infty}.
\end{equation}
Here $C_{(2)} ^\infty(G(\A_F))$ is the space of smooth $L^2$ functions on $G(\A_F) $;  by definition, we have:
\begin{equation}\label{eqn:pqforms global to coh}\Gamma _{(2)}(\Omega ^ {\boldsymbol p,\boldsymbol q}\otimes \mathcal V_\C) \twoheadrightarrow H ^ {\boldsymbol p,\boldsymbol q}_{(2)} ( S(\boldsymbol G), \mathcal V_\C). \end{equation}
Finally, we remark that there is a canonical isomorphism:
\begin{equation}
\begin{split}
    (\mathcal{A}_{(2)}(G(\A_F))\otimes V_\C|_{\boldsymbol K_\infty}\otimes\wedge ^ {\boldsymbol p,\boldsymbol q}\mathfrak p_G^\ast) ^ {\boldsymbol K_\infty}&\xrightarrow{\sim}\Gamma_{(2)} (\Omega ^ {\boldsymbol p,\boldsymbol q}\otimes \mathcal V_\C)\\
    \xphi &\mapsto f_\xphi,\;\;
    f_\xphi (g)  =\rho(g_\infty)\xphi (g).
\end{split}
\end{equation}
Here $\rho(g_\infty)$ is defined via the decomposition (\ref{eq:rho decomposes}).
By composing with
(\ref{eqn:pqforms global to coh}), we obtain a realization of vector-valued automorphic forms in cohomology:
\begin{equation}\label{eqn:Realization cohomology definition}
     (\mathcal{A}_{(2)}(G(\A_F))\otimes V_\C|_{\boldsymbol K_\infty}\otimes\wedge ^ {\boldsymbol p,\boldsymbol q}\mathfrak p_G^\ast) ^ {\boldsymbol K_\infty}\twoheadrightarrow H ^ {\boldsymbol p,\boldsymbol q}_{\text {(2)}} ( S(\boldsymbol G),\mathcal V_\C).
\end{equation}
\subsection{Comparison with Betti cohomology}
 \subsubsection{}\label{subsubsec: diagram section}
 For any local system $\mathcal V$ associated to a \emph{complex} algebraic representation $V$ of $G(F)$, recall the canonical commutative diagram of $G(\A_{F,f})$-modules (cf. \cite[p. 293]{taylor1993thel}):
 \begin{center}
     \begin{tikzcd}
     H_{\operatorname{cusp}}^\ast(S(\boldsymbol G), \mathcal V) \arrow[r]\arrow[d] &H_{(2)}^\ast (S(\boldsymbol G), \mathcal V) \arrow[d,"\sim"] \arrow[rd] & \\
     H_c^\ast(S(\boldsymbol G),\mathcal{V}) \arrow[r] & IH^\ast(S(\boldsymbol G)^\ast,\mathcal V) \arrow[r] & H^\ast(S(\boldsymbol G), \mathcal V)
     \end{tikzcd}
 \end{center}
 Here $IH^\ast(S(\boldsymbol G)^\ast, \mathcal V)$ is the intersection cohomology of the minimal compactification, and the indicated map is an isomorphism by the proof of Zucker's conjecture \cite{looijenga19882,saper1990l2}. Moreover it is known that the composite $H^\ast_{\operatorname{cusp}}(S(\boldsymbol G),\mathcal V) \to H^\ast(S(\boldsymbol G), \mathcal V)$ is injective. Thus $L^2$ cohomology is related to inner cohomology by the inclusions:
 \begin{equation}\label{eq:inner coh sandwich}
     H^\ast_{\operatorname{cusp}} (S(\boldsymbol G), \mathcal V) \subset H^\ast_!(S(\boldsymbol G),\mathcal{V}) \subset \im H^\ast_{(2)} (S(\boldsymbol G),\mathcal V).
 \end{equation}
 \subsubsection{}\label{subsubsection: coefficient field definition etc.}
If $\Pi_f$ is a $\C[G(\A_{F,f})]$-module,  $\Pi_f $
is \emph{defined} over a subfield $E\subset\C $
if there exists a $E [G (\A_{F, f})]$-module $\Pi_f ^ E $
such that $\Pi_f ^ E\otimes_E\C\simeq\Pi_f $. In this case, if $\mathcal V $ is the $E $-local system associated to an $E $-linear representation $V $
of $G (F) $, then we write:
\begin{equation}\label{EQ: isotypic part over number field}
\begin{split}H_? ^\ast (S({\boldsymbol G}), \mathcal V)_{\Pi_{S_f}} &\coloneqq 
\Hom_{E[G(\A_{F,f})]} (\Pi_f^E, H_? ^\ast (S(\boldsymbol G), \mathcal V),\\
H_?^\ast(S(\boldsymbol G), \mathcal V )[\Pi_{S_f}]&\coloneqq \Pi_f^E\otimes_E
H_? ^\ast (S({\boldsymbol G}), \mathcal V)_{\Pi_{S_f}},\end{split}\end{equation}
where $H ^\ast_? $
denotes compactly supported, inner, or singular cohomology as $? =c $, !, or $\emptyset$.
 \subsubsection{}
 We say an irreducible admissible complex $G(\A_{F,f})$-representation $\Pi$ is \emph{Eisenstein} if $\Pi$ is a subquotient of a parabolic induction $\Ind_{P(\A_{F,f})}^{G(\A_{F,f})} \pi_f,$ for a parabolic subgroup $P = MN$ of $G$ and a discrete automorphic representation $\pi$ of $M(\A_F)$. For any admissible $E[G(\A_{F,f})]$-module $H$, we say $H$ is  Eisenstein if all irreducible constituents of $H\otimes \C$ are.
\begin{lemma}\label{lemma: retard corner}
Let $\mathcal V$ be the automorphic local system on $S(\boldsymbol G)$ associated to a $G(F)$-representation $(V, \rho)$. Then the $E[G(\A_{F,f})]$-module  $$H^i(S(\boldsymbol G), \mathcal V)/H^i_!(S(\boldsymbol G), \mathcal V)$$ is  Eisenstein.
\end{lemma}
\begin{proof}
This is well-known; a lucid exposition may be found in the preprint \cite[Chapter 9]{harder2019cohomology}.
\end{proof}
\subsection{Symplectic Shimura varieties}\label{subsection: symplectic Shimura}
\subsubsection{}\label{Subsubsection: local systems for symplectic ShimuraPart one}
When $G =\GSp_{2n} $, equipped with its usual Shimura datum,
the subgroup $\boldsymbol K_\infty $
is just
\begin{equation}
    \label{bold cayenne}
    \boldsymbol K_n\coloneqq\prod_{v} K_{n, v}\subset\GSp_{2n} (F\otimes\R).
\end{equation}
We establish some notation for local systems on $S (\boldgsp_{2n}) $. 
 Suppose given a tuple $\boldsymbol \lambda = (\lambda_v)_{v|\infty}$, where   $\lambda_v =(m_{1, v},\cdots, m_{n, v}) $ is a dominant weight
 of $\SP_{2n,\R} $. We define $(\rho_{\lambda_v},V_{\lambda_v}) $
 to be the unique irreducible $\C$-linear representation of $\GSp_{2n} $
 whose restriction to $\SP_{2n}$ has weight $\lambda_v$ and whose central character is $\omega^{-1}_{m_{1, v}+\cdots+ m_{n, v}}$ in the notation of (\ref{Subsubsection: Omega}). This defines a representation $(\rho_{\boldsymbol\lambda},V_{\boldsymbol\lambda})$ of $\GSp_{2n}(F)$ according to (\ref{eq:rho decomposes}), which clearly descends to $F^c$.
 \begin{prop}
 The representation $(\rho_{\boldsymbol\lambda},V_{\boldsymbol\lambda})$  descends to a $\Q (\boldsymbol\lambda) $-linear representation of $\GSp_{2n} (F)$, where $\Q (\boldsymbol\lambda) $
 is the fixed field of $$\set {\sigma\in\Aut(\C/\Q)\,:\,\lambda_{\sigma\cdot v} =\lambda_{v}\,\,\forall v|\infty}.$$
 \end{prop}
 \begin{proof}
 The proof of \cite[Proposition I.3]{waldspurger1985quelques} applies unchanged.
 \end{proof}
 \subsubsection{}
  In particular, for each such $\boldsymbol\lambda$, we obtain a $\Q(\boldsymbol \lambda)$-local system $\mathcal V_{\boldsymbol \lambda}$ on $S(\boldgsp_{2n})$ such that $\mathcal V_{\boldsymbol \lambda,\C}$ arises from the tuple of representations $(\rho_{\lambda_v},V_{\lambda_v})$ of $\GSp_{2n}(F_v)$ according to (\ref{eq:rho decomposes}).
 \subsection{The case $G = \GL_2$}
\subsubsection{}\label{subsubsec: coh for gl2}
We recall some basic results on the cohomology of $S(\boldsymbol G)$ in the simplest case,  $G = GL_{2}  = \GSp_{2}.$ 
For a tuple of integers $\boldsymbol m = (m_v)_{v|\infty}$ with $m_v\geq 2$, define $\Q(\boldsymbol m)$ to be the fixed field of  \begin{equation}
    \set {\sigma\in\Aut(\C/\Q)\,:\,m_{\sigma\cdot v} = m_v\,\,\forall v|\infty}.
\end{equation}
We then obtain from \S\ref{subsection: symplectic Shimura}
a $\Q (\boldsymbol m) $-local system $\mathcal V_{\boldsymbol m- 2}$ on $S (\boldgl_2) $, where $\boldsymbol m -2 = (m_v -2)_{v |\infty}$.
 \subsubsection{}Let $(p  (+), q  (+))= (1,0) $ and $(p  (-), q  (-)) = (0,1) $, and define $(\boldsymbol p  (\boldsymbol\epsilon),\boldsymbol q  (\boldsymbol\epsilon)) $ to be the plectic Hodge type $(p_v(\epsilon_v), q_v(\epsilon_v))_{v|\infty}$, for any choice of signs $\boldsymbol\epsilon = (\epsilon_v)_{v |\infty} $. Let  $\boldsymbol\chi_{\boldsymbol\epsilon\boldsymbol m}$ be the  character of $\boldsymbol K_1$ from (\ref{Sub sub: weights for unitary group}). 
  Then we have:
\begin{equation}\label{eq: To get realization on Jill to}
\begin{split}
\dim_\C\Hom_{\boldsymbol K_1} \left(
    \boldsymbol\chi_{-\boldsymbol\epsilon\boldsymbol m},\wedge ^ {\boldsymbol p (\boldsymbol\epsilon),\boldsymbol q (\boldsymbol\epsilon)}\mathfrak p ^\ast_{\GL_2}\otimes V_{\boldsymbol m - 2,\C}\right) = 1,\\
    \dim_\C\Hom_{\boldsymbol K_1}\left (
    \boldsymbol\chi_{-\boldsymbol\epsilon\boldsymbol m}^\check,\wedge^{1-\boldsymbol p (\boldsymbol\epsilon), 1 -\boldsymbol q (\boldsymbol\epsilon)}\mathfrak p_{\GL_2}^\ast\otimes V_{\boldsymbol m - 2,\C}^\vee\right) = 1.
    \end{split}
\end{equation}
Let $\pi$ be a cuspidal automorphic representation of $\GL_2(\A_F)$ of weight $\boldsymbol m$, whose central character has infinity type $\omega_{\boldsymbol m}$. Then combining (\ref{eqn:Realization cohomology definition}) with (\ref{eq: To get realization on Jill to}) yields  maps,   well-defined up to scalars:
\begin{equation}\label{EQ: final high epsilon to cohomology map 4G prime}
\begin{split}
\CL_{\boldsymbol\epsilon} : (\pi\otimes\boldsymbol\chi_{-\boldsymbol\epsilon\boldsymbol m}) ^ {\boldsymbol K_1}&\to H_{(2)} ^ {\boldsymbol p (\boldsymbol\epsilon),\boldsymbol q (\boldsymbol\epsilon)} (S (\boldgl_2),\mathcal V_{\boldsymbol m -2,\C})[\pi_f]\\
\CL_{\boldsymbol\epsilon}': (\pi ^\vee\otimes\boldsymbol\chi_{-\boldsymbol\epsilon\boldsymbol m} ^\vee) ^ {\boldsymbol K_1}&\to H_{(2)} ^ {1 -\boldsymbol p (\boldsymbol\epsilon), 1 -\boldsymbol q (\boldsymbol\epsilon)}(S (\boldgl_2),\mathcal V_{\boldsymbol m -2,\C} ^\vee)[\pi_f ^\vee]
\end{split}
\end{equation}
Here $(1 -\boldsymbol p (\boldsymbol\epsilon), 1 -\boldsymbol q (\boldsymbol\epsilon)) = (\boldsymbol p(-\boldsymbol \epsilon), \boldsymbol q(-\boldsymbol \epsilon))$
is the plectic Hodge
type $(1 - p (\epsilon_v), 1 - q (\epsilon_v))_{v|\infty} $.
The following is well-known:
\begin{prop}\label{prop: comparison isomorphism for G'}
For each $\boldsymbol\epsilon $, the maps in (\ref{EQ: final high epsilon to cohomology map 4G prime}) are isomorphisms,
 and $$H_{(2)}^{\boldsymbol p,\boldsymbol q}(S(\boldgl_2), \mathcal V_{\boldsymbol m-2, \C})[\pi_f] =H_{(2)}^{\boldsymbol p,\boldsymbol q}(S(\boldgl_2), \mathcal V^\vee_{\boldsymbol m-2, \C})[\pi^\vee_f]= 0$$ if $(\boldsymbol p,\boldsymbol q)$ is not of the form $(\boldsymbol p(\boldsymbol \epsilon),\boldsymbol q(\boldsymbol \epsilon))$ for some $\boldsymbol\epsilon $. Moreover, if $\pi_f$ is defined over $E\supset \Q(\boldsymbol m) $, then there are $\GL_2(\A_{F,f})$-equivariant isomorphisms
$$H^\ast_{!}( S({\boldgl_2}), \mathcal V_{\boldsymbol m-2,E}) [\pi_f]\otimes_E\C \simeq H ^\ast_{(2)} (S({\boldgl_2}), \mathcal V_{\boldsymbol m,\C}) [\pi_f] $$
and $$H ^\ast_c (S({\boldgl_2}), \mathcal V_{\boldsymbol m,E}) [\pi_f] \simeq H^\ast (S(\boldgl_2), \mathcal V_{\boldsymbol m,E})[\pi_f]\simeq H^\ast_! (S(\boldsymbol G), \mathcal V_{\boldsymbol m}) [\pi_f], $$
and similarly for $\mathcal V_{\boldsymbol m, E}^\vee$ and $\pi_f^\vee$.
\end{prop} \qed
\section{Similitude theta lifting}\label{section: Similitude theta lifting}
\subsection{Local Weil representation and local theta lift}
\subsubsection{}\label{subsubsection:local notation}
Let $\epsilon =\pm 1, $ and let $V$, $W$ be vector spaces over  a field $k$ equipped with nondegenerate $\epsilon$-symmetric and $(-\epsilon)$-symmetric pairings, respectively. We assume $\dim W = 2n $ and $\dim V = 2m $ are even, and that $W$ is equipped with a complete polarization \begin{equation}
    W = W_1\oplus W_2,\;\; W_2 = W_1 ^\ast.
\end{equation} \emph{For simplicity, assume as well that the discriminant character of $V $ is trivial (as will be the case in our applications).}
Let  $G_1 = G_1 (V) $, $G=G (V) $ be the connected isometry and similitude groups, respectively, of $W $, and likewise $H_1 = H _1(W) $ and $H = H (W) $ . Let $P=P (W_1) \subset H(W)$ be the parabolic subgroup stabilizing $W_1, $ $P_1$ its intersection with $H_1$, and $N\subset P_1$ its unipotent radical. 
Also set
\begin{equation}
    R_0 =\set {(h, g)\in H\times G\,:\,\nu_H(h) = \nu_G(g)},
\end{equation}
where $\nu_G: G\to \mathbb G_m$ and $\nu_H: H \to \mathbb G_m$ are the similitude characters.

\subsubsection{}
Assume that $k $ is a local field. Then, for any nontrivial additive character $\psi_k $
of $k $,
 the Weil representation $\omega=\omega_{W,V,\psi_k}$ of $H_1 (k)\times G_1 (k) $ is realized on the Schwartz space $\mathcal S_{k}(W_2\otimes V) $; in this model,  the action of the parabolic $P_1\times G_1\subset H_1\times G_1 $
stabilizing $W_1\times V$ is described as follows.
\begin{equation}
    \begin{cases}
    \omega(1,g)\xphi (x) = \xphi (g^{-1} x), & g\in G_1(k), \\
    \omega(n,1)\xphi(x)=\psi\left (\frac {1} {2}\langle n(x), x\rangle\right)\cdot\xphi (x), & n\in N (k)\subset\Hom (W_2, W_1),\\
    \omega (h (a),1)\xphi(x) =|\det (a)|^m \xphi(a^t x),&a\in \GL(W_1) (k)\subset P_1 (k),\\
    \end{cases}
\end{equation}
where  $GL(W_1)$ is viewed as the Levi factor of $ P_1 $
by the standard embedding
\begin{equation}
    a\mapsto h (a)=\begin {pmatrix} a & 0\\0 & a^{-t}\end{pmatrix}\in P_1.
\end{equation} 
 Following the convention of \cite{roberts1996theta}, $\omega$ extends naturally to $R_0(k)$  by defining
 \begin{equation}\label{Extend two similar tunes}
     \omega \left(\begin{pmatrix} 1& 0 \\ 0 & \nu_G(g) \end{pmatrix}, g\right)\xphi (x) = |\nu_G(g)| ^ {-mn/2}\xphi(g^{-1} x)
 \end{equation}
 for all $g\in G(k)$.
 Note that $\omega$ is trivial on the center $\set{(\lambda,\lambda)}\subset R_0 $.
 \subsubsection{}\label{subsubsec: Fourier}
 Suppose that $V=V_1\oplus V_2 $ is also split; then the preceding construction also defines an action of $R_0(k) $ on $\mathcal S_{k} (W\otimes V_2)$ by interchanging the roles of $V$ and $W$. These two representations are isomorphic via the partial Fourier transform. 
 More precisely, consider the map $\mathcal F:\mathcal S_{k} (W_2\otimes V)\to\mathcal S_{k}(W\otimes V_2) $
 defined by
 \begin{equation}\label{eqn: Fourier_transform}\xphi\mapsto\widehat\xphi,\;\;\widehat\xphi (x_1, x_2) = \int_{(W_2\otimes V_1)(k)}\xphi (z, x_2)\psi (\langle z, x_1\rangle) \d z,
 \end{equation}
 where $x_1\in W_1\otimes V_2,$ $x_2\in W_2\otimes V_2$, and $\d z$ is the self-dual Haar measure with respect to $\psi_k$. Then it is well-known that $\mathcal F $ intertwines the actions of $H_1(k)\times G_1(k)$ on both sides, and it is immediate to check that it intertwines the actions of $$\left (\begin{pmatrix}1 & 0\\0 &\lambda\end {pmatrix},\begin{pmatrix}1 & 0\\0 &\lambda\end{pmatrix}\right)\in R_0(k)$$ according to the definition (\ref{Extend two similar tunes}); so $\mathcal F$ is equivariant for all of $R_0(k)$. 
 \subsubsection{}\label{subsubsection: local theta lift}
  If $\pi $ is an irreducible admissible representation of $H (k) $, then the local theta lift $\Theta (\pi) = \Theta_{W,V}(\pi)$
  is the largest  semisimple  representation of $G(k) $ such that there is a surjection $$\omega_{W,V, \psi_k}\twoheadrightarrow\pi^\vee\boxtimes\Theta (\pi) $$
  of admissible $R_0(k) $-representations. Symmetrically, if $\sigma $ is an irreducible admissible representation of $G (k) $, then the local theta lift $\Theta (\sigma) = \Theta_{V,W}(\sigma)$ is the largest semisimple representation of $H (k) $ admitting a surjection $\omega_{W,V, \psi_k}\twoheadrightarrow\Theta (\sigma)\boxtimes\sigma^\vee $.
  The theta lift does not depend on $\psi_k $ by \cite[Proposition 1.9]{roberts2001global}. 
  \subsubsection{}\label{subsubsec: full sim gp}
  We remark that the  Weil representation extends naturally to the full similitude groups of $V $ and $W $, not just the neutral connected components, and so a theta lift $\Theta(\pi)$ or $\Theta(\sigma)$ may be viewed as a representation of the full similitude group of $V$ or $W$. The drawback of working with neutral connected components of  similitude groups is that we no longer have Howe duality, and in particular the local theta lift may be reducible.
   However, using connected similitude groups is more convenient for our global calculations. 

 \subsection{Global Weil representation and global theta lifts}
 \subsubsection{}
 Now turning to the global situation, assume $k = F $ in (\ref{subsubsection:local notation}).
Also suppose given, for almost every place $v $, self-dual lattices $\mathcal W_v\subset W\otimes F_v$ and $\mathcal V_v\subset V\otimes F_v$, such that $\mathcal W_v$ is compatible with the polarization $W = W_1\oplus W_2$ in the sense that: $$\mathcal W_v\cap (W_1 \otimes F_v) \oplus \mathcal W_v\cap (W_2 \otimes F_v ) = \mathcal W_v.$$ The adelic Schwartz space $\mathcal S_{\A_F}(W_2\otimes V)$ is the restricted tensor product of the local Schwartz spaces $\mathcal S_{F_v}(W_2\otimes V)$ with respect to the indicator function of $(\mathcal W_v\cap (W_2\otimes F_v))\otimes \mathcal V_v.$  The global Weil representation of $R_0({\A_F})$, realized on $\mathcal S_{\A_F}(W_2\otimes V)$, is defined as the restricted tensor product of the local Weil representations (using the characters $\psi_{F_v}$ determined by the fixed global character $\psi$).  Recall the automorphic realization of $\omega$, given by the theta kernel:
 \begin{equation}
     \theta (h,g;\xphi) =\sum_{x\in W_2(F)\otimes V(F)}\omega(h,g)\xphi (x),\;\; \;(h,g)\in R_0 ({\A_F}), \;\;\xphi\in\mathcal S_{\A_F}(W_2\otimes V).  
 \end{equation}
 If $V$ is also split, then we again have the alternate model $\mathcal S_{\A_F}(W\otimes V_2)$, related to $\mathcal S_{\A_F}(W_2\otimes V)$ by the adelic partial Fourier transform.  Note that $$\theta(h,g; \xphi) = \theta(h,g; \widehat\xphi) = \sum_{x\in W\otimes V_2} \omega(h,g)\widehat\xphi (x)$$ by Poisson summation.  
 
\subsubsection{} \label{subsubsec: global theta lived}Let $f\in\mathcal A_0 (H(\A_F)) $ be an automorphic cusp form and choose any $\xphi\in\mathcal S_{\A_F} (W_2\otimes  V)$. Then, fixing a Haar measure $\d h_1 $
 on $H_1 (\A_F) $, the similitude theta lift $\theta_\xphi(f)$  to $G$ is the automorphic function
 \begin{equation}\label{eq: glopbal theta lift def}
     g\mapsto \int_{[H_1]} \theta(h_1h_0,g; \xphi) {f (h_1h_0)}\d h_1, \;\;\; g\in G({\A_F}),
 \end{equation}
 where $h_0\in H({\A_F})$ is any element such that $\nu_H(h_0) = \nu_G(g).$
 
  Likewise, if $f\in\mathcal A_0 (G (\A_F)) $
 is an automorphic cusp form and $\d g_1$ is a Haar measure on $G_1 (\A) $, then the similitude theta lift $\theta_\xphi(f)$ to $H$ is the automorphic function $$h\mapsto\integral_{[G_1]}\theta (h, g_1g_0;\xphi) f (g_1g_0)\d g_1,\;\;\; h\in H (\A_F), $$
 where $g_0\in G (\A_F) $ is any element such that $\nu_G (g_0) =\nu_H (h) $.
 
 If $\pi$ is a cuspidal automorphic representation of $H(\A_F)$, then the similitude theta lift $\Theta(\pi) =\Theta_{W, V}(\pi)$ is the subspace of $\mathcal A(G(\A_F))$ spanned by the theta lifts $\theta_\xphi(f)$ for $f\in \pi$ and $\xphi \in \mathcal S_{\A_F} (W_2\otimes V)$;  if $\pi$ is a cuspidal automorphic representation of $G(\A_F)$, we similarly define $\Theta(\pi) = \Theta_{V,W}(\pi)$ to be the subspace of $\mathcal A(H(\A_F))$ spanned by the theta lifts $\theta_\xphi(f)$ for $f\in \pi$ and $\xphi \in \mathcal S_{\A_F} (W_2\otimes V)$.
 A key property of the global theta lift is its compatibility with the local theta lift. Although this is well-known, we include a proof for the reader's convenience.
 \begin{prop}
 \label{prop: local global compatibility}
 Let $\pi $
 be a cuspidal automorphic representation of either $G (\A_F) $
 or $H (\A_F) $, and suppose that $\Theta (\pi) $
 lies in the $L ^ 2 $ subspace. Then for any automorphic representation $\sigma ={\otimes_v^\prime}\sigma_v \subset\Theta (\pi) $, $\sigma_v $ is a constituent of $\Theta (\pi_v) $
 for all $v $.
 \end{prop}
 \begin{proof}
 Without loss of generality, suppose $\pi $
 is a representation of $G (\A_F) $. Consider the map of $R_0(\A_F)$-representations:
 \begin{align*}\mathcal S_{\A_F} (W_2\otimes V)\otimes\pi\otimes\sigma^\vee&\to \C \\
 \xphi\otimes f\otimes f'&\mapsto \integral_{[Z_H\backslash H]} \theta_\XP (f) (h) f' (h)\d h.\end{align*}
 This map is well-defined and nontrivial by assumption. By duality, it also gives a nontrivial map $$\mathcal S_{\A_F} (W_2\otimes V)\twoheadrightarrow\pi ^\check\boxtimes\sigma, $$
 which is evidently a restricted tensor product. This implies the proposition.
 \end{proof}
 \subsubsection{}
 The theta lift defined in (\ref{subsubsec: global theta lived}) generalizes readily to vector-valued automorphic forms. Suppose $K\subset G ( F\otimes\otimes\R) $
 and $L\subset H (F\otimes\R) $
 are  subgroups which are compact modulo center, and let $$(L\times K)_0 \coloneqq (L\times K)\cap R_0(F\otimes \R).$$ Suppose given finite-dimensional representations $\sigma$ and $\tau$ of $L$ and $K$,
 and let $f \in (\mathcal A_0(H(\A_F))\otimes \sigma)^L$ be a vector-valued automorphic form. Then for a vector-valued Schwartz function
 $$\phi\in \left(\mathcal S_{F\otimes \R}(W_2\otimes V) \otimes \sigma^\vee \otimes \tau\right)^{(L\times K)_0},$$
 and a Schwartz function $$\xphi_f\in\mathcal S_{\A_{F, f}} (W_2\otimes V)\coloneqq \otimes_{v\nmid \infty}^\prime\mathcal S_{F_v} (W_2\otimes V), $$
we may define 
$$\theta_{\xphi_f\otimes\phi} (f) \in (\mathcal A (G (\A_F))\otimes\tau) ^ K $$
by the same formula (\ref{eq: glopbal theta lift def}) as for the scalar-valued theta lift. The vector-valued theta lift from $G$ to $H$ is defined in the same way.

\subsection{Spherical theta correspondence for similitudes}
 \subsubsection{}
 We shall require an explicit description of the spherical similitude theta correspondence in certain cases. Continuing the notation of (\ref{subsubsection:local notation}), assume $k$ is a nonarchimedean local field, that $\psi_k$ is unramified, and  that $V = V_1\oplus V_2 $
 is a split orthogonal space (so that $\epsilon = +$). For this subsection only, for the purposes of clearer comparison with the literature, we let $G'_1 $
 and $G' $
 denote the \textbf{full} isometry and similitude groups of $V $, so that $G'_1$ and $G'$ are disconnected;  likewise for $$R_0' \coloneqq \set{(h, g)\in H\times G'\,:\, \nu_{H} (h) =\nu_{G'}(g)}.
 $$ The Weil representation of $R_0(k)$ extends  naturally to $R_0'(k)$. We assume the additive character $\psi_k$ of $k$ used to define $\omega$ is unramified.
 
 Now choose bases $\set {e_1,\cdots, e_m} $ and $\set {f_1,\cdots, f_n} $
 of $V_1 $ and $W_1, $
 respectively, and let $\set {e_1 ^\ast,\cdots, e_m ^\ast} $ and $\set {f_1 ^\ast,\cdots, f_n ^\ast} $
 be the dual bases of $V_2 $
 and $W_2. $
 Let $T_{G_1}\subset\GL (V_1)\subset G_1 $
 and $T_{H_1}\subset\GL (W_1)\subset H_1 $
 be the standard diagonal tori; then we choose the maximal tori for $G $, $H$, and $R_0$ given (with respect to the bases $\set{e_1,\cdots, e_m, e_1^\ast, \cdots,e_m ^\ast} $
 and $\set {f_1,\cdots, f_n, f_1 ^\ast,\cdots, f_n ^\ast} $) by:
 \begin{equation}
     \begin{split}
         T_G&= T_{G_1}\times\mathbb G_m =\set {\operatorname{diag}(x_1,\cdots, x_m, \lambda x_1^{-1},\cdots, \lambda x_m^{-1})}\\
         T_H&= T_{H_1}\times\mathbb G_m =\set {\operatorname{diag}(y_1,\cdots,y_n,\kappa y_1 ^ {-1},\cdots,\kappa y_n ^ {-1})} \\
         T_{R_0}&= T_H\times_{\mathbb G_m} T_G \simeq T_{G_1}\times T_{H_1}\times\mathbb G_m 
     \end{split}
 \end{equation}
 \subsubsection{}\label{subsubsection: notation for unramified principal series}
 To fix notation, we recall the unramified principal series of $G $ and $H$. The unramified characters of $T_{G_1} (k) $ are parameterized by tuples $\chi_1 = (\alpha_1,\cdots,\alpha_m)\in (\C^\times)^m$, where $$\chi_1(\diagonal (x_1,\cdots, x_m, x_1^{-1}, \cdots, x_n ^ {-1})) =\prod_{i = 1}^m \alpha_i ^ {\ord x_i}.$$
 The unramified characters of $T_G (k)$
are parameterized by $\chi = (\alpha_1,\cdots,\alpha_m, s)\in(\C^\times) ^ { m +1} $, where $$\chi (\operatorname{diag}(x_1,\cdots, x_m,\lambda x_1 ^ {-1},\cdots,\lambda x_m ^ {-1})) =s ^ {\ord \lambda} \prod_{i = 1} ^ {m}\alpha_i ^ {\ord x_i}. $$
Similarly, the unramified characters of $T_{H_1} (k)$ (resp.  $T_H(k)$)
are parametrized by $\mu_1 = (\beta_1,\cdots, \beta_n)\in (\C^\times)^n$ (resp. $\mu = (\beta_1,\cdots, \beta_n, t)\in (\C^\times)^{ n+1}$), and the unramified characters of $T_{R_0} (k)$
are parameterized by $\eta = (\beta_1,\cdots,\beta_n,\alpha_1,\cdots,\alpha_m, u )\in (\C^\times) ^{ n+m+1} $. 
Note that the character $\mu\boxtimes\chi$ of $T_H(k)\times T_G(k)$ pulls back to the character $$\mu\cdot\chi\coloneqq (\beta_1,\cdots,\beta_n,\alpha_1,\cdots,\alpha_m, st)$$ of $T_{R_0}(k)$ under the inclusion $T_{R_0}\subset T_H\times T_G$.

For  Borel subgroups $ B_G=T_GN_G\subset G $
and $B_H = T_HN_H\subset H$, the (normalized) principal series representations $\Ind_{B_G(k)}^{G(k)} \chi$ and $\Ind_{B_H(k)}^{H(k)}\mu$ possess unique spherical subquotients denoted $\pi_\chi$ and $\sigma_\mu$, respectively; note $\pi_\chi$ and $\sigma_\mu$ depend only on the Weyl orbits of $\chi$ and $\mu$. 
Moreover $\sigma_\mu\boxtimes \pi_\chi |_{R_0}$ is the unique spherical subquotient of $\Ind_{B_G(k)\times B_H(k)\cap R_0(k)}^{R_0(k)}\mu\cdot \chi$. 


\begin{prop}\label{prop:spherical}
Suppose $m\leq n $, $\epsilon = +$, and that the residue field of $k$ has odd cardinality $q$. If $\pi_\chi$ is the spherical representation of $G(k)$ associated to $\chi = (\alpha_1,\cdots, \alpha_m, s),$ and if $\Ind_{G(k)}^{G'(k)} \pi_\chi$ is irreducible, then $\Theta(\pi_\chi)$ is the spherical representation $\sigma_\mu$ of $H(k)$ for $$\mu = (\alpha_1,\cdots, \alpha_m, q, q ^ 2,\cdots, q ^ {n - m},s q ^ {- (m^2 - m)/4 - (n^2 +n)/4 + nm/2}).$$
\end{prop}
\begin{proof}
Since $\Ind_{G(k)}^{G'(k)} \pi_\chi$ is irreducible, $\Theta(\pi_\chi)$ is nonzero, irreducible \cite{roberts1996theta}, and unramified \cite[Proposition 1.11]{roberts2001global}. Thus $\Theta(\pi_\chi) = \sigma_\mu$ for some $\mu$, and it remains to determine $\mu$. 

As in \cite[\S4]{rallis1982langlands}, let $\sigma = (\sigma_1,\cdots, \sigma_m) \in \C^m$, and consider for all $\Re(\sigma_i) \gg 0$ the family of integrals:
\begin{equation}
    I(\sigma, \xphi) = \int \xphi\left(\sum_{i = 1}^m a_{ii} f_i^\ast\otimes e_i  + \sum_{1 \leq i < j \leq m} z_{ij} f_i^\ast \otimes e_j\right) \prod_{i = 1}^m|a_{ii}|^{\sigma_i + i - r} \d^\times a_{ii}\prod_{i < j} \d z_{ij},
\end{equation}
where $\xphi\in \mathcal S_k(W_2\otimes V)$. 
\begin{claim}
If $\Re (\sigma_i)\gg 0 $ for all $i$, then there exist Borel subgroups $B_G = T_GN_G\subset G$ and $B_H = T_HN_H\subset H$ such that $N_G$ and $N_H$ act trivially on $V_1,$ $W_1$, respectively, and such that $$Z_\sigma(\xphi)(h, g) \coloneqq I(\sigma, \omega(h, g)\xphi)$$
defines an $R_0'(k)$-intertwining map from $\omega$ to the induced representation  \begin{equation*}\begin{split}I_\sigma&\coloneqq \Ind_{T_{R_0}(k)\cdot (N_H\times N_G)(k)}^{R_0'(k)}\eta(\sigma),\\
\eta(\sigma) &=  (q^{\sigma_1 + 1 - m}, 
\cdots, q ^ {\sigma_i + i - m},\cdots, q ^ {\sigma_m}, q, q ^ 2,\cdots, q ^ {n - m},\\
&q ^ {m -\sigma_1-1},\cdots, q ^ {m -\sigma_i -i},\cdots, q ^ {-\sigma_m}, q ^ {- (m ^ 2 - m)/4 - (n ^ 2+ n)/4+ nm/2})\in (\C ^\times) ^ {n + m +1}.
\end{split}\end{equation*}
\end{claim}
\begin{claimproof}
Note that \begin{equation*}\begin{split}I_\sigma|_{H_1\times G_1'} &\simeq \Ind_{T_{H_1}N_H(k)}^{H_1(k)} \mu_1(\sigma)\boxtimes \Ind_{T_{G_1}N_G(k)}^{G_1'(k)}
\chi_1(\sigma), \\
\mu_1 (\sigma)&= 
(q ^ {\sigma_1+1 - m},\cdots, q ^ {\sigma_i + i},\cdots, q ^ {\sigma_m}, q, q ^ 2,\cdots, q ^ {n - m}),\\
\chi_1(\sigma) &= (q ^ {m -\sigma_1-1},\cdots, q ^ {m -\sigma_i -i},\cdots, q ^ {-\sigma_m}).
\end{split}
\end{equation*}
Thus the fact that $Z_\sigma$ is an $H_1\times G_1' $-intertwining map, for choices of $N_H$ and $N_G$ as in the claim, is a restatement of \cite[Lemma 4.1]{rallis1982langlands}. (See p. 487-489 of \emph{loc. cit.} for the choices of $N_H$ and $N_G$.)
To see that $Z_\sigma$
is an $R_0' $-intertwining map, it therefore suffices to check that, if $$r_\lambda = \left(\begin {pmatrix} 1 &\\&\lambda\end {pmatrix},\begin {pmatrix} 1 &\\&\lambda\end {pmatrix}\right)\in R_0, $$
then $$I (\sigma,\omega (r_\lambda)\xphi) =|\lambda| ^{(m^2 - m)/4 + (n^2 + n)/4 - nm/2}\cdot |\operatorname{det}(r_\lambda, \operatorname{Lie} N_H \times N_G)|^{1/2} I(\sigma, \xphi).$$
By definition, $$I (\sigma, \omega(r_\lambda)\xphi) = |\lambda|^{-nm/2}.$$
On the other hand, one can calculate directly that the determinant factor on the right-hand side is $$|\lambda|^{-(m^2 - m)/4 - (n^2 + n)/4},$$ since $r_\lambda$ acts by the scalar $\lambda^{-1}$ on the rootspaces $\Hom(V_2, V_1) \cap \operatorname{Lie} N_G$ and $\Hom(W_2, W_1)\cap \operatorname{Lie} N_H$, and trivially on the rest of $\operatorname{Lie} N_G\times N_H$.
This concludes the proof of the claim. 
\end{claimproof}

Now choose a hyperspecial subgroup $K_{R_0'}$ of $R_0'(k)$ (arising from maximal self-dual lattices in $W$ and $V$), and let $\mathcal H$ be the Hecke algebra of $\C$-valued, $K_{R_0'}$-biinvariant functions on $R_0'$.
For all $\sigma$ as in the claim, the Hecke action on the unique spherical vector in $I_\sigma$  defines an algebra morphism $z_\sigma: \mathcal H \to \C$. It follows from \cite{rallis1982langlands} that the support of the $\mathcal H$-module $$\mathcal S_k(W_2\otimes V)^{K_{R_0'}}$$ is contained in the Zariski closure of the points $z_\sigma$ of $\Spec \mathcal H$. 
On the other hand, the Satake isomorphism identifies complex points of $\Spec \mathcal H$ with $R_0'$-Weyl orbits of parameters $\eta = (\beta_1, \cdots, \beta_m, \alpha_1,\cdots, \alpha_n, u)$ as above. 
By assumption, there is a surjection $$\mathcal S_k(W_2\otimes V)\twoheadrightarrow \pi_{\chi}^\vee\boxtimes\sigma_\mu,$$ and hence the character $\chi^{-1}\cdot \mu$ lies in the Zariski closure of the Weyl orbit of the parameters $\eta(\sigma)$ in the claim.
However, the $\mu$ listed in the proposition is the only one (up to $H$-Weyl action) satisfying this property.
\end{proof}
 \section{Yoshida lifts on $\GSp_4$}\label{sec: yoshida basics}
 \subsection{Some four-dimensional orthogonal spaces}
 \subsubsection{}\label{subsubsection: notation for orthogonal spaces}
 Let $B $ be a quaternion algebra, possibly split, over a field $k $. Then $B $ comes equipped with a norm $N: B \to k $
 and an involution $b\mapsto b ^\ast $
 such that $bb ^\ast = N (b) $
 for all $b\in B $.
 The $k $-orthogonal space $V_B $ associated to $B $
 is isomorphic to $B $
 as a vector space, with the inner product defined by
 \begin{equation}
     (b_1, b_2)\coloneqq\tr (b_1b_2 ^\ast) = b_1b_2 ^\ast + b_2b_1 ^\ast.
 \end{equation}
 When $B $ is split, we often drop the subscript and abbreviate $V = V_{M_2 (k)} $.
 \subsubsection{}
 One has a map of algebraic groups over $k $:
 \begin {equation}
 \boldsymbol p_Z: B ^\times\times B ^\times\to\operatorname{GO}(V_B)
 \end {equation}
 defined by $$\boldsymbol p_Z (b_1, b_2)\cdot x = b_1xb_2 ^\ast,\;\; x\in V_B. $$
 The kernel of $\boldsymbol p_Z $
 is the antidiagonally embedded $\mathbb G_m $, and $\boldsymbol p_Z $
 is a surjection onto the connected similitude group $\GSO (V_B) $.
 
 If $k $ is a local field, then irreducible admissible representations of $\GSO (V_B) (k)$
 are all of the form $\pi_1\boxtimes\pi_2, $
 where $\pi_i $
 are irreducible admissible representations of
 $B ^\times$
 of the same central character; if $k = F $
, the same is true of automorphic representations of $\GSO (V_B) (\A_F) $.

 \subsection{Elliptic endoscopic $L$-parameters}\label{subsection: elliptic endoscopic parameters}
 \subsubsection{}\label{subsubsection: notation for endoscopic lifts}
The unique elliptic endoscopic group of $\GSp_{4,F}$ is $\GSO(V)$, equipped with the $L$-embedding:
 \begin{equation}\label{eqn:L embedding}
     {}^L \GSO(V) = (\GL_2\times_{\mathbb G_m} \GL_2)(\C)\times\Gal(\overline{F}/F) \hookrightarrow \GSp_4(\C) \times\Gal(\overline{F}/F) = {}^L{ \GSp_4}.
 \end{equation}
 The Langlands functoriality principle for the map (\ref{eqn:L embedding}) then suggests that, to an automorphic representation $\pi = \pi_1\boxtimes\pi_2$ of $\GSO(V)(\A_F)$, one can associate an $L$-packet of automorphic representations $\Pi(\pi_1,\pi_2)$ of $\GSp_4(\A_F)$.
 These $L $-packets and their local analogues are constructed via similitude theta lifting in \cite{roberts2001global,weissauer2009endoscopy}.
  More precisely, for each place $v$ of $F$ and each irreducible admissible representation $\pi_{1,v} \boxtimes \pi_{2,v}$ of $\GSO(V) (F_v)$, one associates a local $L$-packet
 \begin{equation}
     \set{\Pi^+(\pi_{1,v},\pi_{2,v}), \Pi^-(\pi_{1,v}, \pi_{2,v})},
 \end{equation}
 where by convention $\Pi ^ - (\pi_{1,v},\pi_{2,v}) = 0 $
 unless both $\pi_{i, v} $ are discrete series. For all $v$, $\Pi^+(\pi_{1,v},\pi_{2,v})$ is the unique generic member of the $L$-packet, and is explicitly given by the (nonzero, irreducible) local similitude theta lift:
 \begin {equation}
 \Pi ^ + (\pi_{1, v},\pi_{2, v})\coloneqq\Theta_{V, W_4} (\pi_{1, v}\boxtimes\pi_{2, v}).
 \end {equation}
 If $\pi_{i, v} $
 are both discrete series, then they admit Jacquet-Langlands transfers $\pi_{i, v} ^ {B} $
 to $B ^\times $, where $B $
 is the non-split quaternion algebra over $F_v $. In this case, we have
 \begin{equation}
     \Pi ^ - (\pi_{1, v},\pi_{2, v})\coloneqq\Theta_{V_{B}, W_4} (\pi_{1, v} ^ B,\pi_{2, v} ^ B),
 \end{equation}
 a nonzero irreducible representation.
 We remark that the central character of $\Pi^\pm(\pi_{1,v},\pi_{2,v})$ is the common central character of $\pi_{i,v} $ (since the central character of the Weil representation is trivial). The $L $-packets associated to $\pi_v $ and $\pi'_v = \pi_{2,v} \boxtimes \pi_{1,v} $ coincide, but otherwise are all disjoint. Globally, given a cuspidal automorphic representation $ \pi_1 \boxtimes \pi_2 $ of $\GSO(V) (\A_F) $
 and a finite set $S $ of places where $\pi_i $ are both discrete series, we form the adelic representation
 \begin{equation}
     \Pi_S (\pi_1,\pi_2)\coloneqq\sideset{}{'}\bigotimes_{v\not\in S}\Pi ^ + (\pi_{1,v},\pi_{2,v})\otimes\bigotimes_{v\in S}\Pi ^ - (\pi_{1,v},\pi_{2,v}).
 \end{equation}
 \begin{thm}[Weissauer]\label{thm: global endoscopy statement}
  Let $ \pi_1 \boxtimes \pi_2$ be a cuspidal automorphic representation of $\GSO(V) (\A_F)$, where $\pi_1\not\cong \pi_2.$ Then the automorphic multiplicity of $\Pi_S (\pi_1,\pi_2) $ is given by:
  $$m_{\rm{disc}}(\Pi_S(\pi_1,\pi_2)) = m_{\rm {cusp}} (\Pi_S (\pi_1,\pi_2)) =\begin{cases}
  1, &\text {if } | S |\text { is even,}\\0, &\text {if } | S |\text { is odd.}
  \end{cases}$$
  The representations $\Pi_S (\pi_1,\pi_2) $
  constitute a full near equivalence class in the discrete spectrum of $\mathcal A_{(2)}(\GSp_4(\A_F))$, and are generic if and only if $S =\emptyset $. They are tempered and not CAP. Moreoever, if $|S|$ is even, $$\Pi_S(\pi_1,\pi_2)=\Theta_{V_B,W_4}(\pi_1^B\boxtimes \pi_2^B),$$ where $B$ is the unique $F$-quaternion algebra ramified at the set of primes $S$ and $\pi_i^B$ are the Jacquet-Langlands transfers of $\pi_i$ to $B^\times(\A_F)$.
  \end{thm}
  \begin{proof}
   This is a combination of \cite[Theorem 5.2]{weissauer2009endoscopy} (for the multiplicity formula) and \cite[Corollary 5.5]{weissauer2009endoscopy} (for the nonvanishing of the global theta lift); note that, given $\Theta_{V_B, W_4} (\pi_1 ^ B\boxtimes\pi_2 ^ B)\neq 0$, it is cuspidal if $\pi_1^B\neq \pi_2^B$ by \cite[Theorem 4.3]{weissauer2009endoscopy}, and hence abstractly isomorphic to $\Pi_S (\pi_1,\pi_2)$ by Proposition \ref{prop: local global compatibility}.
  \end{proof}
 
\subsection{Yoshida lifts in cohomology}
\subsubsection{}\label{subsubsec: notation for endoscopic lifts in coh}
Fix a tuple $\boldsymbol m = (m_v)_{v |\infty} $
 of  integers such that $m_v\geq 2$ for all $v$. Let $\pi_1,\pi_2 $ be cuspidal automorphic representations of $\GL_2 (\A_F) $
of weights $\boldsymbol m +2 = (m_v +2)_{v |\infty} $
and $\boldsymbol m $, respectively, with equal central characters of infinity type $\omega_{\boldsymbol m}$. 
For a set $S_f $ of \emph{finite} places of $F $
at which $\pi_i $ are both discrete series, set $$\Pi_{S_f}=\bigotimes'_{\substack{v\not\in S_f\\v\nmid\infty}} \Pi^+(\pi_{1,v},\pi_{2,v})\otimes \bigotimes_{v\in S_f}\Pi ^ - (\pi_{1, v},\pi_{2, v}).$$
We consider the local system $\mathcal V_{(\boldsymbol m -2, 0)}$ of $\Q(\boldsymbol m) $-vector spaces on $S (\boldgsp_4) $ according to the conventions of \S\ref{subsection: symplectic Shimura} (the field $\Q (\boldsymbol m) $ is defined in (\ref{subsubsec: coh for gl2})).

 For each $v|\infty,$ let $\tau^+_{m_v}$, resp. $\tau^-_{m_v}$, be the unique irreducible representation of $K_{2,v}$ of central character $\omega_{m_v}^ {-1}$ whose restriction to $U(2)\subset K_{2,v}$ has highest weight $(1,-m_v-1)$, resp. $(m_v + 1,-1)$. Similarly, let $\sigma^+_{m_v}$, resp. $\sigma^-_{m_v}$, be the unique irreducible 
representation of $K_v$ of central character $\omega_{m_v} ^ {-1}$ whose restriction to $U(2)$ has highest weight $(-3,-m_v-1)$, resp. $(m_v+1, 3)$. Note that $\tau^\pm_{m_v}$ are the duals of the minimal $K_v$-types of the representation $\Pi_v^+(\pi_{1,v},\pi_{2,v})$ of (\ref{subsubsection: notation for endoscopic lifts}), and $\sigma^\pm_{m_v}$ are the duals of the minimal $K_v$-types of $\Pi_v^-(\pi_{1,v},\pi_{2,v})$. 

For a subset $S_\infty\subset\set{v|\infty}$ and a collection of signs $\boldsymbol\epsilon=\set{\epsilon_v}_{v|\infty},$ define the $\boldsymbol K_2$-representation 
\begin{equation}
    \boldsymbol\tau_{\boldsymbol m,S_\infty}^{\boldsymbol\epsilon} \coloneqq \bigotimes_{v\in S_\infty} \sigma^{\epsilon_v}_{m_v}\otimes\bigotimes_{\substack{v\not\in S_\infty\\ v|\infty}} \tau^{\epsilon_v}_{m_v}.
\end{equation}
Thus $\boldsymbol \tau^{\boldsymbol \epsilon}_{\boldsymbol m, S_\infty}$ is a minimal $\boldsymbol K_2$-type of $\Pi_S(\pi_1,\pi_2)$, if $S_\infty$ is the set of archimedean places in $S$.
\subsubsection{}\label{subsubsec: defining class map etc}
Now let $(\boldsymbol p(\boldsymbol\epsilon, S_\infty), \boldsymbol q(\boldsymbol \epsilon, S_\infty))$ be the plectic  Hodge type determined by:
\begin{equation}
    (p_v(\boldsymbol\epsilon, S_\infty),  q_v(\boldsymbol\epsilon,S_\infty)) = \begin{cases}
    (3, 0), & \epsilon_v = +, v\in S_\infty,\\
    (2,1), & \epsilon_v = +,v\not\in S_\infty,\\
    (1,2), & \epsilon_v = -,v\not\in S_\infty,\\
    (0, 3), & \epsilon_v = -,v\in S_\infty.\end{cases}
\end{equation}
Thus $(\boldsymbol p (\boldsymbol\epsilon,\emptyset),\boldsymbol q (\boldsymbol\epsilon,\emptyset) = (\boldsymbol p (\boldsymbol\epsilon) +1,\boldsymbol q (\boldsymbol\epsilon) +1 ) $
in the notation of (\ref{subsubsec: coh for gl2}).
An easy calculation shows that \begin{equation}\label{eq:hom space k infty}\dim\Hom_{\boldsymbol K_2} \left(\boldsymbol \tau_{\boldsymbol m, S_\infty}^{\boldsymbol\epsilon}, V_{(\boldsymbol m -2, 0),\C}\otimes \wedge^{\boldsymbol p(\boldsymbol\epsilon,S_\infty),\boldsymbol q(\boldsymbol \epsilon, S_\infty)}\mathfrak p ^\ast_{\GSp_4}\right) = 1.\end{equation}
Hence, if $S = S_f\sqcup S_\infty$ is a finite set of places of $F$ with $|S|$ even, combining (\ref{eqn:Realization cohomology definition}) and (\ref{eq:hom space k infty}) yields a map (well-defined up to a scalar):
\begin{equation}
    \CL^{\boldsymbol \epsilon}_S: \left(\Pi_S(\pi_1,\pi_2) \otimes \boldsymbol\tau^{\boldsymbol\epsilon}_{\boldsymbol m,S_\infty}\right)^{\boldsymbol K_2} \to H^{\boldsymbol p(\boldsymbol \epsilon, S_\infty), \boldsymbol q(\boldsymbol \epsilon, S_\infty)}_{(2)}(S(\boldgsp_4), \mathcal V_{(\boldsymbol m -2, 0),\C})[\Pi_{S_f}].
\end{equation}


 \begin{prop}\label{prop: plectic Hodge types and class maps}
The map $\CL^{\boldsymbol \epsilon}_S$ is an isomorphism of $G(\A_{F,f})$-representations, and moreover
$$H_{(2)}^{\boldsymbol p,\boldsymbol q}(S(\boldgsp_4),\mathcal V_{(\boldsymbol m -2, 0), \C})[\Pi_{S_f}] = 0$$ if $(\boldsymbol p,\boldsymbol q)$ is not of the form $(\boldsymbol p(\boldsymbol\epsilon, S_\infty), \boldsymbol q(\boldsymbol \epsilon, S_\infty))$ for some $S_\infty$ such that $|S_f\cup S_\infty|$ is even.
\end{prop}
\begin{proof}
That $\CL^{\boldsymbol\epsilon}_S$ is an injection follows from \cite{vogan1984unitary} and the calculation of Casimir operators for $G $, cf. \cite{harris1992arithmetic}. The surjectivity and the vanishing of other plectic Hodge types follows from 
 (\ref{eqn:matsushima plectic}), Theorem \ref{thm: global endoscopy statement}, and the calculation of the nonvanishing $(\mathfrak g, K_2)$ cohomology groups:
\begin{align*}\dim H ^ {3,0} (\mathfrak g, K_2;\Pi_v^-(\pi_{1,v},\pi_{2,v}) \otimes V_{m_v,\C}) &= \dim H ^ {0,3} (\mathfrak g, K_2;\Pi_v^-(\pi_{1,v},\pi_{2,v}) \otimes V_{m_v,\C}) = 1,\\
\dim H ^ {2,1} (\mathfrak g, K_2;\Pi_v^+(\pi_{1,v},\pi_{2,v}) \otimes V_{m_v,\C}) &= \dim  H ^ {1,2} (\mathfrak g, K_2;\Pi_v^+(\pi_{1,v},\pi_{2,v}) \otimes V_{m_v,\C}) = 1.
\end{align*}
The dimensions of these cohomology groups are calculated in \cite{vogan1984unitary}; the result is also recalled in \cite{taylor1993thel}.
\end{proof}
\subsubsection{}
Finally, we relate the $\Pi_{S_f}$-isotypic parts of the $L^2$ and singular cohomology. 
\begin{prop}\label{Prop: comparison isomorphism's for G}
Assume $\Pi$ is defined over $E $, where $\Q (\boldsymbol m)\subset E\subset\C $.
Then there exist $\GSp_4(\A_{F,f})$-equivariant isomorphisms
$$H_{!}^\ast( S(\boldgsp_4), \mathcal V_{(\boldsymbol m -2, 0),E}) [\Pi_{S_f}]\otimes_E\C\simeq H_{(2)}^\ast ( S(\boldgsp_4), \mathcal V_{(\boldsymbol m -2, 0),\C}) [\Pi_{S_f}]  $$
and $$H_c^\ast( S(\boldgsp_4), \mathcal V_{(\boldsymbol m -2, 0),E}) [\Pi_{S_f}]\simeq H^\ast (S(\boldgsp_4), \mathcal V_{(\boldsymbol m -2, 0),E})[\Pi_{S_f}]\simeq H_!^\ast (S(\boldgsp_4), \mathcal V_{(\boldsymbol m -2, 0),E}) [\Pi_{S_f}]. $$
\end{prop}
\begin{proof}
By Theorem \ref{thm: global endoscopy statement}, $$H_{\operatorname{cusp}}( S(\boldgsp_4), \mathcal V_{(\boldsymbol m -2, 0)}) [\Pi_{S_f}]\simeq H_{(2)} ( S(\boldgsp_4), \mathcal V_{(\boldsymbol m -2, 0)}) [\Pi_{S_f}], $$
and the first statement follows by the discussion in (\ref{subsubsec: diagram section}). The second assertion is an immediate consequence of Lemma \ref{lemma: retard corner} (and Poincar\'e duality), since $\Pi_{S_f}$ is not Eisenstein.
\end{proof}


\section{Periods of Yoshida lifts}\label{Section:. See Yoshida}
\subsection{The period problem}
\subsubsection{}\label{subsubsec:measure on H}
Let $ \pi_1 \boxtimes \pi_2$ be a cuspidal automorphic representation of $\GSO(V) (\A_F) $, and let $\pi$ be an auxiliary cuspidal automorphic representation of $\GL_2(\A_F)$ such that $\pi ^\check$ and $\pi_i $
have the same central character. Consider the subgroup $$H= \GL_2\times_{\mathbb G_m}\GL_2\subset\GSp_4 $$
and the period integral $\mathcal P_{S,\pi_1,\pi_2,\pi}: \Pi_S(\pi_1,\pi_2)\otimes \pi \to \C$ defined by
\begin{equation}\label{eq:  definition of.}
   \mathcal P_{S,\pi_1,\pi_2,\pi}(\alpha,\beta) = \int_{[Z_H \backslash H]} \alpha(h, h')\cdot \beta(h)\d (h, h'),
\end{equation}
where $H (\A_F) \subset\GSp_4 (\A_F) $
is parameterized by pairs $(h, h ')\in \GL_2(\A_F)\times \GL_2(\A_F) $ such that $\det(h) = \det(h')$. When $\pi_1,$ $\pi_2,$ and $\pi$ are clear from context, we drop them from the notation $\mathcal P_{S,\pi_1,\pi_2,\pi}$. The goal of this section is to calculate $\mathcal P_{S,\pi_1,\pi_2,\pi} $
explicitly (Theorems \ref{thm: non-generic vanishing} and \ref{theorem: new global. Pairing}). The result is applied to the cohomology of Shimura varieties in the next section.
\subsubsection{}
Of course, we must specify a Haar measure on $[Z_H\backslash H]$ for (\ref{eq:  definition of.}) to be well-defined.
Let $\mathcal C =\A_F ^ {\times, 2} F ^\times\backslash\A_F ^\times $, and let $\d c $
be the Haar measure on $\mathcal C $
assigning volume $1 $ to the image of $\widehat {\O}_F $.
As the  measure on $[Z_H\backslash H]$, we take the measure induced by pullback from the surjection $[\SL_2]\times [\SL_2]\times\mathcal C \twoheadrightarrow [Z_H\backslash H]$. The Haar measure on $\SL_2 $
is  described in (\ref{subsubsec:measures}).

\subsubsection{}
Before we begin the calculation of (\ref{eq:  definition of.}), we explain the seesaw diagram that lies behind it:
\begin{center}
    \begin{tikzcd}
    \GSp_4 \arrow[d,dash]\arrow[dr,dash] & \GSO(V_B)\times_{\mathbb G_m} \GSO(V_B) \arrow[d,dash]\arrow[dl,dash] \\ H & \GSO(V_B)
    \end{tikzcd}
\end{center}
Here $B$ is the quaternion algebra ramified at $S$, the vertical lines are inclusions, and the diagonals are similitude dual pairs inside $\GSp_8$; the diagram corresponds to the two decompositions $$W_4\otimes V_B = W_2\otimes V_B\oplus W_2\otimes V_B $$
of $W_{16}$. 
Since $\Pi_{S} (\pi_1,\pi_2) $
is spanned by theta lifts $\theta_{\xphi} (f_1\otimes f_2) $
for $f_i\in\pi_i ^B$, we wish to apply the formal seesaw identity: \begin{equation}
    \langle \theta_\xphi(f_1\otimes f_2)|_H, \beta\otimes \mathbbm 1\rangle_H = \langle f_1\otimes f_2, \theta_\xphi(\beta\otimes \mathbbm 1)|_{\GSO(V_B)}\rangle_{\GSO(V_B)}, 
\end{equation}
Here $\beta\otimes \mathbbm 1$ is the automorphic form $(h,h')\mapsto \beta(h)$ on $H$.
Now, the theta lift from $H $ to $\GSO (V_B)\times_{\mathbb G_m}\GSO (V_B) $
is simply two copies of the theta lift from $\GL_2 $
to $\GSO (V_B) $; restriction to the diagonal amounts to multiplying the theta lifts of $\beta $
and $\mathbbm 1$ on $\GSO(V_B)$.
The theta lift of $\beta $ to $\GSO (V_B) $
will be a vector in $\pi^B\boxtimes \pi^B$, where $\pi^B$ is the Jacquet-Langlands transfer. However, the  theta lift of the constant function is formally divergent;
 to regularize it,  we need a certain second-term Siegel-Weil formula. Ignoring this technicality, the theta lift $\theta_\xphi(\beta\otimes \mathbbm 1) $ restricted to the diagonal $\GSO(V_B)$
should be the product of a vector in $\pi^B\boxtimes \pi^B$ and an Eisenstein series on $\GSO(V_B)$. Of course, the Eisenstein series can only exist when $B$ is split, so  (\ref{eq:  definition of.})
should vanish identically unless $S =\emptyset$. But when $S =\emptyset $, integrating $\theta_\xphi(\beta\otimes 1)$ against the form $f_1\otimes f_2$ gives a Rankin-Selberg integral that unfolds to an Euler product and ultimately an $L$-function. 

Thus to compute $\mathcal P_{S,\pi_1,\pi_2,\pi} $, we first must dispatch the trivial case $S\neq \emptyset$, and then study the theta lift of both cusp forms and constant functions from $\GL_2$
to $\GSO (V)$. This is the content of the next three subsections.
\subsection{Calculation of period integral: trivial case}
\subsubsection{}
The trivial case $S\neq\emptyset$ can be handled easily:
\begin{thm}\label{thm: non-generic vanishing}
If $S\neq\emptyset$, then $\mathcal P_{S} $ is identically zero.
\end{thm}
\begin{proof}
Let $B $ be the quaternion algebra over $F $ ramified exactly at $S $ (recall $| S | $ is even). By Theorem \ref{thm: global endoscopy statement}, it suffices to show the vanishing of all integrals of the form
 $$I (\xphi, g,f)=\int_{[Z_H\backslash H]} \theta_\xphi(g)(h,h') \cdot f(h) \d (h,h'),$$
 where $\xphi\in \mathcal S_\A (W_2\otimes B)$, $W_2\subset W$ a maximal isotropic subspace of the standard four-dimensional symplectic space, and $g\in \pi_1^B\boxtimes \pi_2^B$. 
  Let us fix a place $v$ at which $B $
  ramifies, and a Schwartz function $\xphi ^ {(v)}\in \mathcal S_{{\A_F} ^ {(v)}} (W_2\otimes B). $ Then, holding the other data $f,g$ fixed as well, consider
  the linear map $I_v: \mathcal S_{F_v} (W_2\otimes B)\to\C $ defined by \begin{equation}
      \xphi_v\mapsto I (\xphi ^ {v}\otimes\xphi_v, f,g).
  \end{equation}
Now $I_v$ clearly factors through the maximal quotient $Q$ of $\mathcal S_{F_v} (W_2\otimes B) = \mathcal S_{F_v} (B\oplus B) $ on which $\set {1}\times \SL_2(F_v)\subset H(F_v)\subset\GSp_4 (F_v)$ acts trivially. We claim this quotient is trivial. Indeed, the action of the  Borel subgroup of $\set {1}\times \SL_2(F_v) $ is explicitly described by:
\begin{equation}
    \begin{split}
        \omega \left(1\times \begin{pmatrix} 1 & n \\0 & 1\end{pmatrix},1\right) \xphi_v(b_1, b_2)=\psi\left (\frac {1} {2} n N(b_2)\right)\xphi_v (b_1, b_2)\\
        \omega\left (1\times\begin {pmatrix} a & 0\\ 0 &a ^ {-1} \end{pmatrix},1\right)\xphi_v  = |a | ^ 2\xphi_v (b_1, ab_2).
    \end{split}
\end{equation}
Since $B_v $ is anisotropic, it follows from the first equation that $\mathcal S_{F_v} (W_2\otimes B)\to Q $ factors through $\xphi_v\mapsto\xphi_v (b_1,0)$; then the second equation implies $Q =0.$ Therefore $I_v$ is identically zero for all choices of $(\xphi ^ v, f,g) $, and in particular (since the adelic Schwartz space is generated by factorizable Schwartz functions) all the period integrals $I(\xphi, f, g)$ vanish as well.
 \end{proof}
 \subsection{Lifts of cuspidal representations from $\GL_2$ to $\GSO (V) $}
 \subsubsection{}
 Since $V$ is split, the Weil representation for the pair $(W_2,V)$ has the alternate model given by the complete polarization $V = V_1\oplus V_2,$  where
 \begin{equation}\label{eq:isotropic of V}
   V_1 =   \begin{pmatrix}
        x & y\\0 & 0
      \end{pmatrix},\;\; V_2 =\begin{pmatrix}0 & 0\\z & w\end{pmatrix}.
  \end{equation} 
 \subsubsection{}
  Let $\pi $ be a cuspidal automorphic representation of $\GL_2 (\A_F) $. It is well-known that the theta lift $\Theta(\pi)\subset \mathcal A_0(\GSO(V)(\A_F))$ is isomorphic to the automorphic representation $\pi\boxtimes\pi$ of $\GSO(V)$. 
  To obtain our ultimate period formula, we will require the following calculation:
  \begin{lemma}\label{Lemma: Fourier coefficient of cuspidal left}
  Let $\xphi =\otimes_v\xphi\in\mathcal S_\A (\langle e_2\rangle\otimes W) $
  and $f =\otimes_vf_v\in\pi $ be factorizable vectors, and choose a factorization $$W_{\psi,f}(h) = \prod_v W_{f,v}(h_v),\;\; h = (h_v)\in \GL_2(\A_F)$$ of the global Whittaker function of $f$ (so that $W_f(h_v)(1) = 1$ for almost all $v$). Then the Whittaker coefficient of $\theta_\xphi (f) $ along the standard unipotent subgroup $N\times N\subset\boldsymbol p_Z (\GL_2\times \GL_2) $ is given by:
  $$\theta_{\xphi} (f) (g)_{N\times N,\psi ^ {-1}\times\psi ^ {-1}} =\prod_v\left(\integral_{\mathrm{SL}_2 (F_v)} W_{f,v} (h_vh_{c_v})\omega (h_vh_{c_v}, g)\widehat\xphi (1, 0, 0, -1)\d h_v\right),\;\; c = (c_v) =\det(g).$$
  \end{lemma}
  \begin{proof}
 We compute in two steps. First, for $(h,g)\in R_0(\A_F)$,
 \begin{align*}
     \theta (h,g;\xphi)_{N\times 1,\psi^{-1}\times 1}& = \integral_{[N]}\sum_{x\in W\otimes V_2}\omega(h,ng) \widehat\xphi (x)\psi (n)\d n\\
     & =\integral_{[\mathbb G_a]}\sum_{(z_1, w_1, z_2, w_2)}\psi (a (w_2z_1 - z_2w_1))\omega(h,g)\widehat\xphi (z_1, w_1, z_2, w_2)\psi (a)\d a\\
     &=  \sum_{\substack{(z_1, w_1, z_2, w_2)\\z_1w_2 - w_1z_2 = 1}}\omega(h,g)\widehat \xphi (x) \\
   &= \sum_{\gamma\in \SL_2(F)} \omega(\gamma h, g)\widehat \xphi(1,0,0,-1).
 \end{align*}
 Here $\d n$ is the Haar measure on $N$ such that $[N]$ has volume 1. 
 Now, using the identity $$\omega (n h, g)\widehat\xphi (1, 0, 0, -1)=\omega (h, \boldsymbol p_Z(1,n) g)\widehat\xphi (1, 0, 0, -1),\;\; (g, h)\in R_0,\;n\in N (\A), $$
 we obtain:
 \begin{align*}
     \theta_\xphi (f) (g)_{N\times N,\psi ^ {-1}\times\psi ^ {-1}} &=\integral_{[N]}\integral_{[\SL_2]} \theta (hh_c, \boldsymbol p_Z(1,n)g;\xphi)_{N\times 1,\psi ^ {-1}\times 1}\psi (n) f(hh_c)\d h\d n\\
     & =\integral_{[N]}\integral_{\SL_2(\A)}\omega (hh_c, \boldsymbol p_Z(1,n) g)\widehat\xphi(1,0,0,-1)\psi (n) f (hh_c)\d h\d n\\
     & =\integral_{[N]}\integral_{\SL_2(\A)}\omega ( hh_c, g)\widehat\xphi (1, 0, 0, -1)\psi (n) f (n^{-1}hh_c)\d n\d h\\
     &=\integral_{\SL_2 (\A)}\omega (hh_c,g)\widehat\xphi(1,0,0,-1)W_{\psi, f} (hh_c)\d h,
 \end{align*}
 which gives the lemma.
  \end{proof}
\subsection{A Siegel-Weil identity for $\GSO(V)$}
\subsubsection{Degenerate principal series for $\GSO (V) $}
The  maximal isotropic subspace $V_1\subset V$ of (\ref{eq:isotropic of V}) has stabilizer \begin{equation}
    P = \boldsymbol p_Z(B\times \GL_2)\subset \GSO(V).
\end{equation} 
Let $v$ be a place of $F $, and consider the (normalized) induced representation
 $$\boldsymbol I_v(s) = \Ind_{P(F_v)}^{\GSO(V)(F_v)} \delta_{P_1}^s.$$ 
We also consider the induced representations $I_v(s) = \Ind_{B(F_v)}^{\GL_2(F_v)} \delta_B^s$.
The representations $I_v(s)$ and $\boldsymbol I_v(s)$ are related by the following observation.
\begin{prop}\label{prop:GSO_induced}
The map $$M: \boldsymbol I_v(s) \to I_v(s)$$ defined by $$M(\phi)(g) = \phi(\boldsymbol p_Z(g, 1))$$ is a linear isomorphism and an intertwining map of $\GL_2(F_v)\times \GL_2(F_v)$ representations, if $\GL_2(F_v)\times \GL_2(F_v)$ acts on the left through the quotient $\GL_2(F_v)\times \GL_2(F_v) \twoheadrightarrow GSO(V)(F_v)$ and on the right through the quotient $\GL_2(F_v)\times \GL_2(F_v)  \twoheadrightarrow \GL_2(F_v) \times \set{1}.$
\end{prop}
Then by the well-known theory of principal series for $\GL_2, $ we deduce:
\begin{cor}\label{Cor:induced_properties}
For all places $v$, the representation $\boldsymbol I_v(1/2) $ has a unique irreducible subrepresentation, and the corresponding quotient is the trivial character of $\GSO (V) (F_v). $
\end{cor}
\subsubsection{}
Consider the map $$[\cdot]_v: S_{F_v}(\langle e_2\otimes\rangle V) \to \boldsymbol I_v(1/2)$$ 
defined by $$[\xphi](g) = \omega(h_{\nu(g)},g) \widehat \xphi(0).$$ A standard calculation shows that $[\cdot]$ is equivariant for the action of $R_0(F_v) \subset \GL_2(F_v)\times \GSO(V)(F_v)$ on both sides, where $R_0(F_v)$ acts on $\boldsymbol I_v(1/2)$ through the projection $R_0 \twoheadrightarrow \GSO(V)$. We may then extend $[\xphi]_v$ to a holomorphic section $[\xphi]_v(s)\in \boldsymbol I_v(s)$
by requiring the restriction of $[\xphi]$ to the maximal compact subgroup $K_0\subset \GSO(V)(F_v)$ to be independent of $s$.
\begin{lemma}\label{Lemma:surj_map_to_induced}
For any place $v$, the map $\xphi \mapsto [\xphi]$ identifies $\boldsymbol I_v(1/2)$ with the maximal quotient of $S_{F_v}(\langle e_2\rangle\otimes V)$ on which $\SL_2(F_v) \subset R_0$ acts trivially. 
\end{lemma}
\begin{proof}
By \cite[Theorem II.1.1]{rallis1984howe} (cf. \cite{kudla1990degenerate} in the Archimedian case), the map $\xphi\mapsto [\xphi]$  realizes its image as the maximal quotient of $S_{F_v}(\langle e_2\rangle\otimes V) $ on which $\SL_2(F_v) $
acts trivially. In light of Corollary  \ref{Cor:induced_properties} it suffices to show that there exists a nontrivial map $\mathcal S_{F_v} (\langle e_2\rangle\otimes V)\to\C $
which is equivariant for $\SL_2 (F_v)\times SO (V) (F_v)\subset R_0 (F_v). $
The $L ^ 2 $-norm $$\xphi\mapsto\int_{\langle e_2\rangle\otimes V} |\xphi(z)|^2\d z$$
is such a map (by the Plancherel identity). 
\end{proof}

\subsubsection{Eisenstein series on $GSO(V)$}
Let $\boldsymbol I (s) =\Ind_{P(\A_F)} ^ {\GSO (V) (\A_F)}\delta_{P} ^s $
be the global parabolic induction, and for holomorphic sections $\phi_s\in\boldsymbol I (s) $
consider the Eisenstein series:
\begin{equation}
    \boldsymbol E (g, s;\phi) =\sum_{\gamma\in P (F)\backslash \GSO (V) (F)}\phi_s(\gamma g),\;\;g\in \GSO (V) ({\A_F}),
\end{equation}
which converges for $\Re(s)\gg 0. $ We also consider $I(s) =\Ind_{B ({\A_F})} ^ {GL_2 ({\A_F})}\delta_B ^ s $ and, for holomorphic sections $\phi_s\in I (s) $, the corresponding family of Eisenstein series:
\begin{equation}
    E(g,s;\phi) = \sum_{\gamma\in B (F)\backslash \GL_2 (F)}\phi_s (\gamma g),\;\;g\in \GL_2({\A_F}).
\end{equation}
\begin{prop}\label{prop:GSO_Eisenstein_GL2}
Let $M =\otimes_v M_v: \boldsymbol I (s)\to I (s) $ be the global intertwining map. Then $$E(g_1;s,M(\phi)) =\boldsymbol E(\boldsymbol p_Z(g_1,g_2); s, \phi)$$ as functions on $\C\times GL_2({\A_F})\times GL_2({\A_F})$ for $\Re(s) \gg 0$ and holomorphic sections $\phi \in \boldsymbol I$.\qed
\end{prop}
By Proposition \ref{prop:GSO_Eisenstein_GL2} and the well-known theory of Eisenstein series for $\GL_2, $ $\boldsymbol E (g, s;\phi) $ has a meromorphic continuation to $s\in\C, $
with at most a simple pole at $s =\frac{1}{2}. $ 
Let
\begin{equation}
    [\cdot]_s:\mathcal S_{\A_F} (\langle e_2\rangle\otimes V)\to\boldsymbol I (s)
\end{equation}
be the tensor product of the local maps $[\cdot]_{v,s} $. For each $\xphi\in\mathcal S_{\A_F} (\langle e_2\rangle\otimes V) $, we consider the Laurent series expansion:
\begin{equation}
    \boldsymbol E (g, s;[\xphi]) = \frac{A_{-1}(g;\xphi)}{s - \frac {1} {2}} + A_0 (g;\xphi) + \cdots.
\end{equation}
\begin{lemma}\label{lem:A0_intertwining}
For each $\xphi $, $A_{-1} (g;\xphi) $
is a constant function of $g $. Moreover, the linear map
$$A_0:\mathcal S_{\A_F} (\langle e_2\rangle\otimes V)\to\mathcal A (\GSO (V)(\A_F))$$
is an $R_0(\A_F)$-intertwining operator modulo constant functions.
\end{lemma}
\begin{proof}
The first claim is immediate from Proposition \ref{prop:GSO_Eisenstein_GL2}.
 For the second, the proof of \cite[Proposition 6.4]{gan2014regularized} applies almost verbatim, taking into account Lemma \ref{Lemma:surj_map_to_induced}.
\end{proof}
\subsubsection{The spherical Eisenstein series}
Let $\phi ^ 0_s\in I (s) $ be the unique $\GL_2 (\widehat {\O}_F)\cdot\SO (2) $-spherical section such that $\phi ^ 0_s (1) = 1, $ and let \begin{equation}
    E_0 (g, s)\coloneqq E (g, s;\phi^0_s)
\end{equation}
be the resulting Eisenstein series on $\GL_2(\A_F)$. We record the following:
\begin{prop}\label{prop:constanttermcomputed}
The residue of $E_0(h,s)$ at $s= \frac {1} {2} $ is given by: $$\kappa=\frac{\pi^d\Res_{s=1}\zeta_F(s)}{2|D_F|^ {\frac {1} {2}}\zeta_F(2)}.$$
\end{prop}
\begin{proof}
Although this is standard, we give a sketch for the reader's convenience. In the Fourier expansion of $E_0(h,s)$, the non-constant Fourier coefficients are holomorphic. We therefore wish to calculate $$\Res_{s=\frac{1}{2}}\frac {1} {\vol([N])} \integral_{[N]} E_0(n, s)\d n, $$
where $\d n $ is the Haar measure on $N ({\A_F}) $ induced by the identification $N(\A_F)\simeq \A_F$ and (\ref{places etc}). Unfolding, we obtain (using the Bruhat decomposition of $\GL_2$):
\begin{align*}
    \frac {1} {\vol([N])}\integral_{[N]} E_0(n, s)\d n &= \frac {1} {\vol([N])}\integral_{[N]} \sum_{\gamma\in B(F)\backslash \GL_2(F)}\phi^0_s (\gamma n)\d n \\
    &=\frac {1} {\vol([N])}\integral_{[N]}\phi^0_s(n)\d n + \frac{1}{\vol([N])} \integral_{[N]}\sum_{a\in F} \phi^0_s(w_0an)\d n \\
    &= \frac {1} {\vol([N])}\integral_{[N]}\phi^0_s(n)\d n + \frac{1}{\vol([N])} \integral_{N({\A_F})} \phi^0_s(w_0n)\d n, \\
\end{align*}
where $$w_0=\begin{pmatrix} 0 & 1 \\ -1 & 0\end{pmatrix}$$ is the Weyl element.
The first term is holomorphic in $s$, so we may discard it and compute:
$$ \frac{1}{\vol([N])} \prod_v \integral_{N(F_v)}\phi^0_s(w_0n_v)\d n_v,$$ where $\d n_v$ is the standard Haar measure assigning volume one to $\O_v$. By the Gindikin-Karpelevich formula (e.g. \cite{fleig_gustafsson_kleinschmidt_persson_2018}), this product is $$\frac{1}{\vol([N])} \left(\sqrt{\pi}\frac{\Gamma(s-1/2)}{\Gamma(s)}\right)^d \prod_{v\nmid \infty} \frac{1 - q_v^{-2s-1}}{1 - q_v ^{-2s}} = \frac{1}{\vol([N])} \left(\sqrt{\pi}\frac{\Gamma(s)}{\Gamma(s+1/2)}\right)^d \frac{\zeta_F(2s)}{\zeta_F(2s+1)}. $$ Taking residue at $s =\frac {1} {2}, $
we obtain
$$\kappa=\frac {\pi ^ d\Res_{s = 1}\zeta_F (s)} {2\vol([N]) \zeta_F(2)}.$$
Finally, we may calculate $$\vol([N]) = \text{Vol} (F\backslash {\A_F}/\widehat O_F) = \text{Vol}(\R^d/\O_F) =  |D_F |^ {\frac {1} {2}} $$ by strong approximation.
\end{proof}
\subsubsection{Regularized theta integrals}
We now recall the regularization, due to Kudla and Rallis \cite{kudla1994regularized}, of the (non-convergent) theta integral 
$$g\mapsto \int_{[\SL_2]} \theta( h_1h_{\nu(g)},g; \xphi) \d h_1,\;\; g\in \GSO(V)({\A_F}),$$
where $\xphi\in \mathcal S_{\A_F}(\langle e_2\rangle \otimes V)$.
The first step of the regularization is to define a certain element $z$ of the universal enveloping algebra of $\mathfrak {sl}_2$; for the precise definition, see \cite[\S5.1]{kudla1994regularized}.
 Kudla-Rallis' regularized theta integral (adapted to the similitude case) is then:
\begin{equation}
    I(g,s;\xphi) \coloneqq \frac{1}{\kappa\cdot(4s^2 - 1)} \int_{[\SL_2]} \theta(g, h_1h_{\nu(g)}; \omega(z)\xphi) E_0(h_1,s) \d h_1,\;\; g\in \GSO(V)({\A_F}).
\end{equation}
(The factor of $4s^2 - 1$ is designed to cancel the effect of $\omega(z)$, cf. \cite[\S5.5]{kudla1994regularized}. Our normalization of $s$ differs from \emph{loc. cit.} by a factor of two.) 
The regularized integral $I (g, s;\xphi) $ is a meromorphic function of $s$ whose poles coincide with the poles of $E_0(h_1,s)$. The Laurent expansion about $s = \frac {1} {2} $ has the form:
\begin{equation}
    I (g, s;\xphi) =\frac {B_{-2} (g,\xphi)} {\left (s - \frac {1} {2}\right) ^ 2} +\frac{B_{-1}(g,\xphi)}{s - \frac {1} {2}}+ B_0 (g,\xphi) +\cdots
\end{equation}
By definition, the linear maps \begin{equation}
    B_d: \mathcal S_{\A_F} (\langle e_2 \rangle\otimes V)\to\mathcal A (GSO (V))
\end{equation}
are $\GSO (V) ({\A_F}) $-equivariant, where $g\in \GSO(V)({\A_F})$ acts on the left by $\xphi \mapsto \omega( h_{\nu(g)},g)\xphi.$
\begin{thm}[Gan-Qiu-Takeda]\label{thm: second term identity}
For all $\xphi\in\mathcal S_{\A_F}(\langle e_2\rangle\otimes V) $ and all $g\in \GSO(V)({\A_F})$,
we have:
\begin{align*}B_{-2}(g, \xphi) &= \vol([\SL_2])A_{-1}(g,\xphi) \\
B_{-1}(g,\xphi) &= \vol([\SL_2]) A_0(g,\xphi) +C(\nu(g),\xphi),
\end{align*}
where the volume of $[\SL_2]$ is taken with respect to $\d h_1.$
\end{thm}
\begin{proof}
Fix $\xphi$; it follows immediately from \cite{gan2014regularized} that the identities hold for all $g\in \SO (V) ({\A_F}) $, and for some $C(1,\xphi)$. On the other hand, the map $\xphi \mapsto B_{-2}(\cdot, \xphi)$ is $\SL_2$-invariant and $\GSO(V)({\A_F})$-invariant, in particular $R_0({\A_F})$-invariant; thus it factors through the maximal $\GSO(V)({\A_F})$-quotient of $\boldsymbol I(1/2)$ on which $\SO(V)({\A_F})$ acts trivially, i.e. the trivial character of $\GSO(V)({\A_F})$, and $B_{-2}(\cdot,\xphi)$ is constant. Since $A_{-1}(\cdot, \xphi)$ is also constant by Lemma \ref{lem:A0_intertwining}, the first identity follows.

For the second identity, for all $g = g_1g_0\in \GSO (V) ({\A_F}) $ with $g_0\in 
\SO(V)({\A_F})$, Lemma \ref{lem:A0_intertwining} implies:
\begin{align*}
    A_0(g_1, \omega(h_{\nu(g_0)},g_0)\xphi) &= A_0(g_1g_0,\xphi) + C(g_0,\xphi);
\end{align*}
since this applies to all decompositions $g = g_1g_0$, $C(g_0,\xphi)$ depends only on $\nu(g_0) = \nu(g)$. Combining this with the identity for isometry groups,
\begin{align*}
    B_{-1}(g,\xphi) &= B_{-1}(g_1, \omega(h_{\nu(g_0)},g_0)\xphi) \\
    &= \vol([\SL_2])A_0(g_1, \omega(h_{\nu(g_0)},g_0)\xphi) + C(\xphi) \\
    &= \vol([SL_2])A_0(g_1g_0,\xphi) + C'(\nu(g),\xphi).
\end{align*}
\end{proof}
\subsection{Calculation of the period: nontrivial case}
 \subsubsection{}
 We now assume that $S =\emptyset, $ so that $\Pi =\Pi_\emptyset(\pi_1,\pi_2) $ is generic, and compute $\mathcal P_{\emptyset,\pi_1,\pi_2,\pi} $. 
\begin{thm}
\label{theorem: new global. Pairing}\begin{enumerate}
    \item 
Choose vectors $\xphi_1\in\mathcal S_{\A_F}(\langle e_2\rangle\otimes V)$, $\xphi_2\in\mathcal S_{\A_F}(\langle e_4\rangle\otimes V)$, $\alpha\in \pi_1\boxtimes \pi_2$, and $\beta\in\pi$. Then: 
$$\mathcal P_{\emptyset,\pi_1,\pi_2,\pi}(\theta_{\xphi_1\otimes\xphi_2}(\alpha),\beta)=\Val_{s=\frac{1}{2}}\integral_{[\PGSO(V)]}\boldsymbol E(g,s;[\xphi_2])\alpha(g)\theta_{\xphi_1}(\beta)(g)\d g,$$
where $\PGSO(V)(\A_F) = \PGL_2(\A_F)\times \PGL_2(\A_F)$ is given the product Haar measure.
\item
$\mathcal P_{\emptyset, \pi_1,\pi_2,\pi}$ is identically zero unless $\pi\cong\pi_2^\vee$.
\item  Suppose we are given factorizations:
$$\xphi_1 =\otimes_v\xphi_{1, v}\in\mathcal S_{\A_F} (V),\;\;\xphi_2 =\otimes_v\xphi_{2, v}\in\mathcal S_{\A_F} (V),$$
$$\alpha= \otimes_v\alpha_{v}\in\pi_1\boxtimes \pi_2 ,\;\;\;\beta=\otimes_v\beta_v\in\pi_2 ^\vee,$$ along with decompositions of the global Whittaker functions:
$$ \alpha_{N\times N,\psi\times\psi }(g)=\prod_vW_{\alpha, v}(g_v),\;\;g=(g_v)\in \GSO(V) (\A_F),$$ 
$$ \beta_{N,\psi^{-1}}(h)=\prod_v W_{\beta,v}(h_v),\;\;h=(h_v)\in \GL_2(\A).$$

Then for a sufficiently large finite set of primes $S $,  we have: $$\mathcal P_{\emptyset} (\theta_{\xphi_1\otimes\xphi_2} (\alpha),\beta) = 2|D_F|^{\frac{1}{2}}\cdot\pi^{-d}\frac{L ^ S (1,\pi_1\times\pi_2 ^\vee) L ^ S (1,\Ad\pi_2)}{\zeta_F^S(2)} \prod_{v\in S}\frac{\mathcal Z_v (\xphi_{1, v},\xphi_{2, v},\alpha_{v},\beta_v)} {1 -q_v ^ {-1}} $$
where $\mathcal Z_v (\xphi_{1, v},\xphi_{2, v},\alpha_{v}, \beta_v)$ is the local zeta integral:
\begin{equation}
    \begin{split}
        \integral_{(N\times N\backslash\PGSO(V))(F_v)}\integral_{SL_2 (F_v)}W_{\alpha,v}(g)
W_{\beta, v}(h_1h_c)  \omega(h_1h_c, g)\widehat\xphi_1 (1,0,0,-1)\phi ^ 0 (g_2) [\xphi_2] (g_1)\d h_1\d g\\ c = \det(g_1g_2), \;\;g=\boldsymbol p_Z (g_1, g_2). 
    \end{split}
\end{equation}
Here $\phi ^ 0 (g_2) $ is the standard spherical section of $I(1)$.
\item The $L$-values $L^S(1, \pi_1\times \pi_2^\vee)$ and $L^S(1,\Ad \pi_2)$ are nonzero.
Moreover,  for each place $v $, there exist choices of $\xphi_{i, v},\alpha_{v} $, and $\beta_{v} $ such that $$\mathcal Z_v(\xphi_{1, v},\xphi_{2, v},\alpha_{ v}, \beta_v)\neq 0.$$
\end{enumerate}
\end{thm}
\begin{proof}
First, fix the Haar measure $\d g$ on $\SO(V)(\A_F)$  such that, under the surjective natural map $[\SO(V)]\times \mathcal C\to [\PGSO(V)] = [\PGL_2]\times [\PGL_2]$, the  Haar measure on $[\PGL_2]\times [\PGL_2]$ induced from (\ref{subsubsec:measures}) pulls back to $\d g\d c$.
We expand: 
\begin{equation*}
\tag{$\ast$} \mathcal P_{\emptyset,\pi_1,\pi_2,\pi}\left (\theta_{\xphi_1\otimes\xphi_2} (\alpha),\beta \right) = \integral_{[Z_H\backslash H]}\theta_{\xphi_1\otimes\xphi_2} (\alpha) (h, h')\beta (h),
\end{equation*}
which by definition is:
\begin{align*}
(\ast)&=\integral_{[Z_H \backslash H]}\integral_{[\GSO (V) ^ {\nu (h)}]}\theta (h, g;\xphi_1) \theta (h', g;\xphi_2)\alpha(g)\beta (h)\\
& =\integral_{\mathcal C}\integral_{[\SL_2]}\integral_{[\SO (V)]}\integral_{[\SL_2]}\theta (hh_c, gg_c;\xphi_1)\theta (h' h_c, gg_c;\xphi_2)\alpha (gg_c) \beta (h)\\
& =\integral_{\mathcal C}\integral_{[\SL_2]}\integral_{[\SO (V)]}\theta (h'h_c,g g_c;\xphi_2)\theta_{\xphi_1} (\beta) (gg_c) \alpha (gg_c).
\end{align*}
Now, by the reasoning of \cite[\S5.5]{kudla1994regularized}, the latter integral is  equal to the residue at $s = \frac {1} {2} $ of:
\begin{align*}
    \frac{1}{\kappa\cdot (4s ^ 2-1)}\integral_{\mathcal C}\integral_{[\SL_2]}\integral_{[\SO (V)]}\theta (h h_c, gg_c;\omega (z)\xphi_2)E_0 (h, s) \alpha(g) \theta_{\xphi_1} (\beta) (gg_c)\d g\d h\d c,
\end{align*}
which is meromorphic for $\Re (s)\gg 0$. Here $\kappa$ is as in Proposition \ref{prop:constanttermcomputed}.
Now, by the principle of meromorphic continuation, we may interchange the integrals over $\SL_2 $ and $\SO (V) $, and obtain:
\begin{align*}
    (\ast)&=\integral_{[\PGSO(V)]}B_{-1} (g,\xphi_2) \alpha (g) \theta_{\xphi_1} (\beta) (g) \d {g} \\
    &=\Val_{s = \frac {1} {2}}\integral_{[\PGSO(V)]}\boldsymbol E ( g,s; [\xphi_2]_s) \alpha (g) \theta_{\xphi_1} (\beta) (g) \d {g},
\end{align*}
by Theorem \ref{thm: second term identity} and the cuspidality of $A.$ This is (1). 
For (2), since $\theta_{\xphi_1}(\beta)(g)$
lies in the automorphic representation $\pi'\boxtimes\pi'$ of $\GSO(V)(\A_F)$, it is a linear combination of functions of the form $$\boldsymbol p _Z(g_1,g_2)\mapsto f_1(g_1)f_2(g_2).$$
Combining this observation with Proposition \ref{prop:GSO_Eisenstein_GL2}, it follows that $(\ast)$ is a linear combination of integrals of the form $$\Val_{s=\frac{1}{2}} \integral _{[\PGL_2\times\PGL_2]} E(g_1,s;M[\xphi_2])\alpha (\boldsymbol p_Z (g_1, g_2)) f_1(g_1)f_2(g_2)\d g_1\d g_2,$$
which clearly vanish unless $f_2\in\pi_2^\vee$, i.e., unless $\pi'\cong\pi_2.$
This proves (2).
In order to prove (3), we replace  $(\ast)$ with an equivalent integral that can be unfolded:
\begin{align*}
    (\ast) & =\frac{1}{\kappa}\Res_{s=\frac{1}{2}}\integral_{[\PGSO(V)]} E (g_1, s; M [\xphi_2 ]) E_0(g_2,s) \alpha(g)\theta_{\phi_1}(\beta) (g)\d g,\;\; g=\boldsymbol p_Z (g_1, g_2) \\
    &=\frac{1}{\kappa}\Res_{s=\frac{1}{2}}\integral_{N(\A)\times N (\A)\backslash \PGSO(V) (\A)  } M[\xphi_2]_s (g_1)\phi^0_s (g_2) \alpha_{N\times N,\psi\times\psi}(g)\theta_{\phi_1}(\beta)_{N\times N,\psi^{-1}\times \psi^{-1}}(g)\d g.
\end{align*}
This factors into an Euler product $$(\ast)=\frac {1} {\kappa} \Res _{s=\frac{1}{2}}\prod_v\mathcal Z_v (s,\xphi_{1,v}, \xphi _{2,v}, \alpha_v,\beta_v), $$ where the local zeta integrals are  (applying Lemma \ref{Lemma: Fourier coefficient of cuspidal left}):\begin{equation}
    \begin{split}
        \integral_{(N\times N\backslash\PGSO(V))(F_v)}\integral_{SL_2 (F_v)}W_{\alpha,v}(g)
W_{\beta, v}(h_1h_c)  \omega(h_1h_c, g)\widehat\xphi_1 (1,0,0,-1)\phi ^ 0 (g_2) [\xphi_2] (g_1)\d h_1\d g\\ c = \det(g_1g_2), \;\;g=\boldsymbol p_Z (g_1, g_2). 
    \end{split}
\end{equation}
At an unramified place $v$ such that $W_{\alpha, v} $, $W_{\beta,v}$, $\xphi_{i,v} $
are all the standard spherical vectors, the inner integral $$\integral_{\SL_2(F_v)} W_{\beta, v}(h_1h_c)  \omega(h_1h_c, g)\widehat\xphi_1 (1,0,0,-1) \d h_1 $$
is exactly the standard spherical Whittaker function for $\pi_2^\vee\boxtimes \pi_2^\vee $, by the unramified theta correspondence. Then,  the standard Rankin-Selberg calculations show that we have the Euler factor $$\mathcal Z_v (s,\xphi_{1,v}, \xphi _{2,v}, \alpha_v,\beta_v)=\frac{L_v(s+\frac{1}{2},\pi_1\boxtimes\pi_2^\vee)L_v  (s+\frac{1}{2},\pi_2\boxtimes\pi_2^\vee)}{1-q_v^{-2}}.$$
Now the formula (3) follows by comparing with Proposition \ref{prop:constanttermcomputed}.
The non-vanishing of the $L$-values in (4) is well-known; see for instance \cite{shahidi1981amer} and \cite{hida1988modules}. The non-vanishing of the local zeta integrals at ramified places also follows from the non-vanishing for Rankin-Selberg local zeta integrals, cf. \cite{jacquet1972rs}.
\end{proof}
\section{Proof of main result: special cycles in the generic case}\label{sec:coh generic proof}
In this section, we apply the results of \S\ref{Section:. See Yoshida} to the cohomology of Shimura varieties. Since the Schwartz functions at the Archimedean places must be chosen rather carefully to obtain automorphic forms that contribute to cohomology, we must begin with several local calculations.
\subsection{Archimedean     calculations}\label{subsec: Archimedean calculations}
\subsubsection{}\label{subsubsec: arch Conventions part one} 
We first establish some general conventions for the local Weil representation for the pair $(V,W_{2n,\R})$, where $V = V_{M_2(\R)}$.
 Fix coordinates on $W_{2n}\otimes V $ by:
\begin{equation}
\begin{split}
    (\underline x_1,\cdots,\underline x_{2n}) &\longleftrightarrow\sum e_i\otimes\underline x_i,\\
    \underline x_i = (x_i, y_i, z_i, w_i) &\longleftrightarrow\begin{pmatrix}x_i & y_i\\z_i & w_i\end{pmatrix}.
\end {split}
\end{equation}
Let $K_n\subset G = \GSp_{2n,\R} $ be as in (\ref{Sub sub: weights for unitary group}), let  $H = \GSO(V)$, let $R_0$ be as in (\ref{subsubsection:local notation}), and let $L =Z_H\cdot\boldsymbol p_Z (\SO (2)\times\SO (2))\subset H(\R). $ Also let $L_1 \subset L $
be the kernel of $\nu_H $ restricted to $L $, so that $L_1 =\boldsymbol p_Z (\SO (2)\times\SO (2)) $.
For any integers $m_1, m_2 $
with $m_1\equiv m_2\pmod 2, $ let $\chi_{m_1, m_2} $ be the character of $L$ which is given by $\omega_{m_1}^{-1} = \omega_{m_2}^{-1}$ on $Z_H$ and by $\chi_{m_1}\boxtimes\chi_{m_2} $
on $L_1$. Finally let $(K\times L)_0 = (K\times L)\cap R_0$.
\subsubsection{}\label{subsubsec: Arch conventions part two}
Let $S ^ 0(n)\subset\mathcal S_{F_v} (\langle e_2, \cdots, e_{2n}\rangle\otimes V) $
be the subspace of Schwartz functions of the form $$\xphi (\underline x_2,\cdots, \underline x_{2n})=p(\underline x_2,\cdots,\underline x_{2n})\exp (-\pi (|\underline x_2 | ^ 2+ \cdots +|\underline x_{2n} | ^ 2), $$
where $p $ is a polynomial, and let $S ^ 0_d(n)\subset S ^ 0 (n)$
be the subset such that $p $
is homogeneous of degree $d $.
As a $(\mathfrak r_0, (K_n\times L)_0) $-module, $S ^ 0 $
is isomorphic to the Fock space $\mathcal F_n$ of complex polynomials in $4n$ variables, cf. \cite{howe1989transcending}; the isomorphism does not preserve degrees, but it does carry $S ^ 0_{\leq d}(n) =\oplus_{i\leq d} S ^ 0_i $
isomorphically  onto $\mathcal F_{\leq d} $, the subspace of polynomials of degree less than or equal to $d $.
The following proposition is the key fact we will need about the structure of the $(K_n\times L)_0$-module $S_{\leq d}^0(n)$.
\begin{prop}\label{prop: substitute for Harris}
\begin{enumerate}
    \item If the $U (n) $-representation of highest weight $(a_1,\cdots,a_n)$ appears in $S^0(n)_{\leq d}$, then $|a_1| + \cdots + |a_n| \leq d$.
    \item If $m_1\equiv m_2\pmod 2$ are integers such that $m_1 = \pm m_2$ if $n = 1$, define $$a = \frac{|m_1 + m_2|}{2},\;\; b = \frac{|m_1 - m_2|}{2},$$ 
    and let $\tau$ be the unique representation of $K_n$ whose restriction to $\R^\times$ is $\omega_m^{-1}$ and which has weight $(a,0,\cdots, 0, -b)$ when restricted to $U(n)$. 
    Then $$\dim\left(S^0_{\leq a + b}(n)\otimes \tau\otimes \chi^\vee_{m_1,m_2}\right)^{(K_n\times L)_0} = 1.$$
\end{enumerate}
\end{prop}
\begin{proof}
This follows from \cite[Proposition 4.2.1]{harris1992arithmetic}; see Remark 3.2.2 of \emph{loc. cit.} to translate the $O(2)\times O(2)$ parameters into $\boldsymbol p_Z(\SO(2)\times \SO(2))$-parameters.
\end{proof}
In practice, we supplement this proposition with an explicit calculation:
\begin{prop}\label{prop: S02 covariance}
Suppose $n = 1$ and $m\geq 0$. Then for $\epsilon = \pm 1$, a generator for the one-dimensional space $$ (S ^ 0_{\leq m} (1)\otimes\chi_{\epsilon m}\otimes\chi ^\vee_{m,\epsilon m})^{(K_1\times L)_0}$$
is given by 
$$\xphi^\epsilon_m (x, y, z, w) \coloneqq(x+\epsilon iy + iz -\epsilon w) ^ m\exp (-\pi |\underline{x}| ^ 2).$$
\end{prop}
\begin{proof}
It suffices to show that for all $$(k,\boldsymbol p_Z (k_1, k_2))\in U(1)\times\boldsymbol p_Z (\SO (2)\times\SO (2))\subset\SL_2 (F_v)\times\SO (V) (F_v), $$
we have: $$\omega (k,\boldsymbol p_Z (k_1, k_2))\xphi ^\epsilon_m =\chi_{-\epsilon m} (k)\chi_m (k_1)\chi_{\epsilon m} (k_2)\xphi ^\epsilon_m. $$

For the action of $U(1)\subset\SL_2, $ we calculate on the Lie algebra level using the following formulas for differential $\d \omega$ of the Weil representation:
 \begin {align*}\d\omega\left (\begin {pmatrix} 0 & 0\\1 & 0\end {pmatrix}, 0\right)& =\frac {1} {2\pi i}\left (\frac {\partial^2} {\partial z\partial y} -\frac {\partial ^ 2} {\partial w\partial x}\right),\\
 \d\omega\left (\begin {pmatrix} 0 & 1\\0 & 0\end {pmatrix}, 0\right) & = 2\pi i (xw - yz).\end {align*}
 Since $$\d\omega\left (\begin {pmatrix} 0 & 1\\-1 & 0\end {pmatrix}, 0\right)\xphi ^\epsilon_m = -im\epsilon\xphi ^\epsilon_m, $$
 the lemma follows.
\end{proof}
\subsubsection{}
For the remainder of this subsection,  $n= 2$.
We now define the vector-valued Schwartz functions adapted to constructing cohomology classes on Shimura varieties
as in \S\ref{sec: cohomology}, and compute two related local zeta integrals. 
Fix a choice of sign $\epsilon =\pm$ and an integer $m\geq2$, and let  $\tau_m ^\epsilon $
be the representation of $K_2$ defined in (\ref{subsubsec: notation for endoscopic lifts in coh}).
\begin{rmk}
In practice, it would suffice to perform the calculations below for a single choice of $\epsilon$; we have included both for maximum clarity and for the convenience of the reader.
\end{rmk}
\begin{prop}
\label{prop: one-dimensional harmonic}
Let $\phi_m^\epsilon$ generate the one-dimensional space  $$(S ^ 0_{\leq m +2}(2)\otimes\tau_m^\epsilon\otimes\chi_{(m +2),-\epsilon m}^\vee) ^ {(K_2\times L)_0}. $$
If $\l\in\tau_m ^{\epsilon,\check} $ is a highest (resp. lowest) weight vector if $\epsilon = + $ (resp. $\epsilon = - $),
then $\l(\fee)$ is proportional to $$\xphi_{m +1} ^{-\epsilon} (\underline x_2)\cdot\xphi ^ {\epsilon}_{1} (\underline x_4). $$
\end{prop}
\begin {proof}
The one-dimensionality follows from Proposition \ref{prop: substitute for Harris}(2). For the rest, $\l (\fee) $ is a vector of weight $\epsilon(m+1, -1)$ for $U(2) $. By Proposition \ref{prop: substitute for Harris}(1), $\tau ^\check $ is the only $U(2)$-type containing the highest weight $(\epsilon (m +1), -\epsilon) $
to appear in $S ^ 0_{\leq m +2} (2) $, so $\l (\fee) $
spans the weight $\epsilon(m +1, -1) $
subspace  of $(S ^ 0_{\leq m +2}\otimes\chi_{ (m +2,-\epsilon m)}^\check) ^ {L_1} $; on the other hand $\xphi ^{-\epsilon}_{m +1} (\underline x_2)\cdot\xphi ^ {\epsilon}_{1} (\underline x_4) $
lies in this subspace by Proposition \ref{prop: S02 covariance}, and the proposition follows.
\end{proof}
\begin{prop}\label{prop:Calculating theWheat Dr. Schwartz function}
Let $\l\in\tau_m ^{\epsilon,\check} $ be a vector of weight $(\epsilon m, 0) $. Then $\l (\fee_m^\epsilon) $
is proportional to $$\fee_{m}^\epsilon\coloneqq \left ((m+1)\left((x_4+ i z_4) ^ 2+ (y_4+ iw_4) ^ 2\right) - (x_2+ iz_2) ^ 2- (y_2+ iw_2) ^ 2\right)\xphi ^{-\epsilon}_{m}(\underline x_2) \exp(-\pi|x_4|^2). $$
\end{prop}
\begin{proof}
In light of Proposition \ref{prop: one-dimensional harmonic}, we must apply a lowering (resp. raising) operator to $\xphi ^{-\epsilon}_{m +1} (\underline x_2)\cdot\xphi ^ {\epsilon}_{1, v} (\underline x_4)$
in the case $\epsilon = + $ (resp. $\epsilon = - $). First, on the Lie algebra level, the lowering operator for $K _2$ is:
$$L =\begin {pmatrix}0 & 0 & 1 & i\\0 & 0 & - i & 1\\-1 & i & 0 & 0\\- i & -1 & 0 & 0\end {pmatrix}\in\mathfrak k\otimes\C\subset\mathfrak g\otimes\C, $$
and the raising operator $R $ is its complex conjugate. For compactness of notation, set $L ^ +\coloneqq L $ and $L ^ -\coloneqq R $. Recall the partial Fourier transform of \ref{subsubsec: Fourier}: $$\widehat\xphi (z_1, z_2, z_3, z_4, w_1, w_2, w_3, w_4) =\integral\xphi (\underline x_2,\underline x_4)\Sye (y_2z_1 - x_2w_1 + y_4z_3 - x_4w_3) \d x_2\d y_2\d x_4 \d y_4.$$
Using the identity $$\widehat {\omega (g, 1)\xphi} (z_1, z_2, z_3, z_4, w_1, w_2, w_3, w_4) =\widehat\xphi\left (g ^ {-1}\begin {pmatrix} z_1\\z_2\\z_3\\z_4\end {pmatrix}, g ^ {-1}\begin{pmatrix}
w_1\\w_2\\w_3\\w_4
\end{pmatrix}\right), $$
we calculate that \begin{multline}\omega (L ^\epsilon)\left [p (\underline x_2,\underline x_4)\xphi_0 (\underline x_2)\xphi_0 (\underline x_4)\right]= \big[- (z_2+\epsilon y_2) (\partial_{z_4} p +\epsilon\partial_{y_4} p) + (z_4 -\epsilon y_4) (\partial_{z_2} p -\epsilon\partial_{y_2} p) - (w_2 -\epsilon x_2) (\partial_{w_4} p -\epsilon\partial_{x_4} p) \\ + (w_4+\epsilon x_4) (\partial_{w_2} p +\epsilon\partial_{x_2} p) +\frac {\epsilon} {2\pi} (\partial ^ 2_{y_2, z_4} p +\partial ^ 2_{z_2, y_4} p -\partial ^ 2_{x_2, w_4} p -\partial ^ 2_{w_2, x_4} p)\big]\xphi_0 (\underline x_2)\xphi_0 (\underline x_4), \end{multline}
where we have abbreviated $$\xphi_0(\underline x)\coloneqq \exp(-\pi|\underline x|^2).$$
One can then check that $$\omega (L ^\epsilon)\left[\xphi ^{-\epsilon}_{m +1} (\underline x_2)\cdot\xphi ^ {\epsilon}_{1, v} (\underline x_4)\right]  =2\overline\fee_{m}^{\epsilon},$$
which proves the proposition.
\end{proof}
\subsubsection{}
For later use, we calculate two archimedean local zeta integrals related to $\overline{\fee}_{m}^{\epsilon}$. Write:
\begin{align*}
    \xphi_1 & =\xphi ^{-\epsilon}_{m} \\
    \xphi_2  & =\left ((x + iz) ^ 2+ (y + iw) ^ 2\right)\exp (-\pi|\underline x | ^ 2),\\
    \xphi_1'  & =\left ((x + iz) ^ 2+ (y + iw) ^ 2\right)\xphi ^{\epsilon}_{m}\\
    \xphi_2' & =\exp(-\pi |\underline x | ^ 2),
\end{align*}
so that
\begin{equation}
    \overline{\phi}_{m}^{\epsilon} = (m +1)\xphi_1\otimes\xphi_2 -\xphi_1'\otimes\xphi_2'.
\end{equation}
For each integer $n\geq 2 $ and pair of signs $\epsilon,\delta\in\set {\pm} $, let $W ^\epsilon_{n,\psi ^\delta} $
be the normalized weight $\epsilon n $
vector in the $\psi ^\delta $-Whitaker model of the discrete series representation of $\GL_2 (\R) $
of weight $n $; thus
\begin{equation}
    \label{equation: discrete series Whitaker new form}W_{n,\psi ^\delta} ^\epsilon\begin{pmatrix}
   \epsilon\delta  t ^ {1/2}& 0\\0 & t ^ {1/2}
    \end{pmatrix}=t ^ {n/2} e ^ {-2\pi t},\;\; W_{n,\psi ^\delta} ^\epsilon\begin{pmatrix}
    -\epsilon\delta t ^ {1/2} & 0\\0 & t ^ {-1/2}
    \end{pmatrix}= 0,\;\;\forall t >0.
\end{equation}
\begin{prop}
\label{prop: local Zeta archimedean}
With notation as above, $$\mathcal Z_v\left (\xphi_1,\xphi_2, W ^ -_{m +2,\psi}\otimes W ^\epsilon_{m,\psi}, W ^ {-\epsilon}_{m,\psi}\right) = \frac{\epsilon^{m}\cdot m!(m-1)!}{\pi^{2m+2}2^{3m+2}} $$
and $$\mathcal Z_v\left (\xphi_1',\xphi_{2}', W ^ -_{m +2,\psi}\otimes W ^\epsilon_{m,\psi}, W ^ {-\epsilon}_{m,\psi}\right) = 0.$$
\end{prop}
\begin{proof}
First, note that:
\begin {align*}
\omega (k,\boldsymbol p_Z (k_1, k_2))\xphi_1 &=\chi_{\epsilon m} (k)\chi_m (k_1)\chi_{-\epsilon m} (k_2)\xphi_1,\\
\omega (k,\boldsymbol p_Z (k_1, k_2))\xphi_2 &=\chi_2 (k_1)\xphi_2,\\
\omega (k,\boldsymbol p_Z (k_1, k_2))\xphi_1' & =\chi_{\epsilon m} (k)\chi_{m +2} (k_1)\chi_{-\epsilon m} (k_2)\xphi_1',\\
\omega (k,\boldsymbol p_Z (k_1, k_2))\xphi_2' & =\xphi_2',\;\;\; \forall (k,\boldsymbol p_Z (k_1, k_2))\in U(1)\times\boldsymbol p_Z (\SO (2)\times\SO (2))\subset\SL_2\times\SO (V).
\end{align*}
The first identity is just Proposition \ref{prop: S02 covariance}, and the latter three may be proved similarly. (Alternatively, the action of the first factor $U(1)\subset\SL_2 $
can be deduced from Proposition \ref{prop:Calculating theWheat Dr. Schwartz function}.) 
Next, we compute the inner integral for $\mathcal Z_v\left (\xphi_1,\xphi_2, W ^ -_{m +2,\psi}\otimes W ^\epsilon_{m,\psi}, W ^ {-\epsilon}_{m,\psi}\right)$: \begin{equation}I(g_1,g_2)  = \integral_{\SL_2 (F_v)}\omega (h_1h_c,\boldsymbol p_Z (g_1, g_2))\widehat\xphi_1(1, 0, 0, -1) W ^ {-\epsilon}_{m,\psi} (h_1h_c)\d h_1,\;\; c =\debt (g_1)\debt (g_2) .\end{equation}
 By the archimedean local theta correspondence for $\GL_2\times\GSO (V) $ and  Proposition \ref{prop: S02 covariance}, we have \begin{equation}I (g_1, g_2) =\lambda W ^ +_{m ,\psi ^ {-1}}(g_1)W ^ {-\epsilon}_{m,\psi ^ {-1}} (g_2) \end{equation}
 for a  scalar $\lambda $. To pin down the scalar, it suffices to calculate:
 \begin{equation*}
     \tag{$\ast$}I(g_0),\;\;\; g_0 \coloneqq \boldsymbol p_Z \left (\begin{pmatrix}
     -1 & 0\\0 & 1
     \end{pmatrix},\begin{pmatrix}
\epsilon & 0\\0 & 1
     \end{pmatrix}\right).
 \end{equation*}
 By definition, we have:
 \begin{align*}
     (\ast) & =\integral\omega\left (h_1\begin {pmatrix} 1 & 0\\0 & -\epsilon\end {pmatrix}, g_0\right)\widehat\xphi^{-\epsilon}_{m} (1, 0, 0, -1) W ^ {-\epsilon}_{m,\psi}\left (h_1\begin {pmatrix} 1 & 0\\0 & -\epsilon\end {pmatrix}\right)\d h_1 \\
     &= \integral \omega\left (h_1\begin {pmatrix} 1 & 0\\0 & -\epsilon\end {pmatrix}, g_0\right)\xphi^{\epsilon}_{m} ( x,  y, 0,-1)\psi (y) (-\epsilon) ^ mW ^ {-\epsilon}_{m,\psi}\left (h_1\begin {pmatrix} -\epsilon & 0\\0 & 1\end {pmatrix}\right)\d x\d y \d h_1.
 \end{align*}
 Recall that the  Haar measure on $\SL_2 $
 is given by $$\d h_1 =\frac {\d a\d t\d\theta} {2\pi t ^ 2},\;\; h_1 =\begin{pmatrix}
 1 & a\\0 & 1
 \end{pmatrix}\begin {pmatrix} t ^ {1/2} & 0\\0 & t ^ {-1/2}\end {pmatrix}\begin {pmatrix}\cos\theta &\sin\theta\\-\sin\theta &\cos\theta\end {pmatrix}. $$
 Adopting these coordinates, the integral becomes
 \begin {align*}
 (\ast) & =(-\epsilon)^m\integral\psi (a) t ^ {m/2-2} e ^ {-2\pi t}\psi (- a x)\omega\left (\begin {pmatrix} t ^ {1/2} & 0\\0 & t ^ {-1/2}\end {pmatrix}, 1\right)\xphi^{-\epsilon}_{m} (- x, -\epsilon y, 0,\epsilon)\psi (y)\d x \d y \d a \d t \\
 &=(-\epsilon)^m \int t^{m-1} (-2 + iy)^m \exp (-\pi t (4+ y ^ 2)) \psi(y)\d t\d y\\
 & =(-i\epsilon)^m\frac {(m-1)!} {\pi ^ m}\integral \frac{e^{2\pi i y}\d y}{(y - 2i)^m}\\
 & =(- i\epsilon) ^ m (2\pi i) ^ m \frac {(m -1)!} {(m-1)!\pi ^ m} e ^ {-4\pi} =(2\epsilon)^me^{-4\pi};
 \end{align*}
 hence $\lambda = (2\epsilon)^m.$
 
 By the equivariance properties of $\xphi_2, $
 $$M[\xphi_{2}]\in\Ind_{\overline B(\R)}^{\PGL_2(\R)} |\cdot|^{1/2}$$ is a section of weight two for $\SO(2)$. Thus it is determined by: \begin{equation}M [\xphi_{2}] (1) =\integral\xphi_{2} (x, y, 0, 0)\d x\d y = \frac {1} {\pi}. \end{equation}
 Now, our local zeta integral is given by 
 $$
 (2\epsilon)^m\cdot \left (\integral_{N\backslash\PGL_2 (F_v)} W ^ -_{m +2,\psi} (g_1) [\xphi_{2}] (g_1) W ^ +_{m,\psi ^ {-1}} (g_1)\d g_1\right)\cdot \left (\integral_{N\backslash\PGL_2 (F_v)} W ^ {\epsilon}_{m,\psi} (g_2) W ^ {-\epsilon}_{m,\psi ^ {-1} }(g_2)\phi ^ 0 (g_2)\d g_2\right). $$
Since both integrands are  right $\SO (2) $-invariant, and since the Haar measure on $\PGL_2 (\R) $ is given by
$$\d g =\frac {\d a\d t\d \theta} {\pi t ^ 2},\;\;\; g =\begin {pmatrix} 1 & a\\0 & 1\end {pmatrix}\begin {pmatrix} t & 0\\0 & 1\end {pmatrix}\begin {pmatrix}\cos\theta &\sin\theta\\-\sin\theta &\cos\theta\end {pmatrix},\;\; t\in\R ^\times,\theta\in [0,\pi), $$
we obtain \begin{equation}\frac{(2\epsilon) ^ m} {\pi}\left(\integral_0^\infty t ^ {m} e ^ {-4\pi t}\d t\right)\cdot\left( \integral_0^\infty t^{m-1}e^{-4\pi t}\d t\right) = \frac{\epsilon^{m}\cdot m!(m-1)!}{\pi^{2m+2}2^{3m+2}},\end{equation}
as claimed.

To show that $$\mathcal Z_v\left (\xphi_1',\xphi_{2}', W ^ -_{m +2,\psi}\otimes W ^\epsilon_{m,\psi}, W ^ {-\epsilon}_{m,\psi}\right) = 0,$$
we may ignore scalar factors. The inner integral
\begin{equation}I'(g_1,g_2)  = \integral_{\SL_2 (F_v)}\omega (h_1h_c,\boldsymbol p_Z (g_1, g_2))\widehat\xphi'_1 (1, 0, 0, -1) W ^ {-\epsilon}_{m,\psi} (h_1h_c)\d h_1,\;\; c =\debt (g_1)\debt (g_2)\end{equation}
must be of the form \begin{equation}I' (g_1, g_2) = W' (g_1) W ^ {-\epsilon}_{m,\psi ^ {-1}} (g_2), \end{equation}
where $W' $ is a weight $m +2 $
vector in the $\psi ^ {-1} $-Whittaker model of the discrete series representation of $\GL_2 (\R) $ of weight $m $. To calculate $W' $, let \begin{equation}
    X\coloneqq\begin {pmatrix} - i & 1\\1 & i\end {pmatrix}\in\mathfrak{gl}_2
\end{equation}
be the $\SO (2) $-raising operator. Then, up to scalar, \begin{equation}
    W' = X\cdot W_{m,\psi ^ {-1}} ^ + =\begin {pmatrix} - i & 0\\0 & i\end {pmatrix} W ^ +_{m,\psi ^ {-1}} +2\begin {pmatrix} 0 & 1\\0 & 0\end {pmatrix}\cdot W_{m,\psi ^ {-1}} ^ + - im W_{m,\psi ^ {-1}} ^ +.
\end{equation}
In particular, $W'\begin {pmatrix}t ^ {1/2} & 0\\0 & t ^ {-1/2}\end {pmatrix} = 0 $ 
and\begin{equation}
    W'\begin {pmatrix} - t ^ {1/2} & 0\\0 & t ^ {-1/2}\end{pmatrix}=-2 im t ^ {m/2} e ^ {-2\pi t} +8\pi it ^ {m/2+1}e ^ {-2\pi t}
\end{equation} (up to scalar) for $t>0$. 
To complete the proof, we will show the integral \begin{equation*}
    \tag{$\ast$} \integral_{N\backslash\PGL_2 (F_v)} W' (g_1) [\xphi_2] (g_1) W ^ +_{m,\psi ^ {-1}} (g_1)\d g_1
\end{equation*}
vanishes. Since $[\xphi_2]$ is a section of weight zero, ($\ast$) is proportional to
\begin{align*}
    (\ast\ast) &= \integral_0^\infty (-2 imt ^ {m/2} +8\pi it ^ {m/2+1})\cdot t\cdot t ^ {m/2} e ^ {-4\pi t}\frac {\d t} {t^2}\\
    &= -2im \int_0^\infty t ^ {m -1} e ^ {-4\pi t}\d t +8\pi i\integral_0 ^\infty e ^ {-4\pi t} t ^ m\d t\\
    & = - 2im\frac {(m -1)!} {(4\pi) ^ m} +8\pi i\frac {m!} {(4\pi )^ {m +1}} \\
    &= 0.
\end{align*}
\end{proof}
\subsection{Cohomological span of special cycle}
\subsubsection{}
Let $H = \GL_2\times_{\mathbb G_m} \GL_2\subset \GSp_4 $, viewed as an algebraic group over $F$. Then $\boldsymbol H$ possesses a Shimura datum, and we have a natural embedding of pro-algebraic varieties $$\iota: S(\boldsymbol H)\hookrightarrow  S(\boldgsp_4)\times S({\boldgl_2}), $$ induced from the map on the level of groups: $(h_1,h_2)\mapsto ((h_1,h_2), h_1)$.
For all weights $\boldsymbol m$ as in (\ref{subsubsec: notation for endoscopic lifts in coh}) above,   abbreviate by ${ \mathcal W
}_{\boldsymbol m} $ the local system $\mathcal V_{(\boldsymbol m - 2,0)}^\vee\boxtimes \mathcal V_{\boldsymbol m - 2} $
on $ S(\boldgsp_4)\times S(\boldgl_2)$. Note that
the constant local system $\underline {\Q(\boldsymbol m)} $ on $S (\boldsymbol H)$ is a direct factor with multiplicity one of the pullback $\iota^\ast( {\mathcal W}_{\boldsymbol m}) $, and in particular, we have a composite map (well-defined up to a scalar):
\begin{equation}\label{eqn: map on cohomology for cycle definition}
H ^ {4d}_c ( S(\boldgsp_4)\times S (\boldgl_2), { \mathcal W}^\vee_{\boldsymbol m}) \to H ^{4d}_c (S (\boldsymbol H),\iota^\ast ( {\mathcal W}^\vee_{\boldsymbol m}))\to  H ^{4d}_c (S ({\boldsymbol H}),\Q(\boldsymbol m)). 
\end{equation}
\begin{definition}
The cycle class $[\mathcal Z]\in H^{4d}(S(\boldgsp_4)\times S(\boldgl_2), {\mathcal W}_{\boldsymbol m})(2d)$ is the image of the fundamental class of $S(\boldsymbol H) $ under the map
$$H ^ {0} (S(\boldsymbol H),\Q(\boldsymbol m))\to H ^ {2d} (S(\boldgsp_4)\times S(\boldgl_2), {\mathcal W}_{\boldsymbol m} )(2d)$$
induced by the dual of (\ref{eqn: map on cohomology for cycle definition}).   
We write $$[\mathcal Z]_\ast: H ^ {3d}_c (S(\boldgsp_4), \mathcal V_{(\boldsymbol m-2,0)})(d)\to H^{d} (S(\boldgl_2), \mathcal V_{\boldsymbol m-2}) $$
for the induced map.
\end{definition}
\subsubsection{} \label{subsubsec: recall Poincare} Let $\pi$ be an automorphic cuspidal representation of $\GL_2(\A_F)$ of weight $\boldsymbol m$ whose central character has infinity type $\omega_{\boldsymbol m}$. If $\pi$ is defined over $E$, recall that the trace map  induces a perfect pairing:
$$\langle\cdot,\cdot\rangle: H ^ d_c (S(\boldgl_2), \mathcal V_{\boldsymbol m-2,E})[\pi_f] \times H ^ d (S(\boldgl_2), \mathcal V^\vee_{\boldsymbol m-2,E})[\pi_f^\vee] \to H ^ {2d}_c (S(\boldgl_2), E)\to E(d). $$
\begin{prop}
\label{crop: cohomology. Pairing special cycle variant}
Let $\pi $ be as above, and let $\pi_1,\pi_2 $ be as in (\ref{subsubsec: notation for endoscopic lifts in coh}), with $\Pi =\Pi_S (\pi_1,\pi_2) $, for some $S = S_f\sqcup S_\infty $ such that $| S | $ is even.
\begin {enumerate}
\item For choices of signs $\boldsymbol\epsilon,\boldsymbol\epsilon' $,
let $\sigma_{\boldsymbol\epsilon,\boldsymbol\epsilon'}:\boldsymbol\tau ^ {\boldsymbol\epsilon}_{\boldsymbol m, S_\infty}\to\C $ be the projection onto the weight $(-\boldsymbol\epsilon'\boldsymbol m, 0) $-component (hence $\sigma$ is trivial unless $S_\infty = \emptyset$ and $\boldsymbol \epsilon = \boldsymbol\epsilon'$).
Then the following diagram commutes up to a nonzero scalar:
\vspace{0.3cm}
\begin{center}
\begin{tikzcd}
    \left(\Pi \otimes \boldsymbol\tau_{\boldsymbol m, S_\infty}^{\boldsymbol\epsilon}\right)^{\boldsymbol K_2} \otimes \left(\pi^\vee\otimes \boldsymbol\chi_{-\boldsymbol\epsilon'\boldsymbol m}^\vee\right)^{\boldsymbol K_1} \arrow[r, "\sigma_{\boldsymbol\epsilon,\boldsymbol\epsilon'}\otimes \operatorname{id}"]\arrow[d,"\CL_S^{\boldsymbol\epsilon}\otimes \CL_{\boldsymbol\epsilon'}'"] & \Pi\otimes \pi \arrow[dddd, "\mathcal P_S"] \\
        H^{\ast}_{(2)} (S(\boldgsp_4), \mathcal V_{(\boldsymbol m-2,0),\C})[\Pi_{S_f}] \otimes H^{\ast}_{(2)} (S(\boldgl_2), \mathcal V^\vee_{\boldsymbol m - 2,\C})[\pi_f^\vee]\arrow[d,    "\sim"] & \\
    H_c^{3d} (S(\boldgsp_4),\mathcal V_{(\boldsymbol m-2,0),\C})[\Pi_{S_f}] \otimes H_c^d (S(\boldgl_2), \mathcal V_{\boldsymbol m-2, \C}^\vee)[\pi_f] \arrow[d, "{[\mathcal Z]}_{\ast}\otimes \operatorname{id}"] & \\
    H^d (S(\boldgl_2), \mathcal V_{\boldsymbol m-2,\C}) \otimes H_c^d (S(\boldgl_2), \mathcal V_{\boldsymbol m-2,\C}^\vee) \arrow[d,"{\langle\cdot, \cdot\rangle}"] & \\
    \C \arrow[r,equals] & \C 
    \end{tikzcd}
\end{center}
\vspace{0.3cm}
\item Suppose $S_\infty =\emptyset$ and $\boldsymbol\epsilon = \boldsymbol\epsilon $. After fixing isomorphisms $$\Pi_f\simeq\left (\Pi\otimes\boldsymbol\tau ^ {\boldsymbol\epsilon}_{\boldsymbol m,\emptyset}\right) ^ {\boldsymbol K_2},\;\;\pi_f ^\vee\simeq\left (\pi ^\vee\otimes\boldsymbol\chi_{-\boldsymbol\epsilon\boldsymbol m}^\vee\right) ^ {\boldsymbol K_1}, $$
the composites with the map from (1):
$$\Pi_f\otimes\pi_f ^\vee\xrightarrow{\;\sim\;}\left (\Pi\otimes\boldsymbol\tau ^ {\boldsymbol\epsilon}_{\boldsymbol m,\emptyset}\right) ^ {\boldsymbol K_2}\otimes\left (\pi ^\vee\otimes\boldsymbol\chi_{-\boldsymbol\epsilon\boldsymbol m}^\vee\right) ^ {\boldsymbol K_1}\longrightarrow\C $$
are independent of $\boldsymbol \epsilon $ up to a nonzero scalar.
\end{enumerate}
\end{prop}
\begin{proof}
Let $\l:V_{(\boldsymbol m - 2,0)}\otimes V_{\boldsymbol m - 2}^\check\to \Q(\boldsymbol m)$ be an $H (F) $-invariant projection. There exists a basis $\mathbbm 1_{\boldsymbol H}$ of $\wedge^{4d} \mathfrak p_H$  such that $$\int_{S(\boldsymbol H)} \omega = \int_{[Z_{H}\backslash H]}h^\ast\omega (\mathbbm 1_{\boldsymbol H})\d h$$
for all top-degree forms $\omega$ on $S(\boldsymbol H)$. Then by definition,
$$\langle [\mathcal Z]_\ast \alpha, \beta\rangle =\int_{[Z_{H}\backslash H]}\l\left [h ^\ast\iota^\ast (\alpha \wedge \beta) (\mathbbm 1_{\boldsymbol H})\right]\d h. $$
The composite map:
\begin{equation}\label{eq: composite}
\begin{split}
\boldsymbol \tau_{\boldsymbol m,S_\infty} ^ {\boldsymbol\epsilon}\otimes 
\boldsymbol\chi^\vee_{-\boldsymbol\epsilon'\boldsymbol m}\to \wedge^{\boldsymbol p(\boldsymbol\epsilon,S_\infty),\boldsymbol q(\boldsymbol \epsilon,S_\infty)}\mathfrak p_{\GSp_4}^\ast \otimes \wedge ^ {1-\boldsymbol p (\boldsymbol\epsilon'),1-\boldsymbol q (\boldsymbol\epsilon')}\mathfrak p_{\GL_2} ^\ast\otimes V_{(\boldsymbol m - 2,0)}\otimes V_{\boldsymbol m - 2}\\ \xrightarrow{\iota^\ast\otimes \l}\wedge ^ {\boldsymbol p (\boldsymbol\epsilon, S_\infty) +1 -\boldsymbol p(\boldsymbol\epsilon'), \boldsymbol q (\boldsymbol\epsilon, S_\infty) +1-\boldsymbol q (\boldsymbol\epsilon')}\mathfrak p_H\xrightarrow{\mathbbm 1_{\boldsymbol H}} \C 
\end{split}
\end{equation}
is a map of $U(1)^d\times U(1)^d$-modules, where  the action on $\boldsymbol\chi_{-\boldsymbol\epsilon'\boldsymbol m}$, $\wedge ^ {\boldsymbol p (\boldsymbol\epsilon'),\boldsymbol q (\boldsymbol\epsilon')}\mathfrak p_{\GL_2}$, and 
$V_{\boldsymbol m - 2}$ is through projection to the first factor. In particular, (\ref{eq: composite})
is trivial unless $S_\infty = \emptyset$ and $\boldsymbol\epsilon = \boldsymbol\epsilon' $, in which case it is proportional to the projection onto the weight $(-\boldsymbol\epsilon\boldsymbol m, 0)$-component of $\boldsymbol\tau_{\boldsymbol m,\emptyset}^{\boldsymbol\epsilon}$; moreover, a direct calculation shows it is nonzero, and (1) follows.

For (2), let $$g_{\boldsymbol\epsilon} =\begin{pmatrix}\epsilon_ v & & &\\&\epsilon_v & &\\& & 1&\\& & & 1\end {pmatrix}_{v |\infty}\in G(F\otimes \R) \subset G (\A_F),\;\;\;\; g'_{\boldsymbol\epsilon} =\begin {pmatrix}\epsilon_v &\\& 1\end {pmatrix}_{v |\infty}\in G' (F\otimes\R)\subset G' (\A_F). $$
We have an obvious commutative diagram

\begin {center}
\begin{tikzcd}
  \left(\Pi\otimes\boldsymbol\tau ^ {\boldsymbol\epsilon}_{\boldsymbol m,\emptyset}\right) ^ {\boldsymbol K_2}\otimes\left (\pi ^\vee\otimes\boldsymbol\chi_{-\boldsymbol\epsilon\boldsymbol m}^\vee\right) ^ {\boldsymbol K_1} \arrow[r, "\sigma_{\boldsymbol\epsilon,-\boldsymbol\epsilon}\otimes \operatorname{id}"]\arrow[d] &\Pi\otimes\pi ^\vee\arrow[d]\\
   \left(\Pi\otimes\boldsymbol\tau ^ {\boldsymbol +}_{\boldsymbol m,\emptyset}\right) ^ {\boldsymbol K_2}\otimes\left (\pi ^\vee\otimes\boldsymbol\chi^\vee_{- \boldsymbol m}\right) ^ {\boldsymbol K_1}\arrow [r,"\sigma_{\boldsymbol+,\boldsymbol  -}\otimes \operatorname{id}"]&\Pi\otimes\pi ^\vee
\end{tikzcd}
\end {center}
in which  the vertical arrows are translation by $(g_{\boldsymbol\epsilon}, g'_{\boldsymbol\epsilon})$ and $ \pm $ stands for the constant sign $ (\pm)_v|\infty$.
However, since $(g_{\boldsymbol\epsilon}, g'_{\boldsymbol\epsilon})$ lies in $H (\A_F)\subset G (\A_F)\times G' (\A_F), $
this translation has no effect on the period integral $\mathcal P_S $, and (2) follows.
\end{proof}

\begin{thm}\label{biggie thm}
Let $\pi_1,\pi_2,\pi$ be cuspidal automorphic representations of $G' (\A) $
of weights $\boldsymbol m +2 $, $\boldsymbol m$, and $\boldsymbol m $, respectively, where $\boldsymbol m = (m_v)_{v |\infty} $
for positive  integers $m_v $. Assume that the central characters of $\pi_1 $ and $\pi_2$ agree, and that the the central characters of $\pi_1,$ $\pi_2$, and $\pi$ all have infinity type $\omega_{\boldsymbol m}$.
Let $\Pi_{S_f}$ be as  in (\ref{subsubsec: notation for endoscopic lifts in coh}) for set $S_f$ of finite places of $F$.
Then, for any  coefficient field $E \supset F ^ c $
such that $\Pi,\pi_i, $ and $\pi $
are defined over $E $, the induced map
$$[\mathcal Z]_\ast: H ^ {3d}_! (S(\boldgsp_4), \mathcal V_{(\boldsymbol m - 2,0)})(d) [\Pi_{S_f}]\to H ^ d_!(S(\boldgl_2), \mathcal V_{\boldsymbol m - 2}) [\pi_f] $$
is trivial unless $\pi =\pi_2 $ and $S_f =\emptyset $. In the latter case, $[\mathcal Z]_\ast $ takes the form:
$$\Pi_{\emptyset}\otimes H_! ^ {3d} (S(\boldgsp_4), \mathcal V_{(\boldsymbol m - 2,0)})_{\Pi_{\emptyset}}(d)\xrightarrow{\l\otimes s}\pi_{2,f}\otimes H _!^ d (S(\boldgl_2), \mathcal V_{\boldsymbol m - 2})_{\pi_{2,f}},$$
where $s$ is an surjection and $\l $ is a nontrivial $E $-linear map.
\end{thm}
\begin{proof}
Without loss of generality, suppose $E = \C$.
By Proposition \ref{crop: cohomology. Pairing special cycle variant}, Theorem \ref{thm: non-generic vanishing}, and Theorem \ref{theorem: new global. Pairing}, we immediately reduce to the case $S_f = \emptyset$ and $\pi= \pi_2$. In this case, write $\Pi = \Pi_\emptyset(\pi_1,\pi_2)$.  Under the decomposition $$H ^ {3d}_! (S (\boldgsp_4),\mathcal V_{(\boldsymbol m -2, 0),\C}) [\Pi_f] =\bigoplus_{S_\infty}\bigoplus_{\boldsymbol\epsilon} H ^ {\boldsymbol p (\boldsymbol\epsilon, S_\infty),\boldsymbol q (\boldsymbol\epsilon, S_\infty)}_{(2)} (S (\boldgsp_4),\mathcal V_{(\boldsymbol m -2, 0),\C}) [\Pi_f]$$ provided by \ref{prop: plectic Hodge types and class maps} and \ref{Prop: comparison isomorphism's for G},
Proposition \ref{crop: cohomology. Pairing special cycle variant}(1) implies that $[\mathcal Z]_\ast $
is trivial on components with $S_\infty\neq\emptyset $, and maps $H ^ {\boldsymbol p (\boldsymbol\epsilon,\emptyset),\boldsymbol q (\boldsymbol\epsilon,\emptyset)}_{(2)}(S (\boldgsp_4),\mathcal V_{(\boldsymbol n -2, 0),\C}) (d) [\Pi_f] $
to $H ^ {\boldsymbol p (\boldsymbol\epsilon),\boldsymbol q (\boldsymbol\epsilon)}_{(2)}(S (\boldgl_2),\mathcal V_{\boldsymbol m -2,\C}) [\pi_{2, f}] $. Moreover, by Proposition \ref{crop: cohomology. Pairing special cycle variant}(2)
and Proposition 
 \ref{prop: comparison isomorphism for G'},
 $[\mathcal Z]_\ast$ is a pure tensor $\l\otimes s $, and $s $
 is surjective provided it is nontrivial.
 Thus, for any single choice of $\boldsymbol\epsilon $, it suffices to show that
\begin{equation}
\begin{split}
H^{\boldsymbol p(\boldsymbol\epsilon,\emptyset), \boldsymbol q(\boldsymbol\epsilon,\emptyset)}_{(2)} (S(\boldgsp_4), \mathcal V_{(\boldsymbol m-2,0),\C})[\Pi_f] \otimes H_{(2)}^{\boldsymbol 1-p(\boldsymbol\epsilon),\boldsymbol 1-q(\boldsymbol\epsilon)}(S(\boldgl_2), \mathcal V_{\boldsymbol m-2,\C})[\pi_{2,f}^\vee]  \xrightarrow{\;\;\langle [\mathcal Z]_\ast\cdot,\cdot\rangle\;\;
    }\C
    \end{split}
\end{equation}
is nontrivial.

Indeed, let $$\boldsymbol\phi^{\boldsymbol\epsilon}_\infty = \otimes_{v|\infty}\phi^{\epsilon_v}_{m_v}\in \mathcal S_{F\otimes \R}(\langle e_2,e_4\rangle\otimes V)\otimes \boldsymbol\tau_{\boldsymbol m,\emptyset}^{\boldsymbol \epsilon}\otimes \boldsymbol\chi_{-\boldsymbol\epsilon}^\vee,$$ where $\phi^{\epsilon_v}_{m_v}$ is the vector-valued Schwartz function of Proposition \ref{prop: one-dimensional harmonic}. Also let \begin{equation}
    \theta_{\boldsymbol\epsilon}: \mathcal S_{\A_{F,f}}(\langle e_2,e_4\rangle\otimes V) \twoheadrightarrow \left(\Pi\otimes \boldsymbol\tau^{\boldsymbol\epsilon}_{\boldsymbol m, \emptyset}\right)^{\boldsymbol K_2}
\end{equation}
be the $\C$-linear map $$\xphi_f \mapsto \theta_{\xphi_f\otimes \phi^{\boldsymbol\epsilon}_\infty}(f_1\otimes f_2),$$ where $f_1\in \pi_1$ and $f_2\in \pi_2$ are nonzero newforms of weights $-(\boldsymbol m + 2)$  and $\boldsymbol\epsilon\boldsymbol m,$ respectively. 

Now Proposition \ref{crop: cohomology. Pairing special cycle variant} and Proposition \ref{prop:Calculating theWheat Dr. Schwartz function} imply that the composite map
\begin{equation}
    \mathcal S_{\A_{F,f}}(\langle e_2,e_4\rangle \otimes V)\otimes (\pi_2^\vee\otimes \boldsymbol\chi^\vee_{-\boldsymbol\epsilon\boldsymbol m})^{\boldsymbol K_1}\xrightarrow{\theta_{\boldsymbol\epsilon}\otimes\operatorname{id}} \left(\Pi\otimes \boldsymbol \tau_{\boldsymbol m, \emptyset}^{\boldsymbol\epsilon}\right)^{\boldsymbol K_2} \otimes (\pi_2^\vee\otimes \boldsymbol\chi^\vee_{-\boldsymbol\epsilon\boldsymbol m})^{\boldsymbol K_1}\xrightarrow{\langle[\mathcal Z]_\ast \circ \cl_{\emptyset}^{\boldsymbol\epsilon}, \cl'_{\boldsymbol\epsilon} \rangle}\C
\end{equation}
is given by \begin{equation}
    \xphi_f \otimes \beta\mapsto \mathcal P _\emptyset (\theta_{\xphi_f\otimes \overline \phi_{\infty}^{\boldsymbol\epsilon}}(f_1\otimes f_2), \beta) 
\end{equation}
up to a nonzero scalar,
where $$\overline \phi_{\infty}^{\boldsymbol\epsilon} = \otimes_v \overline \phi_{m_v}^{\epsilon_v}$$ for $\overline \phi_{m_v}^{\epsilon_v} \in\mathcal S_{F_v} (\langle e_2,e_4\rangle\otimes V)$ as in Proposition \ref{prop:Calculating theWheat Dr. Schwartz function}.
Suppose given factorizable test data $$\xphi_{i,f} = \otimes_{v\nmid \infty}\xphi_{i,v}\in \mathcal S_{\A_{F,f}}(V), \;\; i = 1,2$$ and $$\beta = \otimes_v \beta_v\in (\pi_2^\vee\otimes \boldsymbol\chi^\vee_{-\boldsymbol\epsilon\boldsymbol m})^{\boldsymbol K_1};$$  we also fix decompositions of the global Whittaker functionals as in Theorem \ref{theorem: new global. Pairing}.
Then, from Theorem \ref{theorem: new global. Pairing} and Proposition \ref{prop: local Zeta archimedean}, we conclude the formula
\begin{equation}\label{eq: formula}
\langle [\mathcal Z]_\ast\circ \cl_\emptyset^{\boldsymbol\epsilon} (\theta_{\boldsymbol\epsilon}(\xphi_{1,f}\otimes \xphi_{2,f})),\cl_{\boldsymbol\epsilon}' (\beta)\rangle\doteq \frac{L^S(1,\pi_1\times\pi_2^\vee)L^S(1, \Ad\pi_2)}{\zeta_F^S(2)} \prod_{\substack{v\in S\\ v\nmid \infty}} \frac{\mathcal Z_v(\xphi_{1,v},\xphi_{2,v}, f_{1,v}\otimes f_{2,v}, \beta_v)}{1-q_v^{-1}} \cdot \prod_{v|\infty} \frac{W_{\beta_v}(1)}{e^{-2\pi}}, 
\end{equation}
where $\doteq$ denotes equality up to a nonzero constant. 
The nonvanishing of the right-hand side, which follows from Theorem \ref{theorem: new global. Pairing}, concludes the proof.
\end{proof}
\begin{rmk}
In fact, since the formula (\ref{eq: formula}) is uniform in $\boldsymbol\epsilon$, Proposition \ref{crop: cohomology. Pairing special cycle variant}(2) was not strictly necessary for the proof of the theorem. 
\end{rmk}
\section{Non-tempered theta lifts on $\GSp_6 $}\label{sec:nontempered}
\subsection{Arthur parameters}
\subsubsection{}
Let $G $ be a split reductive group over $F $, and recall that a local Arthur parameter for $G $ is a $\widehat G $-conjugacy  class of homomorphisms $$\psi_v:WD_{F_v}\times\SL_2 (\C)\to\widehat G, $$
such that the restriction to $WD_{F_v} $ is a bounded Langlands parameter. 
If $v$ is non-archimedean and $\psi_v$ is unramified, then it determines a local Langlands parameter $\phi_v$ by the rule
\begin{equation}\label{EQ: unramified Arthur to Langlands}
    \phi_v(w) = \psi_v\left(w, \begin {pmatrix}|w | ^ {1/2}_v & 0\\0 & | w | ^ {-1/2}_v\end {pmatrix}\right). \end{equation}
An unramified representation of $G (F_v) $
is said to have Arthur parameter $\psi_v $
if its Langlands parameter is determined by the rule (\ref{EQ: unramified Arthur to Langlands}).
\subsubsection{}
 In \cite{arthur2013endoscopic}, Arthur defines discrete global parameters for $\SP_{2n} $ to be formal (unordered) sums
 \begin{equation}\label{eq:global parameter}
     \oplus \pi_i [d_i],
 \end{equation}
 where:
 \begin{itemize}
     \item $\pi_i $ are cuspidal automorphic representations of $\PGL_{n_i} $;
     \item $d_i\geq 0 $ are integers such that $\sum n_id_i = 2n +1 $;
     \item $\pi_i$ is conjugate symplectic if $d_i$ is even, and conjugate orthogonal if $d_i$ is odd;
     \item the pairs $(\pi_i,d_i)$ are distinct.
 \end{itemize}
The integers $d_i$ are to be interpreted as the dimensions of irreducible representations of $\SL_2(\C)$. Moreover, a global Arthur parameter $\psi $ induces a local Arthur parameter $\psi_v $ for each place $v$ of $F$ via the local Langlands classification for $\GL_{m}$. The reason we must take this rather circuitous definition of global Arthur parameters is the lack of a global Langlands group.
\subsubsection{}
In \cite{arthur2013endoscopic}, Arthur classified discrete automorphic representations of $\SP_{2n}$ by constructing local and global packets for these parameters, denoted $\Pi_{\psi_v} $ and $\Pi_\psi $, 
respectively. An automorphic representation  $\pi\subset \mathcal A_{(2)} (\SP_{2n}(\A_F))$
belongs to $\Pi_\psi$ if and only if, for almost all places $v$, $\pi_v$ is the spherical representation determined by $\phi_v $.
Although a full endoscopic classification for $\GSp_{2n} $ is not available, one can obtain partial results using the pullback  along the natural map $$\iota: \SP_{2n}\to \GSp_{2n}. $$ Indeed, for our purposes the following rather weak result is sufficient.
\begin{prop}\label{Prop: very weak endoscopic classification}
Let $\pi_1 ,\pi_2 \subset \mathcal A_{(2)} (\GSp_{2n} (\A_F)) $ be nearly equivalent discrete automorphic representations of $\GSp_{2n} (\A_F) $, and suppose that some irreducible constituent of $\iota^\ast\pi_1 $ has parameter $\psi$. Then for all places $v $ of $F $, any irreducible constituent of the admissible $\SP_{2n} (F_v) $-module $\iota^\ast\pi_2 $ belongs to the Arthur packet $\Pi_{\psi_v} $.
\end{prop}
\begin{proof}
Let $\pi_{0, v} $ be an irreducible constituent of $\iota^\ast\pi_{2, v} $. There exists a vector $f\in\pi_2$ generating an irreducible $\SP_{2n} (\A_F)) $-module $\pi_0 \subset\mathcal A_{(2)} (\A_F)$ whose component at $v $
is $\pi_{0, v} $.  By construction, $\pi_0$ has Arthur parameter $\psi $, and so the result follows from \cite{arthur2013endoscopic}.
\end{proof}
\subsection{Theta lift from $\GSO(V_B)$ to $\GSp_6$}
\subsubsection{}\label{Subsubsection: big theta lift defined}
For the remainder of this section, fix a \emph{non-split} quaternion algebra $B $ over $F$, and let $\pi$ be a tempered automorphic representation of $PB ^\times $.
We consider the representation $\pi\boxtimes\mathbbm {1} $ of $\GSO(V_B) \simeq B^\times \times B^\times /\mathbb G_m$ and its theta lift $\Theta(\pi\boxtimes\mathbbm 1)$ to $\GSp_6(\A_F)$; this is well-defined because $V_B $ is anisotropic, and descends to $\pigs_6(\A_F)$ because the similitude theta lift preserves central characters. Although this is not strictly necessary for our argument, it follows from \cite{gan2014regularized} that $\Theta (\pi\boxtimes\mathbbm 1) \neq 0$ for any such $\pi$.
However, $\Theta(\pi\boxtimes \mathbbm{1})$ need not be irreducible. 
\begin{prop}\label{prop: L2ness}
The theta lift $\Theta (\pi\boxtimes\mathbbm{1}) $ lies in the $L^2$ subspace
 $\mathcal A_{(2)} (\pigs_6 (\A_F))$. 
\end{prop}
By the usual criterion for square-integrability \cite[Lemma I.4.11]{moeglin1995spectral}, we must check that, for each standard parabolic subgroup $P = MN $ of $\GSp_6, $ the characters of $Z (M) $ appearing in the cuspidal component of the normalized Jacquet module $$\Theta (\pi\boxtimes\mathbbm 1)_N\otimes \delta_P ^ {-1/2} $$ all lie in the interior of the cone spanned by the \emph{negatives} of the characters appearing in the action of $Z(M)$ on $N$.
Since $\pi\boxtimes\mathbbm 1 $
is not a theta lift from $\GSp_2 =\GL_2, $
\cite[Theorem I.1.1]{rallis1984howe}  implies that
the Jacquet modules are given by:
\begin{equation}
   \Theta (\pi\boxtimes\mathbbm 1)_N =
    \begin{cases}
|\cdot|^2\Theta'(\pi\boxtimes\mathbbm 1),& M =\GSp_4\times\GL_1 \\
0, &\text{ otherwise}.
    \end{cases}
\end{equation}
Here $\Theta'(\pi\boxtimes\mathbbm 1) $
denotes the theta lift to $\GSp_4, $ and $|\cdot| $
is the norm character of $\GL_1. $ On the other hand, the action of $Z(M)$ on $N$ is through positive powers of $|\cdot | $, and $\delta_P = |\cdot | ^ 6 $; thus the criterion for square-integrability is satisfied.
\begin{prop}\label{prop: Arthur parrameter for the lift}
Suppose $\Pi $ is an irreducible constituent of $\Theta (\pi\boxtimes\mathbbm 1)$. 
\begin{enumerate}
 \item For all non-archimedean $v $, if $\pi_v $ is unramified with local
Langlands parameter $\xphi_v$, then $\Pi_v $
is unramified  with an Arthur parameter
$\psi_v $ such that the composite
$$W_{F_v}\xrightarrow{\psi_v} \widehat{\pigs_6} =\Spin_7\xrightarrow{r_{\text{spin}}}\GL_8$$
is given by $$\xphi_v\otimes S_2\oplus S_3\oplus S_1,$$ 
where $S_i$ is the $i$-dimensional irreducible representation of $\SL_2$.
\item Any irreducible constituent of $\iota^\ast\Pi $ has global Arthur parameter
$$\operatorname{JL}(\pi)[2]\oplus \mathbbm{1}[3].$$
\end{enumerate}
\end{prop}
\begin{proof}
For (1), Propositions \ref{prop: L2ness} and \ref{prop: local global compatibility} imply that $\Pi_v$ is an irreducible constituent of $\Theta_v(\pi_v\boxtimes \mathbbm 1)$ for all $v$. Since $\pi_v$ is tempered at all unramified $v$, $\Ind_{\GSO(V)(F_v)}^{\operatorname{GO}(V)(F_v)}\pi_v\boxtimes \mathbbm 1$ is irreducible. Adopting the notation of (\ref{subsubsection: notation for unramified principal series}) with $G = \GSO(V)$ and $H = \GSp_6$, $\pi_v\boxtimes \mathbbm 1$ is the spherical representation $\pi_\chi $
for the unramified character of $T_G (F_v)$
defined by $(-\log_q \alpha_v - 1/2, 1/2 -\log_q\alpha_v, \log_q\alpha_v)$, where $\set{\alpha_v, \alpha_v^{-1}}$ are the Sataka parameters of $\pi_v$. Then Proposition \ref{prop:spherical} implies that 
 $\Theta_v(\pi_v\boxtimes \mathbbm 1)$ is the irreducible representation $\sigma_\mu$ with $\mu(\alpha_v) = (-\log_q\alpha_v -1/2, 1/2 -\log_q\alpha_v, -1, 1/2+\log_q\alpha_v)$. Recall that any $\mu $
 determines an unramified  Langlands parameter for $H$: the characters $x_i, \lambda\in  \Hom_{F_v}(T_G, \mathbb G_m)$ correspond to cocharacters in $\Hom_{\C}(\mathbb G_m, \widehat T_G)$, and any unramified character $\mu = (\beta_1, \cdots, \beta_3, t)$ may be viewed as the element $$\lambda(q^{-t})\prod_i x_i(q^{-\beta_i}) \in \widehat T_H(\C).$$
 Then the Langlands parameter of $\sigma_\mu$ is the conjugacy class of the unramified map
 $$\xphi_{\mu}: W_{F_v} \to \widehat T_H(\C) \hookrightarrow \widehat H(\C)$$ such that $\xphi(\Frob_v)$ is the element corresponding to $\mu$. Now, the eigenvalues of $r_{\text{spin}} \circ \xphi_{\mu}$ on $\C^8$ are given by $q^{-t}\prod_{i \in S}q^{-\beta_i}$ as $S$ ranges over subsets of $\set{1,2,3}$. Hence in our case, the eight Frobenius eigenvalues are (with multiplicity) $q ^ {\pm 1/2} \alpha_v^{\pm 1}$, $q,$ $q^{-1}$, 1, and 1.
These Frobenius eigenvalues correspond to the unramified Arthur parameter $\xphi_v\otimes S_2\oplus S_3\oplus S_1$ according to (\ref{EQ: unramified Arthur to Langlands}).

 The second claim follows from the endoscopic classification for $\SP_6$, since the local Arthur parameter for  any irreducible constituent of $\iota^\ast\Pi_v$ is the composition of $\psi_v$ with the projection $\Spin_7\to \SO_7$ (cf. \cite{rallis1982langlands}).
\end{proof}
\subsection{Archimedean Arthur packet and cohomology}
\subsubsection{}
In order to construct Hodge classes from $\Theta  (\pi\boxtimes\mathbbm 1) $, it is essential to understand the structure of the cohomological local Arthur packet $\Pi_{\psi_v} $, where $\psi $ is a global parameter for $\SP_6$ of the kind  in Proposition \ref{prop: Arthur parrameter for the lift}(2) and $v $ is archimedean. We now recall the construction of $\Pi_{\psi_v} $ due to Adams and Johnson \cite{adams1987endoscopic}. (This construction agrees with \cite{arthur2013endoscopic} by \cite{adams2021equivalent,arancibia2019characteristic}.)
\subsubsection{}
We first fix an anisotropic maximal torus $T $ of $G =\SP_6 $. However, we depart from our usual coordinates for $G $, and choose a \emph{complex} basis of the $2n $-dimensional symplectic space with respect to which
\begin{equation}
    \mathfrak {g}_{\C} =\set{\begin{pmatrix}A & B\\C & - A ^ T\end {pmatrix}\bigg| A, B, C\in\mathfrak {gl}_{3,\C},\; B = B ^ T,\; C = C ^ T},
\end{equation}
 \begin{equation}
    T_{\C} =\set {\begin{pmatrix}\lambda_1 & & & & &\\&\lambda_2 & & & &\\& &\lambda_3 & & &\\& & &\lambda_1 ^ {-1} & &\\& & &&\lambda_{2} ^ {-1} &\\& & & & &\lambda_3 ^ {-1}\end {pmatrix}}\simeq (\C ^\times) ^ 3\subset G_{\C},
\end{equation}
and complex conjugation acts by
\begin{equation}
    \begin{pmatrix}A & B\\C & - A ^ T\end {pmatrix}\mapsto \begin{pmatrix}-\overline A^T & \overline C\\\overline B & \overline A.\end {pmatrix}
\end{equation}
We choose the Borel subgroup $B $ of $G_{\C} $
given by
\begin{equation}\label{EQ: Borel subgroup convention}
    B =\set{\begin {pmatrix}A & B\\0 & A ^ {- T}\end{pmatrix}\in G_{\C}\bigg| A\text{ is upper triangular}}.
\end{equation}
Let $\theta $ be the Cartan involution of $G $ that acts on $\mathfrak g_{\C} $ by
\begin{equation}
    \begin{pmatrix}
    A & B\\C & - A ^ T
    \end{pmatrix}\mapsto\begin{pmatrix}
    A & - B\\- C & - A ^ T
    \end{pmatrix}.
\end{equation}
The corresponding maximal compact subgroup of $G $
is $K = U (3) $.
The absolute Weyl group of $G $ is
\begin{equation}
    W (G, T) = \set{\pm 1} ^ 3\rtimes S_3,
\end{equation}
where $S_3 $ acts on $T $ by permuting $(\lambda_1,\lambda_2,\lambda_3) $
and $\set{\pm 1} ^ 3 $ acts by
$$(e_1, e_2, e_3)\cdot (\lambda_1,\lambda_2,\lambda_3) = (\lambda_1 ^ {e_1},\lambda_2 ^ {e_2},\lambda_3 ^ {e_3}). $$
The relative Weyl group $W_\R (G, T) $ of $G$ may be identified with $S_3. $
\subsubsection{}
Let $L\subset G$ be the unique  $\R$-subgroup isomorphic to $U(2)\times SU(1,1) $  such that
\begin{equation}
    L_\C =\set{\left(\begin{array}{ccc|ccc}
        &         &      & & &      \\
      \multicolumn{2}{c}{\smash{\raisebox{.5\normalbaselineskip}{$A$}}} &&& &\\
                 & & a & && b \\
      \hline \\[-\normalbaselineskip]
      & & &&& \\
      &&&  \multicolumn{2}{c}{\smash{\raisebox{.5\normalbaselineskip}{$A^{-T}$}}}&\\
      && c &&& d
    \end{array}\right)\Bigg| A\in\GL_2,\;\begin{pmatrix}
    a & b\\c & d
    \end{pmatrix}\in\SL_2}\subset G_{\C}.
\end{equation}
Then $L_\C $ is the Levi factor of a $\theta $-stable standard parabolic $Q\subset G_\C, $
which does not descend to $\R $. The absolute Weyl group $W (L, T) $
is given by
\begin{equation}
    W (L, T) =\set{(1, 1,\pm 1)}\times\set{(1), (1 2)}\subset W (G, T).
\end{equation}
Note that the local Arthur parameter factors as
\begin{equation}
    \psi_v: W_{F_v}\times\SL_2 (\C)\to {} ^ LL\to {} ^ LG.
\end{equation}
If $\operatorname{JL}(\pi_v )$
is the discrete series representation of $\GL_2 (\R) $
of weight $2k\geq 6, $
then, in the notation of \cite{adams1987endoscopic}, \begin{equation}\Pi_{\psi_v} =\set {A_{\mathfrak q_w} (w^{-1}\cdot (k-3,k-3,0))\,:\, w\in S}, \end {equation}
where $(k-3,k-3,0) $
is viewed as a character of $T $,\begin{equation}
    S = W (L, T)\backslash W (G, T)/W_\R (G, T),
\end{equation}
and, for $w\in S $, $Q_w $ is the $\theta $-stable parabolic subgroup $w ^ {-1} Qw $
of $G_\C $.

\subsubsection{}
Choose representatives
\begin{equation}
    w_0 = 1,\;\; w_1 = (-1, 1, 1),\;\; w_2 = (-1, -1, 1)
\end{equation}
of $S $ in $W (G, T) $, and label the elements of $\Pi_{\psi_v} $
as
\begin{equation}
    \Pi_{\psi_v, i} = A_{\mathfrak q_{w_i}} (w_i^{-1}\cdot (k-3,k-3,0)),\;\; 0\leq i\leq 2.
\end{equation}
  \begin{prop}
\label{prop: local cohomology calculation and Arthur pocket}
For any complex, irreducible algebraic representation $V $
of $G $,
$$H ^\ast (\mathfrak g, K;\Pi_{\psi_v, i}\otimes V) = 0 $$
unless $V$
is the highest weight representation $V_{(k-3,k-3,0)} .$ 
If $V = V_{(k-3,k-3,0)} $,
then the dimensions of the nonzero cohomology groups of the representations $\Pi_{\psi_v, i} $ are computed as follows.
$$
     \dim_\C H ^ {5,0} (\mathfrak g, K;\Pi_{\psi_v, 0}\otimes V) = \dim_\C H ^ {6,1} (\mathfrak g, K;\Pi_{\psi_v, 0}\otimes V) =1$$
     \begin{align*}
     \dim_\C H ^ {2,2} (\mathfrak g, K;\Pi_{\psi_v, 1}\otimes V)& = \dim_\C H ^ {4,4} (\mathfrak g, K;\Pi_{\psi_v, 1}\otimes V) =1\\
     \dim_\C H ^ {3,3} (\mathfrak g, K;\Pi_{\psi_v, 1}\otimes V)&= 2\\ \end{align*}
$$     \dim_\C H ^ {0,5} (\mathfrak g, K;\Pi_{\psi_v, 2}\otimes V)= \dim_\C H ^ {1,6} (\mathfrak g, K;\Pi_{\psi_v, 2}\otimes V) =1$$
\end{prop}
\begin{proof}
Write $Q_{w_i} = L_{i}U_{i}. $
The proposition follows
 from \cite[Proposition 6.19]{vogan1984unitary}, together with the calculations:
 \begin{align*}
     \mathfrak u_0=\left(\begin{array}{ccc|ccc}
        &         &  \ast    & \ast&\ast &\ast      \\
      \multicolumn{2}{c}{\smash{\raisebox{.5\normalbaselineskip}{$I$}}} &\ast&\ast&\ast &\ast\\
                 0&0 & 1 &\ast &\ast & 0\\
      \hline \\[-\normalbaselineskip]
      & & &&& 0\\
      &&&  \multicolumn{2}{c}{\smash{\raisebox{.5\normalbaselineskip}{$I$}}}&0\\
      &&  &\ast&\ast& 1
    \end{array}\right),\;\;\; & W_\R (L_0, T) =\set{1, (12)},\\
    \mathfrak u_1  =\left (\begin {array} {ccc | ccc} 1 & & & & &\\\ast & 1 &\ast & &\ast &\ast\\
    \ast & & 1 & &\ast &\\
          \hline \\[-\normalbaselineskip]
\ast & &\ast & &\ast &\\& & & & 1 &\\\ast & && &\ast & 1\end {array}\right), \;\;\;&W_\R (L_1, T) =\set {1},\\
\mathfrak u_2 =\left (\begin {array} {ccc | ccc}
& &0 & & &\\
  \multicolumn{2}{c}{\smash{\raisebox{.5\normalbaselineskip}{$I$}}} & 0& & &\\
  \ast &\ast & 1& & &\\
  \hline \\[-\normalbaselineskip]
  \ast &\ast &\ast & & &\ast\\
  \ast &\ast &\ast &  \multicolumn{2}{c}{\smash{\raisebox{.5\normalbaselineskip}{$I$}}} &\ast\\
  \ast &\ast & 0 &0 &0 & 1\end{array}\right),\;\;\;&W_\R (L_2, T) =\set {1, (1 2)}.
 \end{align*}
\end{proof}
\subsubsection{}
If $V = V_{(k-3,k-3,0)} $, then $V $ has a $K $-stable Hodge decomposition of weight $m =2k -6 $ defined by\begin{equation}
\label{equation: decomposition of V}
    V^ {p, q} =\set{v\in V: z\cdot v = z ^ {n- 2p}v,\;\forall z\in Z_K\simeq U (1)},\; p + q =m.
\end{equation}
The decomposition (\ref{equation: decomposition of V})
induces
a refined decomposition of the Lie algebra
 cohomology, cf. \cite{zucker1981locally}: $$H ^ {p, q} (\mathfrak g, K;\Pi_{\psi_v, i}\otimes V) =\bigoplus_{\substack{r +s = m\\r, s\geq 0}} H ^ {(p, q); (r, s)} (\mathfrak g, K;\Pi_{\psi_v, i}\otimes V). $$
 For later use, we now calculate this decomposition.
 \begin {prop}\label{crop: Hodge structure on why algebra cohomology}
 
 Let $V = V_{(k -3, k -3, 0)} $. The dimensions of the nonzero refined components of the cohomology groups of the representations $\Pi_{\psi_v, i} $ are computed as follows.
 \begin{align*}
     \dim_\C H ^ {(5, 0); (2k -6, 0)} (\mathfrak g, K;\Pi_{\psi_v, 0}\otimes V) &=\dim_\C H ^ {(6, 1); (2k -6, 0)} (\mathfrak g, K;\Pi_{\psi_v, 0}\otimes V=1
 \end{align*}
 \begin {align*}
 \dim_\C H ^ {(2, 2); (k -3, k -3)} (\mathfrak g, K;\Pi_{\psi_v, 1}\otimes V) &=\dim_\C H ^ {(4, 4); (k -3, k -3)} (\mathfrak g, K;\Pi_{\psi_v, 1}\otimes V) = 1\\
 \dim_\C H ^ {(3, 3); (k -3, k -3)} (\mathfrak g, K;\Pi_{\psi_v, 1}\otimes V) & = 2
 \end {align*}
 \begin {align*}
 \dim_\C H ^ {(0, 5); (0, 2k -6)} (\mathfrak g, K;\Pi_{\psi_v, 2}\otimes V) & =\dim_\C H^ {(1, 6); (0, 2k -6)} (\mathfrak g, K;\Pi_{\psi_v, 2}\otimes V) = 1
 \end {align*}
 \end{prop}
 \begin{proof}
 The proof is similar to \cite[Proposition 11.4]{ichino2018hodge}. Consider the set of coset representatives for $W_\R (G, T)\backslash W (G, T) $
 given by
 \begin{equation}
     W_0\coloneqq\set {w\in W (G, T): w ^ {-1} (\Delta_c ^ +)\subset\Delta ^ +},
 \end{equation}
 where $\Delta ^ + $ is the set of positive roots with respect to the Borel subgroup (\ref{EQ: Borel subgroup convention}) and $\Delta ^ +_c\subset\Delta ^ + $
 is the subset of compact roots (i.e. the positive roots of $K $). In the notation of \cite{zucker1981locally}, let $\mu $ and $\lambda $ be the highest characters of $Z_K $ appearing in the action of $K $ on $\mathfrak g_\C $ and $V $, respectively; identifying characters of $Z_K $ with $\Z $, $\mu = 2 $ and $\lambda = 2k -6 = m $
 (the weight of $V $). Also let $\Lambda = (k -3, k -3, 0) $, considered as a character of $T $.
 By \cite[\S 5]{zucker1981locally}, if  $$H ^ {(p, q); (n - p, m +p -n) }(\mathfrak g, K;\Pi_{\psi_v, i}\otimes V)\neq 0 $$
 for any $i $, then there exists $w\in W_0 $ with $\l (w) = p, $ such that the $K $-representation
 of highest weight $w(\Lambda +\rho) -\rho $
 has central character $\lambda - n\mu = 2k - 2n -6. $
 Here $\rho = (3, 2, 1) $ is the half sum of positive roots for $\mathfrak g_\C $. Thus each $w\in W_0 $
 defines exactly one choice of $p $ and $n $
 such that a nonzero contribution of bidegree $(p, q); (n - p, m + p - n) $ is possible. These choices are summarized in the following table.
 $$\begin{array}{c|c |c | c | c}
      w\cdot (a, b, c) & p =\l (w)& w (\Lambda +\rho) -\rho & n &\text{possible types}\\
      \hline
      (a, b, c) & 0 & (k -3, k -3, 0) & 0 & (0, q); (0, 2k -6)\\
     (a, b, - c) & 1 & (k -3, k -3, -2) & 1 & (1, q); (0, 2k -6)\\
     (a, c, - b) & 2 & (k -3, -1, - k) & k -1 & (2, q); (k -3, k -3)\\
     (b, c, - a) & 3 & (k -4, -1, - k -1) & k & (3, q); (k -3, k -3)\\
     (a, - c, - b) & 3 & (k -3, -3, - k) & k & (3, q); (k -3, k -3)\\
     (b, - c, - a) & 4 & (k -4, -3, - k -1) & k +1 & (4, q); (k -3, k -3)\\
     (c, - b, - a) & 5 & (-2, - k -1, - k -1) & 2k-1 & (5, q); (2k -6, 0)\\
     (- c, - b, - a) & 6 & (-4, - k -1, - k -1) & 2k & (6, q); (2k -6, 0)
 \end{array}$$
 Comparing with Proposition \ref{prop: local cohomology calculation and Arthur pocket}
 completes the proof.
 \end{proof}
 \subsection{Contributions to the cohomology of Shimura varieties}
\subsubsection{}\label{Subsubsection: non-tempered cohomology notation}
Consider the Shimura variety for $\boldgsp_6$ as in \S\ref{subsection: symplectic Shimura}.  Following the notation of (\ref{subsubsec:Plectic notation}), we obtain a local system $\mathcal V_{(\boldsymbol m -2, \boldsymbol m -2, 0)}$ of $\Q(\boldsymbol m)$-vector spaces.
Let $\sigma_{m_v} $ be the unique irreducible representation of $ K_{3,v} $ with trivial central character and whose restriction to $U (3) $ has highest weight $(m_v +1, 0, - m_v -1) $, and let $\boldsymbol\sigma_{\boldsymbol m} $ be the representation $\otimes_{v |\infty}\sigma_{m_v} $ of $\boldsymbol K_3. $
One calculates that
\begin{equation}
     \dim\Hom_{\boldsymbol K_3}\left (\boldsymbol\sigma_{\boldsymbol m}, V_{(\boldsymbol m-2,\boldsymbol m - 2,0),\C}\otimes\wedge ^ {\boldsymbol 2,\boldsymbol 2}\mathfrak p ^\ast_{\GSp_6}\right) = 1,
\end{equation}
where $(\boldsymbol 2,\boldsymbol 2) $ is the constant plectic Hodge type.
Thus we have, from (\ref{eqn:Realization cohomology definition}),
a class map
\begin{equation}
    \left (\mathcal A_{(2)} (\GSp_6 (\A_F))\otimes\boldsymbol\sigma_{\boldsymbol m}\right) ^ {\boldsymbol K_3}\to H ^ {\boldsymbol 2,\boldsymbol 2}_{(2)} (S (\boldgsp_6),{\mathcal V}_{(\boldsymbol m-2, \boldsymbol m - 2,0),\C}).
\end{equation}
\subsubsection{}
We now choose a totally indefinite, non-split quaternion algebra $B $ over $F $.
Let $\pi $ be an auxiliary  automorphic representation
of $PB^\times (\A_F) $
of weight $2\boldsymbol m +2 = (2m_v +2)_{v |\infty} $.
\begin{lemma}\label{Lemma: a source of the hajj teach classes}
If $\widetilde\Pi $ is a discrete automorphic representation of $\GSp_6 (\A_F) $ which is nearly equivalent to a constituent of the theta lift $\iota^\ast\Theta (\pi\boxtimes\mathbbm 1) $, then we have:
\begin{enumerate}
\item View $\mathcal V_{(\boldsymbol m -2,\boldsymbol m -2, 0),\C} $
as a variation of structures of weight 0. Then the $L ^ 2 $-cohomology
$$H_{(2)} ^ {4d} (S (\boldgsp_6),{\mathcal V}_{(\boldsymbol m-2, \boldsymbol m - 2,0),\C}) [\widetilde\Pi_f] $$
is purely of Hodge type $(2d, 2d) $ for the (non-plectic) Hodge structure on $L ^ 2 $ cohomology defined in \cite{zucker1981locally}.
\item The Galois group $\Gal (\overline\Q/F^c) $ acts on the intersection cohomology
$$IH ^ {4d} (S (\boldgsp_6)_{\overline\Q},{\mathcal V}_{(\boldsymbol m-2, \boldsymbol m - 2,0),\overline\Q_\l}) [\widetilde\Pi_f] $$
via $\chi ^ {- 2d} $, where $\chi $
is the $\l $-adic cyclotomic character.
\end {enumerate}
\end{lemma}
\begin{proof}
For
(1), it follows from Propositions \ref{Prop: very weak endoscopic classification}, \ref{prop: Arthur parrameter for the lift}, and \ref{prop: local cohomology calculation and Arthur pocket} that $$H_{(2)}^{4d} (S (\boldgsp_6),{\mathcal V}_{(\boldsymbol m-2, \boldsymbol m - 2,0),\C}) [\widetilde\Pi_f]\simeq \widetilde \Pi_f\otimes\bigotimes_{v|\infty} H ^ {2, 2} ( {\mathfrak \GSp_{6,\R}}, K_3;\pi_v ^ {\text {sm}}\otimes V_{(m_v-2,m_v - 2,0),\C}). $$ Thus it follows from Proposition \ref{crop: Hodge structure on why algebra cohomology} and from \cite{zucker1981locally} that the cohomology is pure of Hodge type $(m, m) $ in Zucker's normalization,
where $m =\sum m_v $.
However, since $ V_{\boldsymbol m,\C} $ has trivial central character and hence total weight 0 in the algebraic normalization, we must twist by $(2d-m,2d-m)$, which shows the claim.
For (2), choose a compact open subgroup $K =\prod K_v\subset\GSp_6(\A_{F, f}) $.
It suffices to show that $\Frob_p $ acts trivially on $$H\coloneqq IH ^ {4d} (S _K(\boldgsp_6)_{\overline\Q}, {\mathcal V}_{(\boldsymbol m-2,\boldsymbol m - 2,0),\overline\Q_\l}) [\widetilde\Pi_f]$$
for almost all $p $ such that $p $ splits completely in $F^c $ and $K_p $ is hyperspecial. Assume without loss of generality that $\widetilde\Pi_p $
is unramified with local Langlands parameter $$\xphi_p: W_{\Q_p}\to {} ^ L\boldgsp_6 =\operatorname{GSpin}_7(\C) ^ d\times W_{\Q_p}, $$ and consider the $8d $-dimensional representation  
defined by the composite:
\begin{equation}\label{Equation: complex composite}W_{\Q_p}\xrightarrow{\xphi_p}\operatorname{GSpin}_7(\C)^d\xrightarrow{r_{\text {spin}} ^ {\otimes d}}\GL_{8d}(\C). \end{equation}
After picking an isomorphism $\C\cong \overline \Q_\l$, (\ref{Equation: complex composite})
defines an $8d $-dimensional $\l $-adic unramified local Galois representation $V_p $.
By \cite[\S2]{lee2022semisimplicity}, since $p $ splits in $F (\boldsymbol m) $, the action of the geometric Frobenius $\Frob_p ^ {-1} $ on $H $
satisfies the characteristic polynomial of $p ^ 3\Frob_p ^ {-1} $ on $V_p $. Now by Proposition \ref{prop: Arthur parrameter for the lift}, for almost every such $p $
the representation $V_p $ is given by
$$\bigotimes_{v | p}\left(\overline\Q_\l (-1)\oplus\overline\Q_\l ^ 2\oplus\overline\Q_\l (1)\oplus\rho_\pi |_{F_v}\oplus\rho_\pi |_{F_v} (1)\right), $$
where $\rho_\pi $ is the $2 $-dimensional $\l $-adic Galois representation associated to $\pi $, which we normalize to have weight one. (Recall that $\rho_\pi $ is pure since $\pi $
is discrete series at a finite place of $F $ \cite{ohta1983zeta}.)
On the other hand, it is known that $H $ is pure of weight $4d $; comparing with the weights in $V_p $, it follows that $\Frob_p $ acts as $p ^ {- 2d} $
on $H $.
\end{proof}

\section{Triple product periods}\label{sec: triple product periods}
\subsection{The vector-valued period problem}
\subsubsection{}
Let $\boldsymbol m $, $\pi_1, $ $\pi_2, $
$\boldsymbol\epsilon $,  $\boldsymbol\tau ^ {\boldsymbol\epsilon}_{\boldsymbol m} $, and $\boldsymbol\sigma_{\boldsymbol m} $
be as in (\ref{subsubsec: notation for endoscopic lifts in coh}) and (\ref{Subsubsection: non-tempered cohomology notation}), and let $B $
be a \emph{non-split} totally indefinite quaternion algebra over $F $, ramified
at a set $S $ of places of $F $
at which $\pi_i $
are both discrete series.
\subsubsection{}
For auxiliary automorphic representations $\pi $ of $PB (\A_F) ^\times $
of weight  $2\boldsymbol m +2, $
we will consider triple product period integrals of $\Theta (\pi\boxtimes\mathbbm 1)$ along the subgroup
\begin{equation}
    \widetilde H\coloneqq (\GSp_4\times_{\mathbb G_m}\GL_2)\subset\GSp_6.
\end{equation}
The maximal compact-modulo-center subgroup of $\widetilde H (F\otimes\R)$
is
\begin{equation}
    (\boldsymbol K_2\times\boldsymbol K_1)_0 \coloneqq(\boldsymbol K_2\times\boldsymbol K_1)\intersection \widetilde H (F\otimes\R).
\end{equation}
 To define the vector-valued period integral, note that (by the classical branching law for unitary groups), the space
 \begin{equation}
     \Hom_{(\boldsymbol K_2\times\boldsymbol K_1)_0} (\boldsymbol\sigma_{\boldsymbol m}\otimes\boldsymbol\tau ^ {\boldsymbol\epsilon}_{\boldsymbol m}\otimes\boldsymbol\chi ^\check_{-\boldsymbol\epsilon\boldsymbol m},\C)
 \end{equation}
 is one-dimensional, say with generator $\l $.
 
We then define, for any auxiliary representation $\pi$ of $PB (\A_F) ^\times $
of weight $2\boldsymbol m +2, $
the triple product period:
\begin{equation}\label{Equation: definition of period integral triple product}
\begin{split}
    \widetilde {\mathcal P}_{S,\pi_1,\pi_2,\pi}^{\boldsymbol\epsilon} (\alpha, \beta,\gamma) &= \integral_{[\widetilde H]}\l (\alpha (g, g')\otimes\beta (g)\otimes\gamma (g'))\d (g, g')\neq 0,\\
    &\alpha\in (\Theta (\pi\boxtimes\mathbbm 1)\otimes\boldsymbol\sigma_{\boldsymbol m}) ^ {\boldsymbol K_3},\;\;\beta\in (\Pi_S (\pi_1\otimes\pi_2)\otimes\boldsymbol\tau ^ {\boldsymbol\epsilon}_{\boldsymbol m}) ^ {\boldsymbol K_2},\;\;\gamma\in (\pi_2 ^\check\otimes\boldsymbol \chi ^\check_{-\boldsymbol\epsilon\boldsymbol m}) ^ {\boldsymbol K_1}.
    \end{split}
\end{equation}
 Since we will not give a precise formula for $\widetilde {\mathcal P} ^ {\boldsymbol\epsilon}_{S,\pi_1,\pi_2,\pi} $
 and are only interested in its non-vanishing, we ignore the problem of normalization. The non-vanishing of (\ref{Equation: definition of period integral triple product}), for a good choice of $\pi $, is the key input to the non-vanishing of the Hodge classes we construct in the next section.
 \subsubsection{}
 The strategy for calculating (\ref{Equation: definition of period integral triple product}) is to use the seesaw diagram:
 \begin{center}
    \begin{tikzcd}
    \GSO(V_B)\times_{\mathbb G_m} \GSO(V_B) \arrow[d,dash]\arrow[dr,dash] &  \GSp_6 \arrow[d,dash]\arrow[dl,dash] \\\GSO (V_B )&\widetilde H 
    \end{tikzcd}
\end{center}
There are two main inputs to the non-vanishing of our period integral (for a good choice of $\pi$): the first is  a vector-valued version of the usual global seesaw identity, and the second is a non-vanishing result for the vector-valued theta lifts along the ``other'' diagonal in the seesaw diagram, i.e. from $\GSp_4 $ and $\GL_2 $
to $\GSO (V_B) $.
\subsection{Vector-valued seesaw identity}\label{subsection: vector seesaw}
\subsubsection{}
Continuing with the notation from (\ref{subsubsec: arch Conventions part one}), let $m \geq 2$ be an integer and let $\sigma_m $
be the unique representation of $K_3 $
of trivial central character and whose restriction to $U(3)$ has highest weight $(m + 1,0,-m-1)$. Let $$\widetilde\fee_m\in (S_{\leq 2m+2}^0(3)\otimes \sigma_m\otimes \chi_{-(2m+2),0})^{(K_3\times L)_0}$$ be a generator, which makes sense by Proposition \ref{prop: substitute for Harris}. If $(K_2\times K_1)_0$ is the intersection of $K_3$ with $\GSp_{4,\R}\times_{\mathbb G_m} \GL_{2,\R}$ inside $\GSp_{6,\R}$, then we have, for any $\epsilon = \pm 1,$
\begin{equation}\label{eq: branching law}\dim \Hom_{(K_2\times K_1)_0} (\sigma_m, \tau_m^{\epsilon}\otimes \chi_{-\epsilon m}^\vee) = 1;\end{equation}
let $\l_\epsilon$ denote a generator. Also let $(K_1\times K_2\times L)_0 = (K_1\times K_2)_0\times L \cap (K_3\times L)_0.$
\begin{prop}\label{Prop: local seesaw}
The Schwartz function $$\l_\epsilon(\widetilde\fee_m) \in \left(\mathcal S_{\R}(\langle e_2, e_4, e_6\rangle \otimes V) \otimes \tau_m^\epsilon\otimes \chi_{-\epsilon m}^\vee\otimes \chi_{2m+2,0}\right)^{(K_1\times K_2\times L)_0}$$ is a nonzero scalar multiple of the tensor product $$\phi_m^\epsilon \otimes \xphi_m^{\epsilon}\in \left(\left(\mathcal S_{\R}(\langle e_2,e_4\rangle \otimes V)\otimes \tau_m^\epsilon\otimes \chi^\vee_{m+2,-\epsilon m}\right)^{(K_2\times L)_0}\otimes \left(\mathcal S_{\R}(\langle e_6\rangle \otimes V)\otimes \chi^\vee_{-\epsilon m}\otimes \chi_{-m,\epsilon m}\right)^{(K_1\times L)_0}\right)\bigg|_{(K_1\times K_2\times L)_0}.$$
\end{prop}
\begin{proof}
Assume $\epsilon = + $; the other case is similar.
Let $\xphi\in S_{\leq 2m + 2}^0(3)$
be the contraction of $\phi_m^\epsilon \otimes \xphi_m^{\epsilon}$ with any nonzero vector. Then $\xphi$ generates the irreducible $U (2)\times U(1) $-representation of highest weight $(1,-m-1, m) $, and it suffices to show that $U(3)\cdot\xphi$  is the irreducible representation of highest weight $(m +1, 0, - m -1) $. Without loss of generality, assume $U (3)\cdot\xphi$ 
is irreducible, say with highest weight $(a,b,c)$. 
It follows (using the branching law for unitary groups)
that
\begin{equation}
    \label{equation: branching law inequality}
    \begin{split}
    a\geq 1\geq b\geq - m -1\geq c,\\
    a+b + c  = 0.
    \end{split}
\end{equation}
On the other hand, considering  
 Proposition \ref{prop: substitute for Harris}(2), we have \begin{equation}\label{equation: Degree any quality}
     |a|+ |b| + |c| \leq 2m + 2.
 \end{equation} 
 The combination of (\ref{equation: branching law inequality}) and (\ref{equation: Degree any quality})
 force $( a, b, c) = (m +1, 0, - m -1) $.
\end{proof}
\subsubsection{}
We now return to the global situation. 
Choose an isomorphism $V_B\otimes_F\R\simeq V\otimes_F\R $, which induces an isomorphism $\GSO(V_B)(F\otimes \R) \simeq \GSO(V) (F\otimes\R) $. Then let $\boldsymbol L = \prod_{v|\infty} L\subset  \GSO(V_B) (F\otimes \R)$, and similarly for $(\boldsymbol K_n\times\boldsymbol L)_0$, etc. We
fix vector-valued Schwartz functions as follows:
\begin{equation}
     \begin{split}
         \boldsymbol\phi ^ {-    \boldsymbol\epsilon}_{\boldsymbol m}= \otimes_{v |\infty}\fee ^ {-\epsilon_v}_{m_v}&\in\left (\mathcal S_{F\otimes\R} (\langle e_2, e_4\rangle\otimes V) \otimes\boldsymbol\tau_{\boldsymbol m} ^ {-\boldsymbol\epsilon}\otimes\boldsymbol\chi ^\check_{(\boldsymbol m +2),\boldsymbol\epsilon\boldsymbol m}\right) ^ {(\boldsymbol K_2\times\boldsymbol L)_0}\text{    \;\;     (Proposition \ref{prop: one-dimensional harmonic})}\\
         &\simeq \left (\mathcal S_{F\otimes\R} (\langle e_2, e_4\rangle\otimes V_B) \otimes\boldsymbol(\tau_{\boldsymbol m} ^ {\boldsymbol\epsilon})^\check\otimes\boldsymbol\chi_ {-(\boldsymbol m +2),-\boldsymbol\epsilon\boldsymbol m}\right) ^ {(\boldsymbol K_2\times\boldsymbol L)_0}\\
         \boldsymbol\xphi^ {-\boldsymbol\epsilon}_{\boldsymbol m}=\otimes_{v |\infty}\xphi_{m_v} ^ {-\epsilon_v}&\in\left (\mathcal S_{F\otimes\R} (\langle e_6\rangle\otimes V)\otimes\boldsymbol\chi_{-\boldsymbol\epsilon\boldsymbol m}\otimes\boldsymbol\chi ^\check_{\boldsymbol m, -\boldsymbol\epsilon\boldsymbol m}\right) ^ {(\boldsymbol K_1\times\boldsymbol L)_0}\text {     \;\;    (Proposition \ref{prop: S02 covariance})} \\
         &\simeq \left(\mathcal S_{F\otimes \R}(\langle e_6\rangle\otimes V_B)\otimes \boldsymbol\chi_{-\boldsymbol\epsilon\boldsymbol m}\otimes\boldsymbol\chi ^\check_{\boldsymbol m, -\boldsymbol\epsilon\boldsymbol m}\right) ^ {(\boldsymbol K_1\times\boldsymbol L)_0}\\
         \widetilde{\boldsymbol\fee}_{\boldsymbol m}= \otimes_{v|\infty}\widetilde\fee_{m_v}&\in \left(\mathcal S_{\R}(\langle e_2,e_4,e_6\rangle V)\otimes \boldsymbol\sigma_{\boldsymbol m}\otimes\boldsymbol\chi_{-(2\boldsymbol m + 2),0}\right)^{(\boldsymbol K_2\times\boldsymbol L)_0}\\
         &\simeq \left(\mathcal S_{\R}(\langle e_2,e_4,e_6\rangle\otimes V_B)\otimes \boldsymbol\sigma_{\boldsymbol m}\otimes\boldsymbol\chi_{-(2\boldsymbol m + 2),0}\right)^{(\boldsymbol K_2\times\boldsymbol L)_0}
     \end{split}
 \end{equation}
\begin{prop}\label{prop: global seesaw}
Let $\l$ be as above.
For all $$\alpha \in\left (\mathcal A_0 (\PGSO (V_B) (\A_F))\otimes\boldsymbol\chi_{2\boldsymbol m +2,0}\right) ^ {\boldsymbol L},\;\;\beta\in
\left (\Pi_S (\pi_1,\pi_2)\otimes\boldsymbol\tau_{\boldsymbol m} ^ {\boldsymbol\epsilon}\right) ^ {\boldsymbol K_2},\;\;\gamma\in\left (\pi_2 ^\check\otimes\boldsymbol \chi_{-\boldsymbol\epsilon\boldsymbol m}^\vee\right) ^ {\boldsymbol K_1},$$
$$\xphi_{1,f}\in\mathcal S_{\A_{F, f}} (\langle e_2, e_4\rangle\otimes V_B),\;\;\xphi_{2, f}\in\mathcal S_{\A_{F, f}} (\langle e_6\rangle\otimes V_B),$$
and up to a nonzero scalar depending on the normalizations, we have the identity:
\begin{equation}
\begin {split}\int_{[Z_{\widetilde H}\backslash\widetilde H]}\l\left (\theta_{\xphi_{1, f}\otimes\xphi_{2, f}\otimes\widetilde {\boldsymbol\fee}_{\boldsymbol m}} (\alpha)(g, g')\otimes\beta (g)\otimes\gamma (g')\right) \d (g,g')=\\\integral_{[\PGSO(V_B)]}\alpha (g)\theta_{\xphi_{1, f}\otimes\boldsymbol\fee_{\boldsymbol m} ^ {-\boldsymbol\epsilon}} (\beta) (g)\theta_{\xphi_{2, f}\otimes\boldsymbol\xphi_{\boldsymbol m} ^ {-\boldsymbol\epsilon}} (\gamma) (g)\d g.\end{split}\end {equation}
\end{prop}
\begin{proof}
This is formal from Proposition \ref{Prop: local seesaw}
and the usual seesaw identity, i.e. exchanging the order of integration.
\end{proof}
\subsection{Proof of the non-vanishing result}
\begin{prop}
\label{prop: other seesaw nonzero}
Let $\pi_i ^ B $
be the Jacquet-Langlands transfers of $\pi_i $
to $B (\A_F) ^\times $.
\begin {enumerate}
\item The map $$\theta_{\boldsymbol\fee ^ {-\boldsymbol\epsilon}_{\boldsymbol m}}:\mathcal S_{\A_{F, f}} (\langle e_2, e_4\rangle\otimes V_B)\otimes (\Pi_S (\pi_1,\pi_2)\otimes\boldsymbol\tau ^ {\boldsymbol\epsilon}) ^ {\boldsymbol K_2}\to\left (\mathcal A (\GSO (V_B)) (\A_F)\otimes\boldsymbol\chi_{- (\boldsymbol m +2), -\boldsymbol\epsilon\boldsymbol m}\right) ^ {\boldsymbol L}, $$
defined by $$(\XP,\alpha)\mapsto\theta_{\XP\otimes\boldsymbol\fee ^ {-\boldsymbol\epsilon}_{\boldsymbol m}} (\alpha), $$
has image containing $\left((\pi_1 ^ B\boxtimes\pi_2 ^ B)\otimes\boldsymbol\chi_{- (\boldsymbol m +2), -\boldsymbol\epsilon\boldsymbol m}\right) ^ {\boldsymbol L}.$
\item The map $$\theta_{\boldsymbol\XP ^ {-\boldsymbol\epsilon}_{\boldsymbol m}}:\mathcal S_{\A_{F, f}} (\langle e_6\rangle\otimes V_B)\otimes\left (\pi_2 ^\check\otimes\boldsymbol\chi_{-\boldsymbol\epsilon\boldsymbol m} ^\check\right) ^ {\boldsymbol K_1}\to\left (\mathcal A (\GSO (V_B)) (\A_F)\otimes \boldsymbol\chi ^\check_{\boldsymbol m, -\boldsymbol\epsilon\boldsymbol m}\right) ^ {\boldsymbol L}, $$
defined by $$(\XP,\alpha)\mapsto\theta_{\XP\otimes\boldsymbol\XP_{\boldsymbol m} ^ {-\boldsymbol\epsilon}} (\alpha), $$
has image containing $\left (((\pi_2 ^ B) ^\check\boxtimes (\pi_2  ^ B) ^\check)\otimes\boldsymbol\chi ^ {\check}_{\boldsymbol m, -\boldsymbol\epsilon\boldsymbol m}\right) ^ {\boldsymbol L}. $
\end{enumerate}
\end{prop}
\begin{proof}
In the general setup of \S\ref{section: Similitude theta lifting}, suppose $\Theta_{V,W}(\pi) = \Pi$ for cuspidal  automorphic representations $\pi$ of $G(V)(\A_F)$ and $\Pi$ of $H(W)(\A_F)$.
Then by definition we have a surjective composite\begin {equation}
\mathcal S_{\A_F} (W_2\otimes V)\xrightarrow{\theta}\mathcal A (R_0 (\A_F))\twoheadrightarrow \Pi\otimes \pi^\vee.
\end {equation}
Now, the theta kernel satisfies
 $$\theta (\XP) (g, h) = \overline{\theta (\overline\XP)}\left (\begin {pmatrix} 1 & 0\\0 & -1\end {pmatrix} g\begin {pmatrix} 1 & 0\\0 & -1\end {pmatrix}, h\right)$$ (cf. \cite{roberts1996theta}),
 so we deduce that $\overline{\Pi\otimes\pi ^\vee}=\Pi ^\check\otimes\pi $
 also appears in the spectrum of the theta kernel. (Recall that the central characters of $\Pi $
 and $\pi $
 must agree since the central character of the Weil representation is trivial.) In particular $\Theta_{W, V} (\Pi)$
  contains the nonzero irreducible constituent $\pi $.
 For (1), take $W = W_4, $
 $V = V_B, $
 and $\Pi =\Pi_S (\pi_1,\pi_2) $.   As in the proof of Proposition \ref{prop: local global compatibility}, the global theta lift gives rise to a nontrivial map:
 \begin {equation}\label {Equation: map from global lift}\mathcal S_{\A_{F}} (\langle e_2, e_4\rangle\otimes V_B) \twoheadrightarrow\Pi_S (\pi_1,\pi_2)^\vee \otimes \Theta_{W_4, V_B} (\Pi_S (\pi_1,\pi_2)).\end{equation}
(Since $\GSO(V_B)$ is anisotropic, all theta lifts are square-integrable.) The map (\ref{Equation: map from global lift}) is a restricted tensor product of local maps. To prove the proposition, it suffices to show that, for all $v |\infty $ and for some vector $0\neq\l\in\tau_{m_v} ^ {\epsilon_v} $, the contraction $\l(\fee_{m_v} ^ {-\epsilon_v} )$
has nontrivial image under the local component
 \begin{equation}\label{eq:Local component of theta lift}
     \mathcal S_{F_v} (\langle e_2, e_4\rangle\otimes V_B)\twoheadrightarrow \Pi ^ + (\pi_{1, v},\pi_{2, v})^\vee\otimes \left (\pi_{1, v} ^ B\boxtimes\pi_{2, v} ^ B\right)
 \end{equation} 
 of (\ref{Equation: map from global lift}).
 Now, by local Howe duality  for the disconnected similitude group $\operatorname{GO}(V_B) $, the local theta lift $\Theta_{W_4, V_B} (\Pi ^ + (\pi_{1, v},\pi_{2, v}) $
 is irreducible when viewed as a representation of $\operatorname{GO}(V_B)  (F_v)$ \cite{howe1989transcending,roberts1996theta}. Hence $$\Theta_{W_4, V_B} (\Pi ^ + (\pi_{1, v},\pi_{2, v}) =  \pi_{1, v} ^ B\boxtimes\pi_{2, v} ^ B\oplus \pi_{2,v}^B\boxtimes \pi_{1,v}^B,$$
 and the map (\ref{eq:Local component of theta lift}) factors as 
 $$\mathcal S_{F_v} (\langle e_2, e_4\rangle\otimes V_B)\twoheadrightarrow \Pi ^ + (\pi_{1, v},\pi_{2, v})^\vee\otimes \left (\pi_{1, v} ^ B\boxtimes\pi_{2, v} ^ B\oplus \pi_{2,v}^B\boxtimes \pi_{1,v}^B\right) \twoheadrightarrow \Pi ^ + (\pi_{1, v},\pi_{2, v})^\vee\otimes \left (\pi_{1, v} ^ B\boxtimes\pi_{2, v} ^ B\right).$$
We now note that $\l(\fee_{m_v}^{\epsilon_v})$ is a harmonic in the sense of \cite{howe1989transcending} by Proposition \ref{prop: substitute for Harris},
 and generates a $U(2)$-type that appears in $\Pi ^ + (\pi_{1, v},\pi_{2, v})^\vee.$ It follows from \cite{howe1989transcending}  that its image under 
 $$\mathcal S_{F_v} (\langle e_2, e_4\rangle\otimes V_B)\twoheadrightarrow \Pi ^ + (\pi_{1, v},\pi_{2, v})^\vee\otimes \left (\pi_{1, v} ^ B\boxtimes\pi_{2, v} ^ B\oplus \pi_{2,v}^B\boxtimes \pi_{1,v}^B\right)$$
 is nontrivial. However,  $\l(\fee_{m_v}^{\epsilon_v})$ generates the $L$-type $\chi_{-(m_v+2),\epsilon_vm_v}^\vee$, which does not appear in $\pi_{2,v}^B\boxtimes \pi_{1,v}^B$. It follows that $\l(\fee_{m_v}^{\epsilon_v})$ has nonzero image under (\ref{eq:Local component of theta lift}). This proves (1).
 The proof of (2)
 is analogous.
\end{proof}
Finally we come to the main result of this section:
 \begin{lemma}\label{llama: non-vanishing input}
 There exists an automorphic representation $\pi $
 of $PB (\A_F) ^\times $
 of weight $2\boldsymbol m +2 $,
 along with vectors
 $$\alpha\in (\Theta (\pi\boxtimes\mathbbm 1)\otimes\boldsymbol\sigma_{\boldsymbol m}) ^ {\boldsymbol K_3},\;\;\beta\in (\Pi_S (\pi_1\otimes\pi_2)\otimes\boldsymbol\tau ^ {\boldsymbol\epsilon}_{\boldsymbol m}) ^ {\boldsymbol K_2},\;\;\gamma\in (\pi_2 ^\check\otimes\boldsymbol \chi ^\check_{-\boldsymbol\epsilon\boldsymbol m}) ^ {\boldsymbol K_1}, $$
 such that:
 $$  \widetilde {\mathcal P}_{S,\pi_1,\pi_2,\pi}^{\boldsymbol\epsilon} (\alpha, \beta,\gamma) = \integral_{[\widetilde H]}\l (\alpha (g, g')\otimes\beta (g)\otimes\gamma (g'))\d (g, g')\neq 0. $$
 \end{lemma}
 \begin{proof}
 First, fix newforms
 $$f_1\in \pi_1^B,\;\; f^ {\boldsymbol\epsilon}_2\in\pi_2 ^ B,\;\; f_2 ^\check\in (\pi_2 ^ B) ^\check,\;\; (f_2^{\boldsymbol\epsilon})^\check\in (\pi_2 ^ B) ^\check $$
 of weights $\boldsymbol m +2, $
 $\boldsymbol\epsilon\boldsymbol m $,
 $\boldsymbol m$,
 and $-\boldsymbol\epsilon\boldsymbol m $, respectively.
 Then Proposition \ref{prop: other seesaw nonzero}
 implies that we may
 choose  vectors $$\beta\in\left (\Pi_S (\pi_1,\pi_2)\otimes\boldsymbol\tau ^ {\boldsymbol\epsilon}_{\boldsymbol m}\right) ^ {\boldsymbol K_2},\;\;\gamma\in (\pi_2 ^\check\otimes\boldsymbol \chi ^\check_{-\boldsymbol\epsilon\boldsymbol m}) ^ {\boldsymbol K_1}$$
 and  Schwartz functions
 $$
\XP_{1, f}\in\mathcal S_{\A_{F, f}} (\langle e_2, e_4\rangle\otimes V_B),\;\;\XP_{2, f}\in\mathcal S_{\A_{F, f}} (\langle e_6\rangle\otimes V_B) $$
 such that:
 \begin{equation}
     \theta_{\XP_{1, f}\otimes\boldsymbol\fee ^ {-\boldsymbol\epsilon}_{\boldsymbol m}} = f_1\otimes f_2  ^ {\boldsymbol\epsilon};\;\;\;\theta_{\XP_{2, f}\otimes\boldsymbol\fee ^ {-\boldsymbol\epsilon}_{\boldsymbol m}} (\gamma) = f_2 ^\check\otimes (f_2 ^ {\boldsymbol\epsilon}) ^\check.
 \end{equation}
 Now, the automorphic function $g\mapsto f_1 (g)\cdot f_2 ^\check (g) $
 corresponds to a Hilbert modular form on $B ^\times $
 of weight $2\boldsymbol m+2 $
 and trivial central character. We may therefore choose some automorphic representation $\pi $
 of $PB (\A_F) ^\times $ of weight $2\boldsymbol m +2, $
 with a vector $\alpha_0$ of weight $-(2\boldsymbol m +2) $,
 such that $$\integral_{[PB ^\times]} \alpha_0(g) f_1 (g)f_2^\check(g)\d g \neq 0.$$
 Now, we turn $\alpha_0 $
 into an automorphic form $\alpha$ on $\PGSO(V_B)(\A_F)$ by setting $\alpha(\boldsymbol p_Z(g_1,g_2)) = \alpha_0(g_1)$. 
 It is clear that $\alpha $
 is a vector in $\left ((\pi\boxtimes\mathbbm 1)\otimes\boldsymbol\chi_{2\boldsymbol m +2, 0}\right) ^ {\boldsymbol L}$.
 Then Proposition \ref{prop: global seesaw} allows us to compute:
 $$\widetilde {\mathcal P}_{S,\pi_1,\pi_2,\pi} (\theta_{\XP_{1, f}\otimes\XP_{2, f}\otimes\widetilde {\boldsymbol\fee}_{\boldsymbol m}} (\alpha),\beta,\gamma) =\left (\integral_{[PB ^\times]}\alpha_0 (g) f_1 (g) f_2^\check (g)\d g\right)\cdot\left (\int_{[PB ^\times]} f_2^ {\boldsymbol\epsilon} (g) (f_2 ^ {\boldsymbol\epsilon}) ^\check (g)\d g\right)\neq 0. $$
 \end{proof}
\section{Proof of main result: Hodge classes in the non-generic case}\label{sec: proof Hodge}
\subsection{Construction}
\subsubsection{}\label{subsection: non-generic main result}
Consider the inclusions of Shimura varieties:
\begin{equation}\label{EQ: the three Shimura varieties for Hodge class}
    S (\boldgsp_6)\xleftarrow{\iota_1}S ( \widetilde {\boldsymbol H})\xrightarrow{\iota_2}S 
    (\boldgsp_4)\times S (\boldgl_2),
\end{equation}
where $$\widetilde H\coloneqq\GSp_{4}\times_{\mathbb G_m}\GL_{2}\subset\GSp_{6} . $$
Note that $\iota_1 ^\ast\mathcal V_{(\boldsymbol m -2,\boldsymbol m - 2,0)} $ contains $\iota_2 ^\ast\mathcal W_{\boldsymbol m} $ as a direct factor with multiplicity one. Since $\iota_2 $ is an open and closed embedding at sufficiently small level, one obtains from (\ref{EQ: the three Shimura varieties for Hodge class}) a map
\begin{equation}
    \label{EQ: back from big Shimura righty cohomology}
    H ^ i (S (\boldgsp_6), {\mathcal V}_{(\boldsymbol m-2,\boldsymbol m -2, 0)})\to H ^ i (S (\boldgsp_4)\times S (\boldgl_2),\mathcal W_{\boldsymbol m}).
\end{equation}
The Hodge classes we construct will be the images of  classes on $S (\boldgsp_6) $
under the map (\ref{EQ: back from big Shimura righty cohomology}).
\subsubsection{}
Let $\pi_1,\pi_2, $ and $\Pi_{S_f} $ be as in (\ref{subsubsec: notation for endoscopic lifts in coh}), where $| S_f |\geq 2 $ is even. We let $B $
be the unique quaternion algebra over $F $ which is ramified at $S_f $ and split at all archimedean places.
For any finite set $\Sigma\supset S_f $
of places of $F $, including all infinite ones, we consider the unramified Hecke algebra with $\Q $-coefficients:\begin {equation}\widetilde {\mathbb T} ^\Sigma =\otimes_{v\not\in\Sigma}\mathcal H(\GSp_6 (F_v),\GSp_6 (\mathcal O_v)).\end{equation}
For an auxiliary automorphic representation $\pi $ of $PB ^\times $ which is tempered, unramified outside of $\Sigma $, and of weight $2\boldsymbol m +2, $
the Hecke action on $\Theta (\pi\boxtimes\mathbbm 1) $
defines a maximal ideal $I ^\Sigma\subset\widetilde {\mathbb T} ^\Sigma $.
\begin{definition}\label{definition: Hodge Tate subspacce}
Fix $\pi $ and $\Sigma $ as above, a sufficiently small compact open subgroup $K=\prod K_v\subset\GSp_6 (\A_{F_f}) $ such that $K_v =\GSp_6 (\O_v) $ for $v\not\in\Sigma $, and a coefficient field $E \supset\Q (\boldsymbol m)$
over which $\Pi_{S_f} $ is defined. Then we define $$\operatorname{Hdg}_E (\pi, K,\Sigma) \subset H ^ {4d} (S (\boldgsp_4)\times S (\boldgl_2),\mathcal W_{\boldsymbol m, E})(2d)\left[\Pi^\vee_{S_f}\boxtimes\pi_2 \right]$$
to be the image of the Tate twist of the composite map
\begin{equation*}
\begin{split}
IH ^ {4d} (S _K(\boldgsp_6) ^\ast,\mathcal V_{(\boldsymbol m-2,\boldsymbol m -2,0), E}) [I ^\Sigma]\to H ^ {4d} (S (\boldgsp_6),\mathcal V_{(\boldsymbol m-2,\boldsymbol m -2,0), E})\xrightarrow{(\ref{EQ: back from big Shimura righty cohomology})}H ^ {4d} (S (\boldgsp_4)\times S (\boldgl_2),\mathcal W_{\boldsymbol m, E})\\\twoheadrightarrow H_! ^ {4d} (S (\boldgsp_4)\times S (\boldgl_2),\mathcal W_{\boldsymbol m, E})\left [\Pi_{S_f}^\vee\boxtimes\pi_2 \right]. 
\end {split}
\end {equation*}
\end {definition}
\begin{lemma}\label{Lemma: get Hodge take classes}
Any $\xi\in\Hodge_E (\pi, K,\Sigma) $
is a Hodge class of weight $(0, 0) $. Moreover, the image of $\xi $ in
$$H_{\et,!} ^ {4d}\left ((S (\boldgsp_4)\times S (\boldgl_2)_{\overline\Q},\mathcal W_{\boldsymbol m, E_\lambda}\right) (2d)\left [\Pi_{S_f}^\vee\boxtimes\pi_{2, f} \right] $$
is $\Gal(\overline\Q/F^c ) $-invariant, where $\lambda $ is any finite prime of $E $.
\end{lemma}
\begin{proof}
Consider the following commutative diagram:
\begin{center}
    \begin{tikzcd}
       H_{(2)} ^ {4d} (S_K (\boldgsp_6),\mathcal V_{(\boldsymbol m-2,\boldsymbol m -2,0),\C}) (2d)[I ^\Sigma]\arrow [r]\arrow [d] & H ^ {4d} (S_K (\boldgsp_6),\mathcal V_{(\boldsymbol m-2,\boldsymbol m -2,0),\C})(2d)\arrow [d]\\
        H_{(2)} ^ {4d} (S (\boldgsp_4)\times S (\boldgl_2),\mathcal W_{\boldsymbol m, \C})(2d)\left [\Pi_{S_f}^\vee\boxtimes\pi_2 \right]\arrow[r, "\sim"] & H_! ^ {4d} (S (\boldgsp_4)\times S (\boldgl_2),\mathcal W_{\boldsymbol m, \C})(2d)\left [\Pi^\vee_{S_f}\boxtimes\pi_2 \right].
    \end{tikzcd}
\end{center}
The left and bottom arrows in this diagram are maps of pure Hodge structures, and by construction $\Hodge_E (\pi, K,\Sigma)\otimes_E\C $ is the image of the composite from the top left to the bottom right. We wish to show that every element of $ H_{(2)} ^ {4d} (S_K (\boldgsp_6),\mathcal V_{(\boldsymbol m-2,\boldsymbol m -2,0),\C}) (2d)[I ^\Sigma]$ has weight $(0, 0) $. If $\widetilde\Pi $ is an automorphic representation of $\GSp_6 (\A_F) $
such that $\widetilde\Pi ^ K_f\neq 0 $
is annihilated by $I ^\Sigma$, then there exists
a Galois element $\sigma\in\Gal (\overline\Q/\Q) $
such that $\widetilde\Pi $ is nearly equivalent to a constituent of $\Theta (\pi\boxtimes\mathbbm 1) ^\sigma $. Since the spherical theta correspondence is defined over $\Q $, it follows that $\widetilde\Pi $ is nearly equivalent to a constituent of $\Theta (\pi ^\sigma\boxtimes\mathbbm 1) $.
Thus, by  Lemma \ref{Lemma: a source of the hajj teach classes}, $\Hodge_E (\pi, K,\Sigma) $ consists of hodge classes of weight $(0, 0) $.

The Galois invariance is similar, replacing $\C$ with $\overline\Q_\l $ and  $L ^ 2 $ cohomology with absolute intersection cohomology in the diagram above.
\end{proof}
\subsection{Nonvanishing}
\subsubsection{}
To test the non-degeneracy of the subspace $\Hodge_(\pi, K,\Sigma) $, we will use the following proposition.
\begin{prop}\label{prop:. Identity for college classes}
Let $\Pi=\Pi_S, $ where $S = S_f\sqcup S_\infty$, and choose an auxiliary $\pi$ as above. Suppose given $$\alpha\in\left (\Theta(\pi\boxtimes\mathbbm 1)\otimes\boldsymbol\sigma_{\boldsymbol m}\right) ^ { {\boldsymbol K}_3}. $$
\begin {enumerate}
\item
Fix choices of signs $\boldsymbol\epsilon =\set {\epsilon_v}_{v |\infty} $ and $\boldsymbol\epsilon' =\set {\epsilon_v}_{v |\infty}$. Then: $$\langle \iota_{2,\ast}\circ\iota^\ast_1(\cl(\alpha)), \cl_{S_\infty}^{\boldsymbol\epsilon}(\beta)\boxtimes\cl'_{\boldsymbol\epsilon'}(\gamma) = \begin{cases}
\widetilde{\mathcal{ P}}_{S,\pi_1,\pi_2,\pi}(\alpha,\beta,\gamma), & \text{if } S_\infty = \emptyset \text{ and }\boldsymbol\epsilon=\boldsymbol\epsilon';\\
0, & \text{otherwise.}
\end{cases}
$$
\item
After choosing isomorphisms
$$\Pi_{S_f}\simeq\left (\Pi\otimes\boldsymbol\tau_{\boldsymbol m} ^ {\boldsymbol\epsilon}\right) ^ {\boldsymbol K_3},\;\;\pi_{2, f} ^\vee\simeq (\pi_2 ^\vee\otimes\boldsymbol\chi_{-\boldsymbol\epsilon\boldsymbol m}) ^ {\boldsymbol K_1}, $$
the composite maps
$$\Pi_{S_f}\otimes\pi ^\vee_{2, f}\xrightarrow{\CL ^ {\boldsymbol\epsilon}\otimes\CL ^ {-\boldsymbol\epsilon}}H ^ {4d}_{(2)} (S (\boldgsp_4)\times S (\boldgl_2),\mathcal W_{\boldsymbol m,\C})\xrightarrow{\langle\iota_{2,\ast}\iota_1 ^\ast\zeta,\cdot\rangle}\C $$
are independent of $\boldsymbol\epsilon $ up to a scalar.
\end {enumerate}
\end{prop}
\begin{proof}
The proof is essentially identical to 
Proposition \ref{crop: cohomology. Pairing special cycle variant}. The only new ingredient is the calculation of the $(\boldsymbol K_2\times \boldsymbol K_1)_0$-equivariant composite:
\begin{equation}\label{eq:composite 2}
\begin{split}
    \boldsymbol\sigma_{\boldsymbol m}\otimes \boldsymbol\tau_{\boldsymbol m, S_\infty}^{\boldsymbol\epsilon}\otimes \boldsymbol\chi_{-\boldsymbol\epsilon'\boldsymbol m}^\vee \to \wedge ^ {\boldsymbol 2,\boldsymbol 2}\mathfrak p ^\ast_{\GSp_6}\otimes V_{(\boldsymbol m -2,\boldsymbol m -2, 0),\C}\otimes\wedge ^ {\boldsymbol p (\boldsymbol\epsilon, S_\infty),\boldsymbol q (\boldsymbol\epsilon, S_\infty)}\mathfrak p ^\ast_{\GSp_4}\otimes V_{(\boldsymbol m -2, 0),\C}\\\otimes\wedge^{1-\boldsymbol p (\boldsymbol\epsilon'), 1 -\boldsymbol q (\boldsymbol\epsilon')}\mathfrak p ^\ast_{\GL_2}\otimes V_{\boldsymbol m -2,\C} ^\check
    \to\wedge ^ {3 + \boldsymbol p (\boldsymbol\epsilon, S_\infty) -\boldsymbol p (\boldsymbol\epsilon'), 3+\boldsymbol q (\boldsymbol\epsilon, S_\infty) -\boldsymbol q (\boldsymbol\epsilon')}\mathfrak p ^\ast_{\widetilde H}\xrightarrow{\mathbbm 1_{\widetilde H} }\C,
    \end{split}
\end{equation}
which is automatically trivial unless $S_\infty =\emptyset $ and $\boldsymbol\epsilon =\boldsymbol\epsilon'$, in which case one can check that it is not trivial.
\end{proof}
\begin {thm}\label{main theorem  Hodge classes}
    Let $\pi_1 $ and $\pi_2 $ be cuspidal automorphic representations of $\GL_2 (\A) $
of weights $\boldsymbol m +2 $ and $\boldsymbol m $, respectively, where $\boldsymbol m = (m_v)_{v |\infty} $
for positive  integers $m_v $. Assume that the central characters of $\pi_1 $ and $\pi_2$ agree and have  infinity type $\omega_{\boldsymbol m}$.
Let $\Pi_{S_f}$ be as  in (\ref{subsubsec: notation for endoscopic lifts in coh}) for a set $S_f$ of finite places of $F$ such that $| S_f |\geq 2 $ is even, and choose a coefficient field $E\supset\Q (\boldsymbol m) $
over which $\pi_i $ and $\Pi_{S_f} $ are defined.
Then there exists a triple $(\pi, K,\Sigma) $ as in Definition \ref{definition: Hodge Tate subspacce}  and a Hodge class $$\xi\in\Hodge_E(\pi, K,\Sigma)\subset H ^ {4d} _!(S (\boldgsp_4)\times S (\boldgl_2),\mathcal W_{\boldsymbol m, E}) (2d)\left [\Pi_{S_f}\boxtimes\pi_2 ^\vee\right] $$
such that the induced map
$$\xi_\ast: H_! ^ {3d}(S (\boldgsp_4),\mathcal V_{\boldsymbol m, E}) (d) [\Pi_{S_f}]\to H_! ^ d (S (\boldgl_2),\mathcal V'_{\boldsymbol m, E}) [\pi_{2, f}]$$
is of the form
$$\Pi_{S_f} ^ E\otimes H_! ^ {3d} (S (\boldgsp_4),\mathcal V_{\boldsymbol m, E})_{\Pi_{S_f}} (d)\xrightarrow{\l\otimes s} \pi_{2, f} ^ E\otimes H_! ^ d (S (\boldgl_2),\mathcal V'_{\boldsymbol m, E})_{\pi_{2, f}}, $$
where $s $ is a surjection and $\l $ is a nontrivial $E $-linear map. Moreover, the image of $\xi$ in $\l$-adic \'etale cohomology is $\Gal(\overline{ \Q}/F^c)$-equivariant for all $\l$.
\end {thm}
\begin{proof}
The same argument as for Theorem \ref{biggie thm} implies $\xi_\ast $
is always a pure tensor $\l\otimes s $, and the map $s $  is always either trivial or a surjection.
Thus it suffices to show that there exists a \emph{complex} Hodge class $\xi\in\Hodge_\C (\pi, K,\Sigma) $ with $\xi_\ast\neq 0. $
However, this is guaranteed by Lemma \ref{llama: non-vanishing input} and Proposition \ref{prop:. Identity for college classes}.
\end{proof}
\bibliographystyle{plain} 
\bibliography{mybib}
\end{document}